\documentclass[11pt]{amsart}
\usepackage{amsmath,amsthm,amsfonts,amscd,amssymb,eucal,latexsym,mathrsfs,stmaryrd,enumitem,bm,xcolor,multirow,nicematrix,caption,ifthen}

\newtheorem{theorem}{Theorem}[section]
\newtheorem{corollary}[theorem]{Corollary}
\newtheorem{lemma}[theorem]{Lemma}
\newtheorem{proposition}[theorem]{Proposition}

\theoremstyle{definition}
\newtheorem{definition}[theorem]{Definition}
\newtheorem{remark}[theorem]{Remark}
\newtheorem{notation}[theorem]{Notation}
\newtheorem{example}[theorem]{Example}

\newtheorem{standing notation}[theorem]{Standing notation for groups and spaces}

\theoremstyle{plain}
\newcounter{theoremintro}
\newtheorem{theoremi}[theoremintro]{Theorem}
\newtheorem{corollaryi}[theoremintro]{Corollary}

\newcommand{\Aut}{{\rm Aut}}

\newcommand{\Ad}{{\rm Ad}\,}
\newcommand{\diag}{{\rm diag}}
\newcommand{\id}{{\rm id}}

\newcommand{\cB}{{\mathcal B}}
\newcommand{\cH}{{\mathcal H}}

\newcommand{\cZ}{{\mathcal Z}}

\newcommand{\sC}{{\mathscr C}}
\newcommand{\sD}{{\mathscr D}}
\newcommand{\sE}{{\mathscr E}}

\newcommand{\sM}{{\mathscr M}}

\newcommand{\sO}{{\mathscr O}}
\newcommand{\sP}{{\mathscr P}}

\newcommand{\sS}{{\mathscr S}}
\newcommand{\sT}{{\mathscr T}}
\newcommand{\sU}{{\mathscr U}}
\newcommand{\sV}{{\mathscr V}}
\newcommand{\sW}{{\mathscr W}}

\newcommand{\Cb}{{\mathbb C}}
\newcommand{\Zb}{{\mathbb Z}}
\newcommand{\Tb}{{\mathbb T}}

\newcommand{\Nb}{{\mathbb N}}

\newcommand{\tr}{{\rm tr}}
\newcommand{\alg}{{\rm alg}}

\newcommand{\comp}{c}

\newcommand{\eps}{\varepsilon}
\newcommand{\unit}{1}

\newcommand{\pr}{\prime}

\numberwithin{equation}{section}

\DeclareMathOperator{\Sym}{Sym}

\DeclareMathOperator{\orb}{orb}

\DeclareMathOperator{\Act}{Act}

\DeclareMathOperator{\WA}{WA}

\DeclareMathOperator{\Astar}{A*}

\DeclareMathOperator{\base}{base}
\DeclareMathOperator{\towertop}{top}
\DeclareMathOperator{\Fix}{Fix}

\newlength{\leftstackrelawd}
\newlength{\leftstackrelbwd}
\def\leftstackrel#1#2{\settowidth{\leftstackrelawd}%
{${{}^{#1}}$}\settowidth{\leftstackrelbwd}{$#2$}%
\addtolength{\leftstackrelawd}{-\leftstackrelbwd}%
\leavevmode\ifthenelse{\lengthtest{\leftstackrelawd>0pt}}%
{\kern-.5\leftstackrelawd}{}\mathrel{\mathop{#2}\limits^{#1}}}

\allowdisplaybreaks

\begin{document}

\title{Stable rank one in nonnuclear crossed products}

\begin{abstract}
We initiate an investigation into the local structure  
of simple nonnuclear C$^*$-crossed products by showing that
stable rank one is generic within two natural classes of minimal actions of free groups
on the Cantor set. The arguments also apply to some other free product groups.
Our approach is inspired by Li and Niu's stable rank one
theorem in the amenable setting and also yields a streamlined argument in that case,
along with a generalization to product actions. 
\end{abstract}

\author[Bell]{Jamie Bell}
\address
{Jamie Bell, Mathematisches Institut, University of M\"unster, Einsteinstr.\ 62, 48149 M\"unster, Germany}
\email{jbell@uni-muenster.de}

\author[Geffen]{Shirly Geffen}
\address
{Shirly Geffen, Mathematisches Institut, University of M\"unster, Einsteinstr.\ 62, 48149 M\"unster, Germany} 
\email{sgeffen@uni-muenster.de}
 
\author[Kerr]{David Kerr}
\address
{David Kerr, Mathematisches Institut, University of M\"unster, Einsteinstr.\ 62, 48149 M\"unster, Germany}
\email{kerrd@uni-muenster.de}
 
\date{November 24, 2025}
 
\maketitle

\tableofcontents

\section{Introduction}

Amenability has long established itself as a fundamental and versatile tool
for probing the finer internal structure of noncommutative operator algebras.
The most definitive and consequential classification results in the realms of von Neumann algebras
and simple C$^*$-algebras are premised on amenability as a basic hypothesis.
This is explained by the fact that, in parallel with Rokhlin-type properties in ergodic theory
that permit one to decompose a dynamical system into almost invariant towers,
operator-algebraic amenability is associated with robust approximation
by finite-dimensional algebras, which provide a versatile combinatorial apparatus for 
detecting, manipulating, and managing structural data. One can thus view amenability 
both as a local organizing principle (form) and as an object of study in its own right by way
of the algebras that possess it as a property (content). 

What is remarkable is that the content of the content 
doesn't simply reduce to form in the sense of there being 
a unique amenable object: there are in fact rich classes of 
amenable von Neumann algebras and C$^*$-algebras, their existence owing to the
fact that local matricial structure can be pieced together asymptotically 
in a great variety of incommensurable ways.
There are, to be sure, significant technical and qualitative differences 
between von Neumann algebras and C$^*$-algebras in the way that this bounty manifests itself, and the very first 
classification result in the field---the uniqueness of the hyperfinite II$_1$ factor established by Murray
and von Neumann---did reveal there to be a unique separable amenable object 
among noncommutative probability spaces, i.e., factors possessing a normal tracial state
(the abstract notion of amenability
for von Neumann algebras actually came later, in various guises including most prominently injectivity, but
these were shown by Connes to all be equivalent to hyperfiniteness in the separable case).
It was eventually discovered, however, that there is an abundance of nontracial amenable factors
and of simple nuclear (i.e., amenable) C$^*$-algebras, both tracial and nontracial.

Amenability can also be applied to great effect even when the algebras in question are 
themselves either nonamenable or a priori not known to be amenable, as in Connes's use of the McDuff
property to prove that injectivity implies hyperfiniteness for separable factors \cite{Con76}
and in Ozawa's proof of his solidity result for the von Neumann algebras of hyperbolic groups,
which combines both von Neumann algebraic and C$^*$-techniques \cite{Oza04}.
On the more purely C$^*$-algebraic side, amenability is everywhere manifest in the scenography
around exactness, nuclear embeddability, and related properties \cite{BroOza08}, and 
the questions of simplicity and unique trace for reduced group C$^*$-algebras were discovered
in \cite{BreKalKenOza17,Ken20} to hinge precisely on the threshold between amenability and nonamenability within the group.

In the present work we take this principle in a new direction by using amenability---in the form of Rokhlin tower
decompositions and F{\o}lner tilings---to show that nonnuclear tracial crossed products 
coming from minimal actions of free groups on the Cantor set frequently have stable rank one.
Inspired by the relation of Bass stable rank to cancellation phenomena in algebraic $K$-theory, 
Rieffel introduced the notion of (topological) stable rank as a dimension-type invariant that is similarly related 
to cancellation issues in (topological) $K$-theory within the framework of unital C$^*$-algebras \cite{Rie83}.
The property of stable rank one boils down to the density of invertible elements 
and implies both that the Murray--von Neumann semigroup has cancellation and that $K_1$
can be expressed without stabilizing as the unitary group modulo the path component of the identity.
It is the operative hypothesis in the analysis of divisibility phenomena that has 
driven recent progress on the Cuntz semigroup and its applications to longstanding problems like rank realization,
which Thiel succeeded in confirming for all simple C$^*$-algebras of stable rank one  
\cite{Rob12,Thi20,AntPerRobThi22}.
Many simple separable unital finite C$^*$-algebras are known to have stable rank one, notably
those satisfying the regularity property of $\cZ$-stability that has come
to play a key role in the Elliott classification program \cite{Ror04}, and more generally
those that are pure in the sense of Winter (i.e., have strict comparison and are almost divisible) \cite{Lin25}. 
Stable rank one also holds
for crossed products of free minimal actions of FC (and in particular Abelian) groups
on compact metrizable spaces \cite{AlbLut22,LiNiu20,Nar22,Nar24b},
some of which fail to be $\cZ$-stable even when the group is $\Zb$ \cite{GioKer10},
as well as for the reduced group C$^*$-algebras of free groups and, more generally, acylindrically hyperbolic groups
\cite{DykHaaRor97,DykHar99,GerOsi20,Rau25}. 
On the other hand, Villadsen showed in \cite{Vil99} that all possible 
values of stable rank are realized by simple separable nuclear C$^*$-algebras. Moreover,
the stable rank of a simple C$^*$-algebra is infinite as soon as a nonunitary isometry is present.
Granted that one stays within the realm of finite C$^*$-algebras, however, one can 
interpret stable rank one as a regularity condition expressing a kind of zero-/one-dimensionality. 
We direct the reader to \cite{SchTikWhi25} for a state-of-the-art picture of how stable rank one 
interweaves into 
the structure and classification theory of nuclear C$^*$-algebras.

Despite all of this progress, little seems to be known 
about the value of stable rank for naturally arising nonnuclear simple separable finite C$^*$-algebras,
in particular those arising as crossed products,
once one moves beyond the reduced group C$^*$-algebras treated in \cite{DykHaaRor97,DykHar99,GerOsi20,Rau25,AmrGaoKunPat25,Vig25,Oza25}. 
The arguments in these papers rely on spectral or topological-dynamical phenomena tied to the geometry of the group,
and it is not clear whether they can be adapted to handle crossed products. Our approach, being rooted in
the framework of amenability and invariant measures,
is completely different. As we are working in the Baire category framework, we also need to
develop some of the descriptive set theory of spaces of
actions of free groups and other free products on the Cantor set, which does not seem to have been explored much
in the literature before the recent appearance of the paper \cite{DouMelTsa25}.

We now formulate our main results.
For a countable discrete group $G$ and a compact metrizable space $X$, we
write $\WA (G, X)$ for the set of all topologically free minimal actions $G \curvearrowright X$ 
that are weakly mixing and admit an invariant Borel probability measure.
We equip this with the topology of elementwise compact-open convergence, which is Polish
(see Sections~\ref{S-preliminaries} and \ref{S-spaces of actions} for more details).

\begin{theoremi}\label{T-main 1}
Let $G$ be a residually finite countable discrete group and $H$ an amenable countable discrete group containing a normal infinite cyclic subgroup.
Let $X$ be the Cantor set.
For a generic action in $\WA (G*H,X)$ the reduced crossed product $C(X)\rtimes_\lambda (G * H)$ has stable rank one.
\end{theoremi}

Given that the free group $F_d$ on $d\geq 2$ generators can be written as the free product $F_{d-1} * \Zb$,
we obtain the following as a special case.

\begin{corollaryi}\label{C-main 1}
Let $d\geq 2$ and let $X$ be the Cantor set.
For a generic action in $\WA (F_d, X)$ the reduced crossed product $C(X)\rtimes_\lambda F_d$ has stable rank one.
\end{corollaryi}

To establish Theorem~\ref{T-main 1} we introduce a dynamical notion of {\it square divisibility}
that distills some of the operator-algebraic structure
used by Li and Niu in \cite{LiNiu20} to prove a stable rank one result in the context of amenable acting groups.
Like them, we follow the basic strategy pioneered by R{\o}rdam in \cite{Ror91}
in which elements $a$ satisfying $ad=da=0$ for some nonzero positive element $d$
are unitarily rotated into nilpotent elements
(see the introduction to Section~\ref{S-stable rank one}), a procedure that square divisibility permits us to implement in our 
crossed product context.
In fact we develop two versions of this square divisibility (Definitions~\ref{D-SD} and \ref{D-weak SD})
which are tailored to two different applications although they both imply stable rank one,
as we show in Theorems~\ref{T-SR1} and \ref{T-SR1 weak SD} by applying some of the techniques from \cite{LiNiu20}.
The second version, called {\it weak square divisibility}, is only defined over the Cantor set but has the additional virtue that it is  
a $G_\delta$ property by the way the definition is locally formulated in terms of open conditions,
so that we can apply the Baire category theorem in a natural manner. After a series of lemmas that serve to establish the density
of the open sets at play in the definition of weak square divisibility,
we are thereby able to deduce, under the hypotheses of Theorem~\ref{T-main 1}, that the weakly squarely divisible actions in 
$\WA (G*H,X)$ form a dense $G_\delta$ set (Theorem~\ref{T-SD free products}). 
In conjunction with Theorem~\ref{T-SR1 weak SD}, this yields Theorem~\ref{T-main 1}.

In order to establish the aforementioned density we also need to invoke the stronger form of
square divisibility for minimal actions of countably infinite amenable groups satisfying (dynamical) comparison 
and the uniform Rokhlin property (URP),
a fact that we establish in Theorem~\ref{T-amenable} in the general setting of compact metrizable spaces.
For the purposes of Theorem~\ref{T-main 1} we only need to worry about the Cantor set and do not need
to know that square divisibility itself implies stable rank one. We show nevertheless in Theorem~\ref{T-SR1} that the latter 
implication does indeed hold under more general compact metrizable hypotheses, and note in Theorem~\ref{T-SR1 weak SD}
that essentially the same argument also gives the implication for weak square divisibility in the Cantor setting,
which is what we use to derive Theorem~\ref{T-main 1}. 
Theorems~\ref{T-SR1} and \ref{T-amenable} in combination show that, for minimal actions of countably infinite amenable groups 
on compact metrizable spaces, the URP and comparison together imply that the 
crossed product has stable rank one, so that we recover the stable rank one result 
of Li and Niu from \cite{LiNiu20} in a weaker form that replaces their hypothesis of Cuntz comparison on open sets (COS)
with the purely dynamical notion of comparison (actually neither of the latter two properties is known to fail among minimal
actions, and comparison holds in all of the cases where COS is known to hold). We thus obtain, for example, a 
streamlined proof of stable rank one for the crossed products of free minimal actions of infinite Abelian groups 
on compact metrizable spaces, with comparison in this case having been proved in \cite{Nar22} and the URP in \cite{Nar24b}.
Our methods also allow for the following generalization to product actions, 
which is a consequence of Theorems~\ref{T-SR1} and \ref{T-SD product}.

\begin{theoremi}\label{T-product} 
Let $G\curvearrowright X$ and $H\curvearrowright Y$ be minimal actions of countable discrete groups on 
compact metrizable spaces with $G$ infinite, and suppose that the first action has the URP and comparison.
Then the reduced crossed product of the product action $G\times H\curvearrowright X\times Y$
given by $(g,h)(x,y) = (gx,hy)$ has stable rank one.
\end{theoremi}

That the conjunction of the URP and comparison should really be considered a single property, 
in analogy with almost finiteness as the combination of almost finiteness in measure and comparison,
has been borne out in recent work of Naryshkin in \cite{Nar24b}, which establishes several equivalent formulations
of this conjunction, terminologically abbreviated to URPC, 
and derives some remarkable applications to shift embeddability.

Our approach to stable rank one in the amenable setting, which underpins not only Theorem~\ref{T-product}
but also Theorem~\ref{T-main 1} (as well as Theorem~\ref{T-main 2} below),
differs from that of Li and Niu by leveraging the Rokhlin towers coming from the URP to greater effect
so as to create, as in the definition of square divisibility, an array of open sets whose 
complement is small but whose boundaries are in turn much smaller than this complement.
It is this relative smallness at two different scales that allows us to work with a simpler version of the apparatus 
at play in \cite{LiNiu20}.

The proof of Theorem~\ref{T-product} goes in the direction of trying to show that if $G\curvearrowright X$ 
is squarely divisible then so is every product action of the form $G\times H\curvearrowright X\times Y$
(a stronger statement that we have been unable to verify), but we need some extra local information 
beyond square divisibility that we get from the URP and comparison. It is nevertheless interesting to 
compare Theorem~\ref{T-product} with recent results on the permanence of dynamical and C$^*$-algebraic
regularity properties under taking products. Among regularity properties for C$^*$-algebras, 
$\cZ$-stability hits the sweet spot of being 
both highly consequential (classification being the highest payoff) and accessible to verification 
in a great many cases. An illustration of its robustness is the trivial fact that a minimal tensor
product $A\otimes_{\min} B$ is $\cZ$-stable as soon as one of the factors is.
In particular, the reduced crossed product of a product action $G\times H \curvearrowright X\times Y$
is $\cZ$-stable as soon as this is the case for one of the factors. Once one replaces $\cZ$-stability
by a weaker regularity property like stable rank one, however, the issue of permanence under taking product actions
can become quite tricky. This is even
already true for almost finiteness, the closest dynamical analogue to $\cZ$-stability that we have in the
setting of amenable acting groups. See \cite{KopLiaTikVac23}, where the 
permanence under products of both almost finiteness and almost finiteness in measure 
was established using C$^*$-algebra technology, 
and also \cite{KerLi24}, where it was shown for the related small boundary property
using purely dynamical methods.

Theorem~\ref{T-product} is not as unrelated to Theorem~\ref{T-main 1} as it might appear at first glance.
The proof of Theorem~\ref{T-main 1} also leverages the URP and comparison, via Theorem~\ref{T-amenable},
in the context of a product construction. In that case however we are approximating a given action of $G*H$
via its diagonal product with another action, and so we do not have the freedom of starting
from separate actions of two different groups as in Theorem~\ref{T-product}. This renders
the analysis much more complicated, although the verification of square divisibility is similar
in the two settings. In particular, for Theorem~\ref{T-main 1} we need to build a machine for producing
diagonal actions that are weakly mixing and minimal, a task that involves
some ergodic theory and is carried out in Section~\ref{S-diagonal machine}.

Unfortunately we have been unable to show that the generic action in Theorem~\ref{T-main 1} or Corollary~\ref{C-main 1} 
does not belong to a single conjugacy class, although we strongly suspect that every conjugacy class is meagre,
and indeed this is known to be the case among weakly mixing minimal actions of $\Zb$ on the Cantor set \cite[Theorem 1.2 and Lemma 8.1]{Hoc08}. On the other hand,
a generic minimal action of $F_d$ on the Cantor set admitting an invariant Borel probability measure is conjugate to the universal odometer action
\cite[Theorem 1.5]{DouMelTsa25} (the case $d=1$ was treated in \cite{Hoc08}). 
One can already immediately see from the definition of the topology
on spaces of actions on the Cantor set that the property of having a given action on a finite set as a factor is stable under perturbations whenever the acting group is finitely generated. 
This explains why in Theorem~\ref{T-main 1} we impose the property of weak mixing, which rules out nontrivial finite factors.
On the other hand, Proposition~\ref{P-trivial factor} shows that, when $X$ is the Cantor set, a generic action in $\WA (F_d ,X)$ 
has the property that the homeomorphisms defined by the standard generators all factor onto the trivial action
on the Cantor set and hence are very far from themselves being minimal.
That $\WA (G*H, X)$ in Theorem~\ref{T-main 1} is nonempty can either be seen as a special
case of a general phenomenon recorded in Proposition~\ref{P-Polish} 
or by means of the concrete examples constructed in Section~\ref{S-examples}.

In the case of $F_d$, if we further restrict our scope to actions that are minimal and spectrally aperiodic on generators
then we have in fact been able to establish both the genericity of stable rank one and the meagreness of orbits,
so that we are indeed capturing a relatively large class of actions, all rather different than the generic ones in $\WA (F_d ,X)$,
which as mentioned above are far from being minimal on generators. That this class 
is nonempty is illustrated by the examples in Section~\ref{S-examples}.
We write $\Astar (F_d , X)$ for the set of all topologically free actions $F_d \curvearrowright X$ on the Cantor set
that have an invariant Borel probability measure and are strictly ergodic (i.e., minimal and uniquely ergodic) and 
spectrally aperiodic (Definition~\ref{D-spectrally aperiodic})
on each standard generator. In Theorem~\ref{T-SD free groups} we show that
weak square divisibility is generic in $\Astar (F_d , X)$. Together with
Theorem~\ref{T-SR1 weak SD}, this yields generic stable rank one:

\begin{theoremi}\label{T-main 2}
Let $d\geq 2$ and let $X$ be the Cantor set.
For a generic action in $\Astar (F_d , X)$ the reduced crossed product $C(X)\rtimes_\lambda F_d$ has stable rank one.
\end{theoremi}

The meagreness of orbits we record as follows and establish in Section~\ref{S-meagre}.

\begin{theoremi}\label{T-meagre}
Let $d\geq 2$ and let $X$ be the Cantor set.
Then every orbit in $\Astar (F_d , X)$ under the conjugation action of the homeomorphism group of $X$ is meagre.
\end{theoremi}

We begin in Section~\ref{S-preliminaries} by laying out some general definitions and basic dynamical background.
Section~\ref{S-SD} introduces the dynamical properties of square divisibility and weak square divisibility.
These then appear as the operative hypotheses for Theorem~\ref{T-SR1} and \ref{T-SR1 weak SD} 
on stable rank one in Section~\ref{S-stable rank one}.
Section~\ref{S-amenable} establishes square divisibility for minimal actions of countably infinite discrete amenable groups 
on compact metrizable spaces under the assumptions of the URP and comparison (Theorem~\ref{T-amenable}).
With this at hand, we state and prove Theorem~\ref{T-SD product} on square divisibility in 
product actions in Section~\ref{S-product}. Section~\ref{S-diagonal machine}
is devoted to the diagonal action machine that will be used in later sections for various purposes.
After setting up the relevant spaces of actions in Section~\ref{S-spaces of actions},
we then turn to the proofs of Theorem~\ref{T-SD free products} and \ref{T-SD free groups} (generic weak square divisibility) 
in Sections~\ref{S-SD I} and \ref{S-SD II} and 
of Theorem~\ref{T-meagre} (meagreness of orbits) in Section~\ref{S-meagre}.
To conclude we present some examples of squarely divisible actions of free groups in
Section~\ref{S-examples}.
\medskip

\noindent{\it Acknowledgements.}
The authors were supported by the Deutsche Forschungsgemeinschaft
(DFG, German Research Foundation) under Germany's Excellence Strategy EXC 2044-\linebreak 390685587,
Mathematics M{\"u}nster: Dynamics--Geometry--Structure, and 
by the SFB 1442 of the DFG. Portions of this work were carried out during visits of the third
author to Queen's University Belfast in July 2025
as part of the INI program ``C$^*$-Algebras: Classification and Dynamical Constructions''
and to Chongqing University in September 2025.

\section{Notational conventions and preliminaries}\label{S-preliminaries}

\noindent {\bf Standing notation for groups and spaces.}
Throughout the paper $G$ and $H$ are countable discrete groups, to be explicitly subject to extra hypotheses 
depending on the circumstances. The identity element of a group will always be denoted $e$.
By $X$ and $Y$ we always mean compact metrizable spaces, often to be explicitly specialized to the Cantor set.
\medskip

We write $M(X)$ for the convex set of all Borel probability measures on $X$ equipped with the weak$^*$ topology,
under which it is compact.
We use the notation $G\curvearrowright X$ to denote an action of $G$ on $X$ by homeomorphisms. Often we write the action via the simple concatenation $(s,x)\mapsto sx$ but 
when two or more actions are at play we will typically use symbols such as $\alpha$ to avoid confusion,
so that the notation for the action becomes $(s,x)\mapsto \alpha_s x$. We similarly use $sA$ or $\alpha_s A$
for the image of a set $A\subseteq X$ under $s$, and write $LA$ for a set $L\subseteq G$ to mean $\bigcup_{s\in L} sA$.
For sets $V\subseteq X$ and $E\subseteq G$ we write $V^E$ for the intersection $\bigcap_{s\in E} s^{-1}V$. 

When $G = \Zb$ the action is generated by a single transformation $T: X\to X$ associated to the generator $1\in\Zb$,
and so we will usually tacitly identify $\Zb$-actions and single transformations (i.e., homeomorphisms of $X$)
despite the abuse of notation this inevitably leads to. We will also use the notation $T\curvearrowright X$
for a transformation $T:X\to X$.

The action $G\curvearrowright X$ is {\it minimal} if there are no closed $G$-invariant subsets of $X$ other than $\emptyset$ and $X$ itself,
or, equivalently, every $G$-orbit is dense. The {\it stabilizer} of a point $x\in X$ is the fixed point set
$\{ s\in G : sx = x \}$, which forms a subgroup of $G$. The action is {\it free} if the stabilizer of every point is trivial,
and {\it topologically free} if the set of points with stabilizer equal to $\{ e \}$ is dense (in which case it is actually a dense $G_\delta$ set). The action is \emph{faithful} if for every $s\in G\setminus \{e\}$ there exists an $x\in X$ such that $sx\ne x$, i.e., the associated homomorphism of $G$ into the homeomorphism group of $X$ is injective.  

The action $G\curvearrowright X$ is {\it transitive} if for all nonempty open sets $U,V\subseteq X$ there exists $s\in G$ such that
$sU \cap V \neq\emptyset$, which is equivalent to the existence of a dense orbit (using that $X$ is compact and metrizable). The action is \textit{topologically mixing} if $G$ is infinite and for all nonempty open sets $U,V\subseteq X$ there exists a finite subset $F\subseteq G$ such that $sU\cap V\neq \emptyset$ for all $s\in G\setminus F$, and
{\it topologically weakly mixing} if for all nonempty open sets $U_1 , U_2 , V_1 , V_2 \subseteq X$
there exists $s\in G$ such that $sU_1 \cap U_2 \neq\emptyset$ and $sV_1 \cap V_2 \neq\emptyset$,
which is equivalent to the transitivity of the diagonal action $G\curvearrowright X\times X$ as given by
$s(x,y) = (sx,sy)$. We also simply say {\it mixing} or {\it weakly mixing} if it is clear that 
the property we are referring to is not its measure-theoretic version. 

For an action $G\curvearrowright X$, the set of all $G$-invariant measures in $M(X)$ is denoted $M_G (X)$, or $M_\alpha (X)$ if our action has a name $\alpha$. This is a compact convex set whose extreme points are precisely 
the ergodic measures. The action is \emph{uniquely ergodic} if $M_G(X)$ is a singleton, and \emph{strictly ergodic} if it is minimal and uniquely ergodic. 

In Sections~\ref{S-diagonal machine}, \ref{S-SD I}, and \ref{S-SD II} we will have occasion to use some ergodic theory for 
p.m.p.\ (probability-measure-preserving) actions $G\curvearrowright (Z,\zeta )$. 
For this we will outsource most of the basic background and terminology to \cite{KerLi16}.
As in the topological setting, we identify a p.m.p.\ $\Zb$-action on $(Z,\zeta )$ with the generating
transformation $T$ associated to $1\in\Zb$, which we also write as $T \curvearrowright (Z,\zeta )$.
When we say that a p.m.p.\ action $G\curvearrowright (Z,\zeta )$ is {\it free} we mean that the set of 
points $x\in Z$ whose stabilizer $\{ s\in G : sx = x \}$ is trivial has measure one.
An action $G\curvearrowright X$ is 
\textit{essentially free} if $M_G(X)\neq\emptyset$ and for every $\mu\in M_G(X)$ the p.m.p.\ action $G\curvearrowright (X,\mu)$ is free.
It is easy to see that essentially free minimal actions are automatically topologically free. A p.m.p.\ action $G\curvearrowright (Z,\zeta)$ is \textit{mixing} if $G$ is infinite and for all measurable sets $A,B\subseteq Z$ and $\eps>0$ there exists a finite set $F\subseteq G$ such that $\left|\zeta(sA\cap B)-\zeta(A)\zeta(B)\right|<\eps$ for all $s\in G\setminus F$, and it 
is {\it weakly mixing} if for all finite collections $\Omega$
of measurable subsets of $Z$ and $\eps>0$ 
there exists an $s\in G$ such that $\left|\zeta(sA\cap B)-\zeta(A)\zeta(B)\right|<\eps$ for all $A,B\in\Omega$.
It is readily seen that if $G\curvearrowright X$ is an action and $\mu\in M_G(X)$ is a measure
of full support such that $G\curvearrowright (X,\mu)$ is mixing (resp.\ weakly mixing) in the measure-theoretic sense 
then $G\curvearrowright X$ is topologically mixing (resp.\ topologically weakly mixing).

For a closed set $A\subseteq X$ and an open set $B\subseteq X$ we write $A\prec B$
($A$ is {\it (dynamically) subequivalent} to $B$) if there exist open sets $U_1 , \dots , U_n \subseteq X$ and $s_1 , \dots , s_n \in G$
such that $A\subseteq \bigcup_{i=1}^n U_i$ and the sets $s_1 U_1 , \dots , s_n U_n$ are pairwise disjoint subsets
of $B$. For a set $F\subseteq G$ we write $A\prec_F B$ if $A\prec B$ and we can choose the group elements 
$s_1 , \dots , s_n$ in the definition of subequivalence to belong to $F$.
We also write $A\prec_\alpha B$ and $A\prec_{\alpha ,F} B$ if the action $\alpha$ needs to be made explicit.
If $X$ is zero-dimensional and $A$ and $B$ are clopen subsets of $X$
such that $A\prec B$ then one can take the sets $U_1 , \dots , U_n$ in the definition of subequivalence
to form a clopen partition of $A$ \cite[Proposition~3.5]{Ker20}.

The action $G\curvearrowright X$ has {\it (dynamical) comparison} if for every closed set $A\subseteq X$ and 
open set $B\subseteq X$ satisfying $\mu (A) < \mu (B)$ for every $\mu\in M_G (X)$ one has $A\prec B$ \cite[Definition~3.2]{Ker20}.
One can also express this using pairs of open sets $A$ and $B$ by interpreting subequivalence in this case 
to mean $A_0 \prec B$ for every closed set $A_0 \subseteq A$. 

A {\it tower} for the action $G\curvearrowright X$ is a pair $(S,B)$ where $S$ is a nonempty finite subset of $G$ (the {\it shape})
and $B$ is a nonempty subset of $X$ (the {\it base}) such that the sets $sB$ for $s\in S$ (the {\it levels}) are pairwise disjoint. 
The tower is open, clopen, etc., if the levels are open, clopen, etc. 
A \textit{castle} is a finite collection $\{(S_k,V_k)\}_{k\in K}$ of towers such that the sets $S_k V_k$
for $k\in K$ are pairwise disjoint. The {\it remainder} of the castle is the complement 
$X\setminus \bigsqcup_{k\in K} S_k V_k$.
The castle is open, clopen, etc., if the towers are open, clopen, etc.

Let $E$ be a finite subset of $G$ and $\delta > 0$. For a set $F\subseteq G$ we write $F^E$ for $\bigcap_{s\in E} s^{-1} F$.
A finite set $F\subseteq G$ is said to be {\it $(E,\delta )$-invariant} if $|F^{E\cup \{ e \}} | \geq (1-\delta )|F|$.
The existence, for each finite set $E\subseteq G$ and $\delta > 0$, of an $(E,\delta )$-invariant finite set $F\subseteq G$
is the F{\o}lner characterization of amenability for $G$. Another characterization of amenability for $G$ 
is the nonemptiness of $M_G (X)$ for every action $G\curvearrowright X$ on a compact metrizable space.
The class of amenable groups includes finite groups and
Abelian groups and is closed under taking subgroups, quotients, extensions, and direct limits (the bootstrap
class one thereby obtains comprises the {\it elementary amenable groups}).
The canonical examples of nonamenable groups are the free groups 
$F_d = \langle a_1 , \dots , a_d \rangle$ on $d\geq 2$ generators. 

Associated to amenable groups and their actions are strong tiling properties, a couple of which we will exploit in 
Section~\ref{S-amenable} on the way to establishing Theorems \ref{T-main 1} and \ref{T-main 2}.
One of these is the Ornstein--Weiss quasitiling theorem for F{\o}lner subsets of $G$ (see Lemma~\ref{L-tiling}).
The other is the {\it uniform Rokhlin property (URP)} for an action $G\curvearrowright X$, which 
requires the existence, for every finite set $E\subseteq G$ and $\delta > 0$, of an open castle whose towers have
$(E,\delta )$-invariant shapes and whose remainder $R$ 
satisfies $\sup_{\mu\in M_G (X)} \mu (R) < \delta$ \cite[Definition~3.1]{Niu22}. This is similar to the stronger property of {\it almost finiteness}, which is more directly related to $\cZ$-stability \cite{Ker20}
and requires (i) that the remainder $R$ instead be subequivalent to a set made up of a 
small (i.e., less than the given $\delta > 0$) proportion of levels in each tower, and (ii) that
the tower levels additionally have small diameter with respect to a given compatible metric \cite[Definition~8.2]{Ker20}. It is known that free actions of countably infinite discrete groups on compact metrizable spaces with finite covering dimension satisfy the URP, that free minimal actions of countably infinite Abelian groups
on compact metrizable spaces satisfy the URP \cite[Corollary~E]{Nar24b}, that actions of elementary amenable groups and of finitely generated
groups of subexponential growth on the Cantor set are almost finite \cite{DowZha23,KerNar21} 
(see also \cite{Nar24} for a more general result), 
and that a generic action of a fixed countably infinite amenable group on the Cantor set is almost finite \cite[Theorem~4.2]{ConJacKerMarSewTuc18}.
For free (or even essentially free \cite{GarGefGesKopNar24}) actions of countably infinite amenable $G$,
almost finiteness implies comparison \cite{Ker20}
and is equivalent to it when the action has the small boundary property, which is automatic
in the case that $X$ is finite-dimensional \cite{KerSza20}.

For more on amenability and nonamenability, especially in connection with dynamical phenomena, see \cite{KerLi16}.

From an action $G\curvearrowright X$ one forms, in combination with the left regular representation $\lambda: G\to \cB(\ell^2(G))$,
the reduced crossed product C$^*$-algebra $C(X)\rtimes_\lambda G$. This is a certain completion
of the algebraic crossed product consisting of the sums $\sum_{s\in E} f_s u_s$ where $E$ is a finite subset of $G$,
the $u_s$ are fixed unitaries indexed by $G$ via a group homomorphism $s\mapsto u_s$, and
the ``coefficients'' $f_s$ are functions in $C(X)$, with the multiplication
determined by the relation $u_s f u_s^{-1} = sf$ for all $s\in G$ and $f\in C(X)$ where $(sf)(x) = f(s^{-1} x)$ for $x\in X$.
One can also form other completions, including a maximal one (the full crossed product), and if the group $G$ 
is amenable, or more generally if the action is amenable, then the full and reduced crossed products will canonically coincide. 
A key technical feature of the reduced crossed product is the existence of a {\it faithful} conditional
expectation $E : C(X)\rtimes_\lambda G \to C(X)$ satisfying $E(f_s u_s ) = 0$ whenever $s\neq e$.
Another important fact is that the reduced crossed product $C(X)\rtimes_\lambda G$
is simple as a C$^*$-algebra whenever the action $G\curvearrowright X$ is minimal and topologically free \cite{ArcSpi94}. 

A unital C$^*$-algebra $A$ has \emph{stable rank one} if the set of elements in $A$ that generate it as a left ideal
is dense, or, equivalently, if the set of invertible elements in $A$ is dense. 
Stable rank one implies stable finiteness, which in turn implies the existence of a quasitrace \cite{Han81,BlaHan82}.
A quasitrace on the reduced crossed product of an action $G\curvearrowright X$ restricts on $C(X)$ to 
integration with respect to a $G$-invariant Borel probability measure, so that
$M_G(X) \ne \emptyset$ is a necessary condition for the crossed product to have stable rank one.
For definitions and more on C$^*$-algebras we refer the reader to \cite{Bla05,BroOza08}. 

For fixed $G$ and $X$ we denote
by $\Act (G,X)$ the space of all actions $G\curvearrowright X$. We endow $\Act (G,X)$ with the topology
of pointwise compact-open convergence on individual group elements. If we fix a compatible metric $d$ on $X$
then a basis for this topology is given by the sets
\[
U_{\alpha ,E,\delta} = \{ \beta\in\Act (G,X) : \sup\textstyle_{x\in X} d(\beta_s x,\alpha_s x) < \delta \text{ for all } s\in E \}
\]
where $\alpha\in\Act (G,X)$, $E$ is a finite subset of $G$, and $\delta > 0$. Moreover we can endow $\Act (G,X)$
itself with a compatible metric by fixing an enumeration $s_1 , s_2 , \ldots$ of $G$ (assuming $G$ is infinite
and adjusting notation otherwise) and setting
\[
\rho (\alpha , \beta ) = \sum_{k=1}^\infty \frac{1}{2^k} \sup_{x\in X} d(\alpha_{s_k} x , \beta_{s_k} x) .
\]
The space $\Act (G,X)$ is complete under this metric and separable as a topological space, and so it is Polish.

In the case that $X$ is the Cantor set, the topology on $\Act (G,X)$ also has as a basis the open sets
\[
U_{\alpha ,E,\sP} = \{ \beta\in\Act (G,X) : \beta_s A = \alpha_s A \text{ for all } A\in\sP \text{ and } s\in E \}
\]
where $\alpha\in\Act (G,X)$, $E$ is a finite subset of $G$, and $\sP$ is a clopen partition of $X$.

In Section~\ref{S-spaces of actions} we will introduce certain subspaces of $\Act (G,X)$
in preparation for the genericity and meagreness results of Sections~\ref{S-SD I}, \ref{S-SD II}, and \ref{S-meagre}.
In the case of the genericity theorems we will need the following fact in the case that $\alpha$
is the product of $\beta$ with some other action and $h$ is the projection map onto the first coordinate. The idea is essentially the same as in the proof of \cite[Theorem~4.2]{ConJacKerMarSewTuc18}.

\begin{proposition}\label{P-extension approximation}
Suppose that $X$ and $Y$ are the Cantor set.
Let $\alpha\in\Act (G,X)$ and $\beta\in\Act (G,Y)$ and suppose that there is a continuous surjection $h: Y\to X$ such that
$\alpha\circ h = h\circ\beta$. Let $\{ \sP_i \}_{i\in I}$ be the net of all clopen partitions of $X$ ordered
by refinement. Then for every $i\in I$ there is a homeomorphism $g_i : Y\to X$ satisfying 
$g_i (h^{-1} (A)) = A$ for every $A\in\sP_i$, and if $\{ g_i \}_{i\in I}$ is any collection of such homeomorphisms
then the actions $\beta_i$ defined by $\beta_{i,s} = g_i \circ \beta_s \circ g_i^{-1}$ for $s\in G$ converge to $\alpha$.
\end{proposition}

\begin{proof}
Given an $i\in I$, for each $A\in\sP_i$ both $A$ and $h^{-1} (A)$ are nonempty clopen subsets
(the latter thanks to surjectivity of $h$) and hence are homeomorphic to the Cantor set.
Since $\sP_i$ and $\{ h^{-1} (A) : A\in\sP_i \}$ are clopen partitions of $X$ and $Y$, respectively, this permits us to construct a 
homeomorphism $g_i : Y\to X$ with the property that $g_i (h^{-1} (A)) = A$ for every $A\in\sP_i$.

Suppose, for every $i\in I$, that $g_i$ is a homeomorphism with this property.
Let $E$ be a finite subset of $G$ and $\sP$ a clopen partition of $X$. Given our description of the topology on $\Act(G,X)$ in terms of clopen partitions, in order to show that the actions $\beta_i$ defined by $\beta_{i,s} = g_i \circ \beta_s \circ g_i^{-1}$ for $s\in G$ converge to $\alpha$ it suffices to demonstrate that there exists an $i_0\in I$ such that $\beta_i\in U_{\alpha,E,\sP}$ for all $i\geq i_0$.

Let $i_0\in I$ be such that $\sP_{i_0}$ refines each of the clopen partitions $\{ \alpha_{s^{-1}} A : A\in\sP \}$ for $s\in E\cup \{ e \}$, i.e., $\sP_{i_0}=\bigvee_{s\in E\cup \{e\}}s^{-1}\sP$. 
Let $i\geq i_0$, $s\in E$, and $A\in\sP$. Then there exist $B_1 , \dots , B_n , C_1 , \dots , C_m\in\sP_i$ 
such that $A = \bigsqcup_{j=1}^n B_j$ and $\alpha_s A = \bigsqcup_{k=1}^m C_k$, so that
\begin{align*}
g_i (h^{-1} (A)) 
= g_i \bigg(\bigsqcup_{j=1}^n h^{-1} (B_j )\bigg)
= \bigsqcup_{j=1}^n g_i (h^{-1} (B_j ))
= \bigsqcup_{j=1}^n B_j 
= A
\end{align*}
and similarly $g_i (h^{-1} (\alpha_s A)) = \bigsqcup_{k=1}^n C_k = \alpha_s A$. It follows that
\begin{align*}
\beta_{i,s}(A)
=(g_i \circ\beta_s \circ g_i^{-1} )(A) 
= g_i (\beta_s  (h^{-1} (A)))
= g_i (h^{-1} (\alpha_s A))
= \alpha_s A.
\end{align*}
That is, $\beta_i\in U_{\alpha,E,\sP}$, as required.
\end{proof}

Finally we observe the following stability of subequivalence under perturbations of the action. This will be needed in the proof of Lemma~\ref{L-open}.

\begin{proposition}\label{P-subequivalence}
Let
$A$ and $B$ be subsets of $X$ with $A$ closed and $B$ open, and $F$ a finite subset of $G$.
Then the set of all $\alpha\in\Act (G,X)$ such that $A\prec_{\alpha ,F} B$ is open.
\end{proposition}

\begin{proof}
Let $\alpha$ be an action in $\Act (G,X)$ such that $A\prec_{\alpha ,F} B$.
Then there exist open sets $U_s \subseteq X$ for $s\in F$
such that $A\subseteq \bigcup_{s\in F} U_s$ and the sets $\alpha_s U_s$ for $s\in F$ 
are pairwise disjoint and contained in $B$. Fix a compatible metric $d$ on $X$.
For a set $D\subseteq X$ and $\eps > 0$ write 
$D^{\eps} = \{ x\in D: d(x,D^\comp ) > \eps \}$.
The open sets $U_s^{\eps}$ for $s\in F$ and $\eps > 0$ cover $A$, and so by compactness there is a 
particular $\eps > 0$ such that the sets $U_s ^{\eps}$ for $s\in F$ cover $A$.
Using compactness and the fact that the maps $\alpha_s$ for $s\in F$ are homeomorphisms, we can find a $\delta>0$ such that $\alpha_s\overline{U_s^{\eps}}\subseteq (\alpha_sU_s)^{\delta}$ for all $s\in F$. It is now easy to check that whenever $\beta\in U_{\alpha,F,\delta}$ one has $\beta_s \overline{U_s^{\eps}} \subseteq \alpha_s U_s$ for every $s\in F$, so that $A\prec_{\beta ,F} B$. 
\end{proof}

\section{Square divisibility}\label{S-SD}

The notions of square divisibility (Definition~\ref{D-SD}) and weak square divisibility (Definition~\ref{D-weak SD})
that we introduce here are dynamical abstractions 
of part of the apparatus used to establish stable rank one in \cite{LiNiu20}. A key difference with \cite{LiNiu20}
is that our definitions localize the subequivalences at play and thus do not make any blanket assumption
of comparison. In the case of weak square divisibility this localization is set up so that we obtain a $G_\delta$
property in the space of actions (Proposition~\ref{P-SD G delta}). This enables us to establish genericity results for stable rank one 
(Theorems~\ref{T-SD free products} and \ref{T-SD free groups}) without having to directly confront the descriptive-set-theoretic nature of 
either comparison or stable rank one, which we do not know to be $G_\delta$ conditions.
We do not even know if any of the actions in our genericity results satisfy comparison.

At the same time, our genericity results will rely on the fact that essentially free minimal actions of amenable groups on the Cantor set with comparison are $(O_1,O_2,E)$-squarely divisible in the sense of Definition~\ref{D-SD} for all nonempty open sets $O_1,O_2$ and finite sets $e\in E\subseteq G$, as shown in Theorem~\ref{T-amenable} (such actions have the URP, independently of the comparison hypothesis \cite{GarGefGesKopNar24}). 

Definition~\ref{D-SD}, as well as Theorem~\ref{T-amenable} and the resulting stable rank one result (Theorem~\ref{T-SR1}), deal with the general setting of compact metrizable $X$. This will yield, for example, a streamlined
approach to establishing stable rank one for crossed products of free minimal $\mathbb{Z}^d$-actions, 
as originally shown by Li and Niu \cite{LiNiu20}. It is also conceivable that these results apply to some actions
of nonamenable groups on higher-dimensional spaces, although the present paper will focus exclusively on the Cantor
set once we cross the threshold into nonamenability in the later sections.

\begin{definition}\label{D-PE}
Let $G\curvearrowright X$ be an action. We say that the members of a collection $\{V_m\}_{m=1}^{M}$ of subsets of $X$
are \textit{pairwise equivalent} if there exist finitely many 
Borel subsets $\{C_k\}_{k\in K}$ of $X$
satisfying $V_1=\bigsqcup_{k\in K} C_k$ and, for every $k\in K$, elements $\{s_{k,m}\}_{m=1}^{M}$ of $G$ 
with $s_{k,1}=e$ so that $V_m=\bigsqcup_{k\in K}s_{k,m}C_k$ 
and the sets $s_{k,m}C_k$ for $k\in K$ and $m= 1,\ldots, M$ have pairwise disjoint closures in $X$.
\end{definition}

Note that pairwise equivalence implies that the sets $\{V_m\}_{m=1}^{M}$ have pairwise disjoint closures.

\begin{definition}
Let $G\curvearrowright X$ be an action. For a set $E\subseteq G$, we say that the members of a collection $\sV$ of subsets of $X$ 
are {\it pairwise $E$-disjoint} if the sets $E\overline{V}$ for $V\in\sV$ are pairwise disjoint.
\end{definition}

\begin{definition}\label{D-SD}
Let $G\curvearrowright X$ be an action.
Let $O_1$ and $O_2$ be open subsets of $X$ and let $E$ be a finite subset of $G$ containing $e$.
The action $G\curvearrowright X$ is {\it $(O_1 , O_2 ,E)$-squarely divisible} if there exist an $n\in\Nb$, a collection $\{ V_{i,j} \}_{i,j=1}^n$ 
of pairwise equivalent and pairwise $E$-disjoint open subsets of $X$, and, writing $V = \bigsqcup_{i,j=1}^n V_{i,j}$,
an open set $U\subseteq X$ with $\partial V \subseteq U$ such that,
defining $V_1 = \bigsqcup_{i=1}^n V_{i,1}$, $R = V^\comp$,
and $B = \overline{V}\cap ((V\cap \overline{U}^\comp )^E )^\comp$, one has the following:
\begin{enumerate}
\item $\overline{V_{i,1}} \prec O_1 \cap \bigsqcup_{j=2}^n V_{i,j} \cap B^\comp$ for every $i=1,\dots ,n$,

\item $R\prec O_2 \cap V \cap (V_1 \cup B)^\comp$, and

\item $B\cup (\overline{U}\cap R)\prec O_2 \cap \overline{V\cup U}^\comp$.
\end{enumerate}
Given a nonempty open set $O\subseteq X$, the action
is {\it $O$-squarely divisible} if there exist three nonempty open sets $O_0 \subseteq O$ and $O_1 , O_2 \subseteq X$
with pairwise disjoint closures such that 
\begin{itemize}
\item[(iv)] $\overline{O_1}\sqcup \overline{O_2}\prec O_0$, and
\item[(v)] the action is $(O_1 , O_2 ,E)$-squarely divisible
for all finite sets $E\subseteq G$ containing $e$.
\end{itemize}
Finally, the action is {\it squarely divisible} if it 
is $O$-squarely divisible for all nonempty open sets $O\subseteq X$.
\end{definition}

Observe that $B$ contains $\overline{U}\cap V$, and so $B\cup (\overline{U}\cap R)$ contains $\overline{U}$. 
Thus $B\cup (\overline{U}\cap R)$ is a thickening of the topological boundary between the tower and the remainder.
This thickening involves both topological and group-theoretic aspects, namely the use of both $U$ and $E$. 
Condition (i) says that the tower levels are small,
condition (ii) says that the remainder is small, and condition (iii) says that the thickened boundary $B\cup (\overline{U}\cap R)$ is small.
In each case this smallness is expressed with respect to certain parts of the tower structure in conjunction with an independently prescribed parts of the space $X$, namely $O_1$ and $O_2$. Figure~\ref{F-SD} illustrates the definition in the case
of a Cantor set, where one can take all of the sets at play to be clopen and the set $U$ to be empty, as Proposition~\ref{P-SD0Dim}
will demonstrate.

\begin{figure}
\includegraphics[width = 0.9\textwidth]{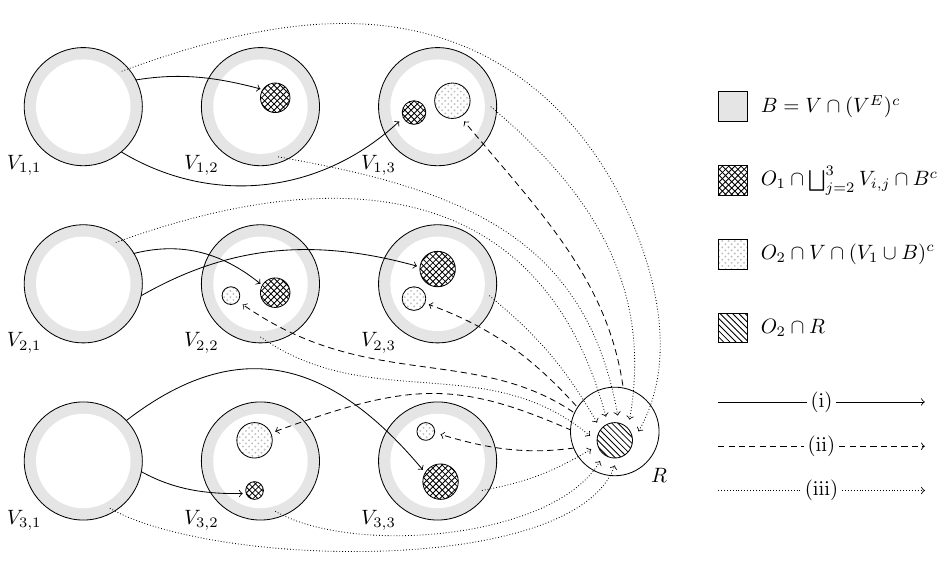}
\caption{$(O_1,O_2,E)$-square divisibility ($n=3$) in the Cantor set setting 
with $U = \emptyset$ (as in Proposition~\ref{P-SD0Dim}) and arrows indicating subequivalences.}
\label{F-SD}
\end{figure}

The fact that we take the sets $V_{i,j}$ to be indexed by a square array is not essential
for establishing Theorems~\ref{T-SR1} and \ref{T-SR1 weak SD}, 
where a more general rectangular array would be sufficient, 
as long as it is not too thin in one direction. In our applications of Theorems~\ref{T-SR1} and \ref{T-SR1 weak SD}, 
however, we can always get away with a square array, and so we have built it into the definition (and terminology) 
in order to take advantage of the notational economy it provides.

The separation property $\overline{O_0}\cap (\overline{O_1}\sqcup \overline{O_2})=\emptyset$ 
that is satisfied by the sets $O_0$, $O_1$, and $O_2$ in the definition of $O$-square divisibility 
is designed for the purpose of being able to construct 
a unitary, via Lemma~\ref{L-unitary}, in the verification of stable rank one in Theorem~\ref{T-SR1}.
That the set $\overline{O_1} \sqcup \overline{O_2}$ be subequivalent to (and not merely contained in)
the given open subset $O$ provides us with a certain flexibility that is crucial
for the proof of Theorem~\ref{T-SD product}. This aspect of the definition is not needed, however, 
in the amenable setting of Theorem~\ref{T-amenable}, where 
we show that minimal actions with the URP and comparison
possess the stronger property of $(O_1 , O_2 ,E)$-squarely divisibility for all 
nonempty open sets $O_1 , O_2 \subseteq X$ 
and all finite sets $e\in E\subseteq G$, which also happens to be critical to the proof of Theorem~\ref{T-SD product}. 
That this is indeed a stronger property we record in the following proposition.

\begin{proposition}\label{P-stronger}
Suppose that $G$ is infinite.
Let $G\curvearrowright X$ be a topologically free minimal action which is $(O_1 , O_2 ,E)$-squarely divisible for all 
nonempty open sets $O_1 , O_2 , \subseteq X$ with $\overline{O_1} \cap \overline{O_2} =\emptyset$
and all finite sets $e\in E\subseteq G$. Then the action is squarely divisible.
\end{proposition}

\begin{proof}
Let $O$ be a nonempty open subset of $X$. Since $G$ is infinite and the action is topologically free and minimal, $X$ has no isolated points.
Choose a point $x\in X$. By minimality, there exists $x_0\in O$ distinct from $x$ and $s\in G$ such that $sx=x_0$. 
Using the continuity of the action
we can then find open neighbourhoods 
$O_x\subseteq X$ and $O_0\subseteq O$ of $x$ and $x_0$, respectively, such that $\overline{O_x} \cap \overline{O_0} = \emptyset$
and $sO_x\subseteq O_0$, the latter of which gives us $O_x \prec O_0$.
Since $X$ has no isolated points, there are open sets $O_1,O_2\subseteq X$ with disjoint closures such that $\overline{O_1}\sqcup\overline{O_2}\subseteq O_x$.
It now follows that $\overline{O_1}\sqcup\overline{O_2}\prec O_0$. Since by hypothesis the action is $(O_1, O_2 ,E)$-squarely divisible for all finite sets $E\subseteq G$ containing $e$, we deduce that the action is $O$-squarely 
divisible and hence squarely divisible.
\end{proof}

\begin{remark}
Knowing that a topologically free minimal action of an infinite group $G$ is $(O_1 , O_2 ,E)$-squarely divisible 
for all $O_1$, $O_2$, and $E$ (as in Theorem~\ref{T-amenable}) will imply stable rank one via Proposition~\ref{P-stronger}
and Theorem~\ref{T-SR1}, but one can reach this conclusion more directly by
simply taking $O_1$ and $O_2$ to be subsets of $O$ and dispensing with the construction of the 
unitary associated to the subequivalence $\overline{O_1} \sqcup \overline{O_2} \prec O_0 \subseteq O$ in the proof of Theorem~\ref{T-SR1}.
\end{remark}

In Definition~\ref{D-SD} the set $U$ and the subequivalences involving it 
are critical for allowing us to rotate the non-invertible element in the proof of Theorem~\ref{T-SR1}
to an element $b$ that is ``spectrally null'' around the boundary of $V$ (this will allow us to 
build ersatz permutation unitaries using homotopies that do not spectrally interfere with $b$,
so that these unitaries will act like genuine permutations on $b$ and rotate it to something nilpotent).
Not unrelated is the fact that when the space $X$ is zero-dimensional
the set $U$ and its associated subequivalences also allow us to
establish the simpler characterization of $(O_1 , O_2 , E)$-squarely divisibility in terms of clopen sets,
recorded in Proposition~\ref{P-SD0Dim} below. For this we need the following compactness principle
for subequivalences.

\begin{lemma}\label{L-subequivalence}
Let $G\curvearrowright X$ be an action. Let $A\subseteq X$ be a closed set and $B\subseteq X$ an open set
such that $A\prec B$. Then there exist an open set $V\supseteq A$ and an open set $W\subseteq B$ with $\overline{W} \subseteq B$
such that $\overline{V}\prec W$.
\end{lemma}

\begin{proof}
By assumption there exist a collection $\{ U_i \}_{i=1}^n$ of open subsets of $X$ covering $A$
and elements $s_1 , \dots , s_n \in G$ such that the sets $s_i U_i$ are pairwise disjoint subsets of $B$.
Using normality and compactness, we find for each $i=1,\ldots,n$ open subsets $V_i,W_i\subseteq X$ with
\[
V_i\subseteq \overline{V_i}\subseteq W_i\subseteq \overline{W_i}\subseteq U_i 
\]
so that $A\subseteq \bigcup_{i=1}^{n}V_i$. Set $V=\bigcup_{i=1}^{n}V_i$ and $W=\bigsqcup_{i=1}^{n}s_iW_i$. 
Then $V\supseteq A$, $\overline{W}\subseteq B$, and $\overline{V}\prec W$, completing the proof.
\end{proof}

\begin{lemma}\label{L-SD prime}
Let $G\curvearrowright X$ be an action on the Cantor set. Let $O_1$ and $O_2$ be open subsets of $X$ and let $E$ be a finite subset of $G$ containing $e$, and suppose that 
the action is $(O_1,O_2,E)$-squarely divisible. Then there are closed sets $C_1 \subseteq O_1$
and $C_2 \subseteq O_2$ such that for all open sets $O_1'$ and $O_2'$ satisfying
$C_1 \subseteq O_1' \subseteq O_1$ and $C_2 \subseteq O_2' \subseteq O_2$
the action is $(O_1',O_2',E)$-squarely divisible.
\end{lemma}

\begin{proof}
Apply Lemma~\ref{L-subequivalence} to the definition of $(O_1,O_2,E)$-square divisibility to find open subsets $\{W_i\}_{i=1}^{n}$ of $X$ such that $\overline{W_i}\subseteq O_1\cap \bigsqcup_{j=2}^{n}V_{i,j}\cap B^\comp$ and $\overline{V_{i,1}}\prec W_i$ for each $i=1,\ldots,n$, and set $C_1=\bigcup_{i=1}^{n}\overline{W_{i}}$. To define $C_2$ one
similarly applies Lemma~\ref{L-subequivalence} only this time using the conditions (ii) and (iii) in the definition 
of $(O_1,O_2,E)$-square divisibility.
\end{proof}

\begin{remark}\label{R-SD prime}
It follows from the construction in the proof of Lemma~\ref{L-SD prime} that if $O_1$ and $O_2$ are nonempty 
then so are $C_1$ and $C_2$.
\end{remark}

\begin{proposition}\label{P-SD0Dim}
Let $G\curvearrowright X$ be an action on the Cantor set. Let $O_1$ and $O_2$ be clopen subsets of $X$ and let $E$ be a finite subset of $G$ containing $e$. The action $G\curvearrowright X$ is $(O_1,O_2,E)$-squarely divisible if and only if there exist an $n\in\mathbb{N}$, a collection $\{V_{i,j}\}_{i,j=1}^{n}$ of pairwise equivalent and pairwise $E$-disjoint clopen subsets of $X$, and, writing $V=\bigsqcup_{i,j=1}^{n}V_{i,j}$, $V_1=\bigsqcup_{i=1}^{n}V_{i,1}$, $R=V^\comp$, and $B=V\cap (V^E)^\comp$, one has the following:
\begin{enumerate}
\item[(i)] $V_{i,1}\prec O_1\cap \bigsqcup_{j=2}^{n}V_{i,j}\cap B^\comp$ for every $i=1,\ldots,n$,
\item[(ii)] $R\prec O_2\cap V\cap (V_1\cup B)^\comp$, and 
\item[(iii)] $B\prec O_2\cap R$.
\end{enumerate}
Moreover, given a nonempty open set $O\subseteq X$, the action is $O$-squarely divisible if and only if there exist nonempty disjoint clopen sets $O_0, O_1,O_2\subseteq X$ with $O_0\subseteq O$ such that 

\begin{itemize}
\item[(iv)] $O_1\sqcup O_2\prec O_0$, and
\item[(v)] the action is $(O_1,O_2,E)$-squarely divisible for all finite sets $E\subseteq G$ containing $e$.
\end{itemize}
Finally, the action is squarely divisible if and only if it is $O$-squarely divisible for all nonempty clopen sets $O\subseteq X$.
\end{proposition}

\begin{proof}
Suppose that $G\curvearrowright X$ is $(O_1 , O_2 , E)$-squarely divisible. Then there exist $n\in\Nb$,
a collection $\{ \tilde{V}_{i,j} \}_{i,j}^n$ of pairwise equivalent and pairwise $E$-disjoint open subsets of $X$,
and an open neighbourhood $U$ of $\partial \tilde{V}$ satisfying the conditions in Definition~\ref{D-SD}
with $\tilde{V}= \bigsqcup_{i,j=1}^n \tilde{V}_{i,j}$, $\tilde{V}_1 = \bigsqcup_{i=1}^n \tilde{V}_{i,1}$, $\tilde{R} = \tilde{V}^\comp$, and
$\tilde{B} = \overline{\tilde{V}}\cap((\tilde{V}\cap \overline{U}^\comp)^E)^\comp$. Let $\{ \tilde{C}_k \}_{k\in K}$ and $\{ s_{k,i,j} \}_{i,j=1}^n$ for $k\in K$ be as in the definition
of pairwise equivalence for the sets $\tilde{V}_{i,j}$. 
Since the collection $\{\tilde{V}_{i,j}\}_{i,j=1}^{n}$ is $E$-disjoint, we can
use the continuity of the action, together with the compactness and normality of $X$,
to find open sets $Y_{i,j}\supseteq \tilde{V}_{i,j}$ which are $E$-disjoint.
In view of the separation properties these sets and group elements
are required to satisfy, the zero-dimensionality of $X$ implies the existence of clopen sets $\{C_k\}_{k\in K}$
such that $\tilde{C}_k \subseteq C_k$ for each $k\in K$ and
the sets $s_{k,i,j} C_k$ are pairwise disjoint with $s_{k,i,j} C_k\subseteq (\tilde{V}_{i,j}\cup U) \cap Y_{i,j}$ for $k\in K$ and $i,j=1,\ldots , n$. 
Indeed, one first finds pairwise disjoint clopen sets $W_{k,i,j}\subseteq (\tilde{V}_{i,j}\cup U) \cap Y_{i,j}$ 
satisfying $s_{k,i,j}\overline{\tilde{C}_k}\subseteq W_{k,i,j}$ for $k\in K$ and $i,j=1,\ldots,n$ and then defines 
$C_k =\bigcap_{i,j=1}^{n}s_{k,i,j}^{-1}W_{k,i,j}$.
For all $i,j=1,\dots , n$ set $V_{i,j} = \bigsqcup_{k\in K} s_{k,i,j} C_k \supseteq \tilde{V}_{i,j}$.
Then the sets $V_{i,j}$ are clopen, pairwise equivalent, and pairwise $E$-disjoint. Set 
$V = \bigsqcup_{i,j=1}^n V_{i,j} \supseteq \tilde{V}$,
$V_1 = \bigsqcup_{i=1}^n V_{i,1} \supseteq \tilde{V}_1$, $R = V^\comp \subseteq \tilde{R}$, 
and $B = V\cap ( V^E )^\comp$.

Since $\overline{\tilde{V}_{i,1}} \prec O_1 \cap \bigsqcup_{j=2}^n \tilde{V}_{i,j} \cap \tilde{B}^\comp$ for every $i=1,\ldots , n$ by condition (i) 
in the definition of square divisibility, Lemma~\ref{L-subequivalence} shows that we can
take the sets $C_k$ to be smaller if necessary so as to arrange for every $i=1,\ldots,n$ that
\[
V_{i,1} \prec O_1 \cap \bigsqcup_{j=2}^n \tilde{V}_{i,j} \cap \tilde{B}^\comp
\subseteq O_1 \cap \bigsqcup_{j=2}^n V_{i,j} \cap B^\comp.
\]
Since we chose $s_{k,i,j}C_k$ to be contained in $\tilde{V}_{i,j}\cup U$ for all $i$, $j$, and $k$, we have $V_1 \subseteq \tilde{V}_1\cup U$ and hence
\begin{align*}
\tilde{V}\cap(\tilde{V}_1\cup \tilde{B})^\comp
=\tilde{V}_1^\comp\cap (\tilde{V}\cap \overline{U}^\comp)^E 
&=(\tilde{V}_1\cup U)^\comp\cap (\tilde{V}\cap \overline{U}^\comp)^E \\ 
&\subseteq V_1^\comp\cap V^E = V\cap(V_1\cup B)^\comp.
\end{align*}
This implies
\[
R \subseteq \tilde{R}\prec O_2 \cap \tilde{V} \cap (\tilde{V}_1 \cup \tilde{B})^\comp 
\subseteq O_2 \cap V \cap (V_1 \cup B)^\comp .
\]
Finally, since we chose each $s_{k,i,j}C_k$ to be contained in $\tilde{V}\cup U$ we have $V \subseteq \tilde{V} \cup U$ and hence, applying this fact to get both of the inclusions below, 
\[
B \subseteq \tilde{B}\cup (\overline{U} \cap \tilde{R}) \prec O_2 \cap \overline{\tilde{V}\cup U}^\comp \subseteq O_2 \cap R .
\]
This establishes the forward implication in the first part of the lemma
statement. The reverse implication is immediate from the definition of $(O_1 , O_2 , E)$-square divisibility.

The reverse implication of the second part of the lemma statement is trivial, and
so let us assume, in the other direction, that the action is $O$-squarely divisible.
Then there are nonempty open sets $\tilde{O}_0 \subseteq O_0$ and $\tilde{O}_1 , \tilde{O}_2 \subseteq X$ with pairwise disjoint closures 
such that $\overline{\tilde{O}_1}\sqcup \overline{\tilde{O}_2}\prec \tilde{O}_0$ and the action is $(\tilde{O}_1,\tilde{O}_2,E)$-squarely divisible for all finite subsets $E\subseteq G$ containing $e$. 
By Lemma~\ref{L-SD prime} (see also Remark~\ref{R-SD prime}) and the zero-dimensionality of $X$,
we can then find, for each finite set $E\subseteq G$ containing $e$, 
nonempty clopen sets $C_{1,E}\subseteq \tilde{O}_1$ and $C_{2,E}\subseteq \tilde{O}_2$ such that for all open sets $O_1'$ and $O_2'$ satisfying
$C_{1,E}\subseteq O_1'\subseteq \tilde{O}_1$ and $C_{2,E}\subseteq O_2'\subseteq \tilde{O}_2$ the action is $(O_1',O_2',E)$-squarely divisible.
Consider the (nonempty) clopen sets $O_1:=\bigcup_{E} C_{1,E}$ and $O_2:=\bigcup_E C_{2,E}$, where the (countable) unions range over all finite subsets $e\in  E\subseteq G$. By construction, the action is $(O_1,O_2,E)$-squarely divisible for all finite subsets $e\in E\subseteq G$.
Moreover, since $O_1\subseteq \tilde{O}_1$ and $O_2\subseteq \tilde{O}_2$ 
we have $O_1\sqcup O_2\prec \tilde{O}_0$.
By Lemma~\ref{L-subequivalence} and the zero-dimensionality of $X$, there exists a nonempty clopen set $O_0\subseteq \tilde{O}_0$ such that $O_1\sqcup O_2\prec O_0$. Finally we observe that $O_0,O_1,O_2$ are pairwise disjoint and that $O_0\subseteq O$.

Since every nonempty open subset of $X$ contains a clopen set, the reverse implication in the third part of the lemma statement is immediate, while the forward
implication is trivial.
\end{proof}

We do not know if the set of squarely divisible actions in $\Act (G,X)$ is a $G_\delta$. 
The problem is the universal quantification over $E$ in condition~(v) of Definition~\ref{D-SD}, which is nested within the universal quantification over $O$ and existential quantification over $O_0,O_1$, and $O_2$.
In order to remedy this deficiency so that we can establish the genericity results of Sections~\ref{S-SD I} and \ref{S-SD II},
we formulate the following weaker property in the Cantor set setting 
that will both be a $G_\delta$ condition in the space of actions and imply stable rank one.
The definition appears somewhat ad hoc in the way that the set $FEF$ appears, but this enables
an order of quantification that permits us to construct the unitary in the first part of the proof of Theorem~\ref{T-SR1 weak SD}.

\begin{definition}\label{D-weak SD}
Let $G\curvearrowright X$ be an action on the Cantor set.
Given a nonempty clopen set $O\subseteq X$ and a finite set $e\in E\subseteq G$, the action
is {\it $(O,E)$-squarely divisible} if there exist pairwise disjoint nonempty clopen sets $O_0,O_1,O_2\subseteq X$ with $O_0\subseteq O$
and a finite symmetric set $e\in F\subseteq G$ with $O_1 \sqcup O_2 \prec_F O_0$ such that the action is $(O_1 , O_2 ,FEF)$-squarely divisible.
We say that the action is {\it weakly squarely divisible} if it is $(O,E)$-squarely divisible
for every nonempty clopen set $O\subseteq X$ and finite set $e\in E\subseteq G$.
\end{definition}

\begin{remark}\label{R-0dimOSD}
For an action $G\curvearrowright X$ on the Cantor set, we observe that $O$-square divisibility 
for a nonempty clopen set $O\subseteq X$ implies $(O,E)$-square divisibility for all finite sets $e\in E\subseteq G$ 
(see Proposition~\ref{P-SD0Dim}). 
In particular, square divisibility implies weak square divisibility. We do not know whether the converse holds.  
\end{remark}

To verify that weak square divisibility is a $G_\delta$ condition in the space
of actions on the Cantor set, we first record a couple of lemmas.

\begin{lemma}\label{L-openA}
Suppose $X$ is the Cantor set. Let $O_1,O_2$ be nonempty clopen subsets of $X$ and let $e\in E\subseteq G$ be a finite set. 
Then the set of all $\alpha\in \Act(G,X)$ that are $(O_1,O_2,E)$-squarely divisible is open.
\end{lemma}

\begin{proof}
Let $\alpha\in \Act(G,X)$ be $(O_1,O_2,E)$-squarely divisible. By Proposition~\ref{P-SD0Dim} we can find an $n\in\mathbb{N}$ 
and a collection $\{V_{i,j}\}_{i,j=1}^{n}$ of pairwise equivalent and pairwise $E$-disjoint clopen subsets of $X$ such that 
the conditions (i)--(iii) in Proposition~\ref{P-SD0Dim} are satisfied. Write $\Omega= \{1,\ldots,n\}^2$. 
The pairwise equivalence of the sets $\{V_{i,j}\}_{i,j=1}^{n}$ provides a finite clopen partition of $V_{1,1}=\bigsqcup_{k\in K}C_k$ and elements $s_{k,p}\in G$ for $k\in K$ and $p\in \Omega$ with $s_{k,(1,1)}=e$ for all $k\in K$ such that $V_p=\bigsqcup_{k\in K}\alpha_{s_{k,p}}C_k$. 
Let $\sP$ be a finite clopen partition of $X$ containing both $\{C_k\}_{k\in K}$ and $\{V_p\}_{p\in\Omega\setminus\{(1,1)\}}$, and let $F\subseteq G$ be a finite set containing $E\cup E^{-1}\cup \{s_{k,p}\}_{k\in K, p\in \Omega}$. For all $\beta\in U_{\alpha,F,\sP}$ the sets $\{V_{i,j}\}_{i,j=1}^{n}$ are pairwise equivalent under $\beta$, with the equivalence 
being implementable by the same sets $\{C_k\}_{k\in K}$ given that $\alpha_{s_{k,p}}C_k=\beta_{s_{k,p}}C_k$ for all $k\in K$ and $p\in \Omega$. The sets $\{V_{i,j}\}_{i,j=1}^{n}$ are also $E$-disjoint under $\beta$ since $\beta_sV_p=\alpha_sV_p$ for all $s\in E$ and $p\in\Omega$. Note also that, for the same reason, the boundary $B=V\cap (V^E)^\comp$, which in general depends on the action $\alpha$, remains unchanged for $\beta$.

With the sets $V_{i,j}$ and $B$ fixed in this way across all actions in $U_{\alpha,F,\sP}$,
we can now pass to a smaller open neighbourhood of $\alpha$ on which the subequivalences (i)--(iii) in Proposition~\ref{P-SD0Dim} 
always hold (the argument is identical to the proof of Proposition~\ref{P-subequivalence}, where we additionally had to control the group elements which implement the subequivalence).
\end{proof}

\begin{lemma}\label{L-open}
Suppose $X$ is the Cantor set.
Let $O$ be a nonempty clopen subset of $X$ and $e\in E$ a finite subset of $G$.
Then the set $\sW_{O,E}$ of all $(O,E)$-squarely divisible actions in $\Act (G,X)$ is open.
\end{lemma}

\begin{proof}
Let $\alpha\in\sW_{O,E}$. Then there exist nonempty disjoint clopen sets $O_0, O_1 , O_2 \subseteq X$ with $O_0\subseteq O$
and a finite symmetric set $e\in F\subseteq G$ with $O_1 \sqcup O_1 \prec_F O_0$ such that the action is $(O_1 , O_2 ,FEF)$-squarely divisible.
By Proposition~\ref{P-subequivalence}, the subequivalence $O_1 \sqcup O_1 \prec_F O_0$ is stable under 
perturbation of $\alpha$, and by Lemma~\ref{L-openA} so is the property of $(O_1 , O_2 ,FEF)$-square divisibility. It follows that $\alpha$ has a neighbourhood contained in $\sW_{O,E}$,
and so we conclude that $\sW_{O,E}$ is open.
\end{proof}

In the proof below we use the well-known fact that
the collection of clopen subsets of the Cantor set is countable. One can see this
C$^*$-algebraically by observing that the projections in $C(X)$ 
(which are precisely the indicator functions of clopen sets) are at norm distance $1$ 
from each other. Since $C(X)$ is separable, this implies that there are only countably many projections. 

\begin{proposition}\label{P-SD G delta}
Suppose $X$ is the Cantor set.
Then the set of all weakly squarely divisible actions in $\Act (G,X)$ is a $G_\delta$.
\end{proposition}

\begin{proof}
Write $\sO$ for the countable collection of nonempty clopen subsets of $X$ and
$\sE$ for the countable collection of finite subsets of $G$ containing $e$.
By Lemma~\ref{L-open}, for every $O\in\sO$ and $E\in\sE$ the set $\sW_{O,E}$ of all $(O,E)$-squarely divisible
actions in $\Act (G,X)$ is open, and so the intersection $\bigcap_{O\in\sO , E\in\sE} \sW_{O,E}$,
which is equal to the set of all weakly squarely divisible actions in $\Act (G,X)$, is a $G_\delta$.
\end{proof}

One obstruction to square divisibility and weak square divisibility is strong ergodicity.
Recall that a p.m.p.\ action $G\curvearrowright (Z,\zeta )$ is {\it strongly ergodic} if for every 
sequence $( W_n )_{n\in\Nb}$ of Borel subsets of $X$ with $\lim_{n\to\infty}\zeta (sW_n \Delta W_n )= 0$ for all $s\in G$
one has $\lim_{n\to\infty}\zeta (W_n )(1-\zeta (W_n )) = 0$.

\begin{proposition}\label{P-strong ergodicity}
Let $G\curvearrowright X$ be a minimal squarely divisible action on an infinite compact metrizable space
or a minimal weakly squarely divisible action on the Cantor set. Let $\mu\in M_G(X)$.
Then the p.m.p.\ action $G\curvearrowright (X,\mu )$ is not strongly ergodic.
\end{proposition}

\begin{proof}
The argument is the same in both cases, and so
we assume that $G\curvearrowright X$ is a squarely divisible minimal action on an infinite compact metrizable space.
This implies that $\mu$ is atomless and has full support.
Thus for every $m\in\Nb$ there exists an open set $O_m \subseteq X$ 
with $0<\mu(O_m)\leq 1/m$.
Since $G$ is countable, there exists an increasing sequence $( E_m )_{m=1}^\infty$ of finite subsets of $G$ containing $e$ such that $G=\bigcup_{m=1}^\infty E_m$.
Let $m\in\Nb$. By assumption our action is $O_m$-squarely divisible, and so there exist nonempty open sets $O_{m,0}\subseteq O_m$ and $O_{m,1},O_{m,2}\subseteq X$ with disjoint closures such that $\overline{O_{m,1}}\sqcup \overline{O_{m,2}}\prec O_{m,0}$ and the action is $(O_{m,1},O_{m,2},E_m)$-squarely divisible for each $m\in \Nb$. 
Denote by $\{V_{m,i,j}\}_{i,j=1}^{n_m}$ the sets obtained by this square divisibility and set $V_m=\bigsqcup_{i,j=1}^{n_m}V_{m,i,j}$.
Since for every closed set $A\subseteq X$ and open set $B\subseteq X$ the relation $A\prec B$ clearly implies $\mu (A) \leq \mu (B)$, condition~(iii) in Definition~\ref{D-SD} implies that $\mu(V_m\setminus V_m^{E_m})\leq 1/m$. 

Set $\tilde{W}_m=\bigsqcup_{i=1}^{\lfloor{n_m/2}\rfloor}\bigsqcup_{j=1}^{n_m}V_{m,i,j}$ and $W_m= \tilde{W}_m^{E_m}$.
We claim that
\begin{gather}\label{E-claimone}
\lim_{m\to\infty}\mu(sW_m\Delta W_m)=0
\end{gather}
for all $s\in G$.
Let $s\in G$ and $\eps>0$.
Pick $M\in\Nb$ so that $1/M<\eps /2$ and $s\in E_M$, and let $m\geq M$.
We will show that $V_m^{E_m}\cap \tilde{W}_m\subseteq W_m$. 
Let $x\in V_m^{E_m}\cap \tilde{W}_m$ and $r\in E_m$.
It suffices to verify that $rx\in \tilde{W}_m$. 
Since $x\in \tilde{W}_m$ and $V_{m,i,j}$ are $E_m$-disjoint, we have
$rx\notin \bigsqcup_{i=\lfloor{n_m/2}\rfloor +1}^{n_m}\bigsqcup_{j=1}^{n_m}V_{m,i,j}$. On the other hand we have $rx\in V_m$
by assumption and hence $rx\in \tilde{W}_m$, as we wanted to show.
Observe also that $sW_m\cup W_m\subseteq \tilde{W}_m$.
Using the invariance of $\mu$, we obtain
\begin{align*}
\mu(sW_m\Delta W_m)
&\leq \mu(\tilde{W}_m\setminus W_m)+\mu(\tilde{W}_m\setminus sW_m)\\
&=2\mu(\tilde{W}_m\setminus W_m)\\
&\leq 2\mu( \tilde{W}_m\setminus (V_m^{E_m}\cap \tilde{W}_m) )\\
&= 2\mu( \tilde{W}_m\setminus V_m^{E_m} )\\
&\leq 2\mu(V_m\setminus V_m^{E_m}) \leq \eps ,
\end{align*}
which verifies (\ref{E-claimone}).

We next check that
\begin{gather}\label{E-claimtwo}
\lim_{m\to\infty}\mu(W_m)=\frac12 ,
\end{gather}
which will establish strong ergodicity.
Let $\eps>0$.
Since the sets $\{V_{m,i,j}\}_{i,j=1}^{n_m}$ are pairwise equivalent and $\mu$ is invariant, they all have equal $\mu$-measure.
Writing $R_m=V_m^\comp$, condition~(ii) in Definition~\ref{D-SD} implies that $\mu(R_m)\leq 1/m$.
Together with condition~{(i)} in Definition~\ref{D-SD}, we get
\[
1=n_m^2\mu(V_{m,1,1})+\mu(R_m)\leq \frac{n_m^2+1}{m}.
\]
Let $M\in \Nb$ be such that $1/\sqrt{M-1}<\eps /2$ and let $m\geq M$.
Note that $1/n_m<\eps /2$. Clearly $\mu(W_m)\leq 1/2$,
and so we aim at showing that $\mu(W_m)\geq 1/2-\eps$.
The pairwise equivalence of the sets $V_{m,i,j}$ implies that
$2\mu(\tilde{W}_m)+n_m\mu(V_{m,1,1})\geq \mu(V_m)$.
Since $\mu(V_{m,1,1})\leq 1/n_m^{2}$, we obtain
\begin{align*}
\mu(\tilde{W}_m)\geq \frac{1}{2}\mu(V_m)-\frac{1}{2}n_m\mu(V_{m,1,1}) \geq \frac{1}{2}\mu(V_m)-\frac{1}{2n_m} \geq \frac{1}{2}\mu(V_m)-\frac{\eps}{4}. 
\end{align*}
Since $\mu(R_m)\leq 1/m< \eps /2$ we have $\mu(V_m)\geq 1-\eps /2$, and
we conclude that $\mu(\tilde{W}_m)\geq 1/2- \eps /2$.
As shown in the first half of the proof, $\mu(\tilde{W}_m\setminus W_m)\leq \mu(V_m\setminus (V_m)^{E_m})\leq 1/m
<\eps /2$.
Thus
\[\mu(W_m)=\mu(\tilde{W}_m)-\mu(\tilde{W}_m\setminus W_m)\geq \frac{1}{2}-\eps, \]
which verifies (\ref{E-claimtwo}).
\end{proof}

\begin{remark}
Suppose that the action $G\curvearrowright (X,\mu )$ in Proposition~\ref{P-strong ergodicity}
is free and ergodic, in which case the von Neumann algebra crossed product $L^\infty (X,\mu )\rtimes G$
is a II$_1$ factor.
The fact that the action is not strongly ergodic implies 
that $L^\infty (X,\mu )\rtimes G$ has property $\Gamma$. 
Since for a given $m\in\Nb$ the sets $\{ V_{m,i,j} \}_{i,j=1}^{n_m}$ in the proof are pairwise equivalent,
one can use them to construct matrix embeddings $M_{n_m^2} \hookrightarrow L^\infty (X,\mu )\rtimes G$
that are almost unital in trace norm (up to the indicator function of the complement of their union). However, while
the indicator functions of the set $W_m$ are asymptotically central in $L^\infty (X,\mu )\rtimes G$
(thereby witnessing property $\Gamma$) the images of these matrix algebras, or of any of their noncommutative
unital matrix subalgebras, cannot be made asymptotically central in general, as this would show that
$L^\infty (X,\mu )\rtimes G$ has the McDuff property, which would imply
that $G$ is inner amenable \cite[Proposition 4.1]{DepVae18}. 
This would mean in particular that $G$ cannot be one of the free groups $F_d$ for $d\geq 2$,
although these groups admit plenty of squarely divisible actions possessing an invariant Borel probability
measure $\mu$ for which the corresponding p.m.p.\ action is free and ergodic, as shown 
in Propositions~\ref{P-examples} and \ref{P-SD examples}.
\end{remark}

\begin{example}\label{E-notSD}
By \cite{Wei12,Ele21}, if $G$ is countably infinite then there exists a free minimal action $G \curvearrowright X$ on the Cantor set which is universal
for free p.m.p.\ actions of $G$, i.e., for every free measure-preserving action $G \curvearrowright (Z,\zeta )$ on a standard atomless
probability space there exists a $\mu\in M_{G} (X)$ such that the p.m.p.\ actions $G \curvearrowright (X,\mu )$
and $G \curvearrowright (Z,\zeta )$ are measure conjugate.

Fix a countably infinite $G$ and let $(Y_0,\nu_0)$ be a standard probability space
which is nontrivial, i.e., which has no atom of full measure. It is well-known that the Bernoulli action
$G\curvearrowright (Y_0^G,\nu_0^G)$ is free and mixing.
If $G$ is moreover assumed to be nonamenable then $G\curvearrowright (Y_0^G,\nu_0^G)$ is also strongly ergodic.
By the previous paragraph, there exists a free minimal action $G\curvearrowright X$ on the Cantor set and $\mu\in M_G(X)$ such that
the actions $G\curvearrowright (X,\mu)$ and $G\curvearrowright (Y_0^G,\nu_0^G)$ are measure conjugate.
Then the p.m.p.\ action $G\curvearrowright (X,\mu)$ is free, mixing, and strongly ergodic. 
By Proposition~\ref{P-strong ergodicity}, the action $G\curvearrowright X$ is neither squarely divisible nor weakly squarely divisible.
Since $\mu$ has full support by minimality, we also see that $G\curvearrowright X$ is (topologically) mixing.

This shows in particular that for $G=F_d$ with $d\geq 2$ there 
exists a free minimal mixing action $F_d \curvearrowright X$ on the Cantor set
that is neither squarely divisible nor weakly squarely divisible. This action
belongs to the space $\WA (F_d ,X)$ that is
the subject of the generic weak square divisibility result in Section~\ref{S-SD I}
and of Corollary~\ref{C-main 1} in the introduction (see Definition~\ref{D-spaces}).
\end{example}

To round out this section we prove that square divisibility is an invariant of continuous orbit equivalence.
Let $G\curvearrowright X$ and $H\curvearrowright Y$ be two free actions and $\varphi : X\to Y$
a homeomorphism such that $\varphi (Gx) = H\varphi (x)$ for all $x\in X$. Then there are cocycles
$\kappa : G\times X\to H$ and $\lambda : H\times Y\to G$ uniquely determined, as a consequence of freeness,
by the equations $\varphi (gx) = \kappa (g,x)\varphi (x)$ and $\varphi^{-1} (hy) = \lambda (h,y)\varphi^{-1} (y)$.
These cocycles are Borel but not in general continuous. When $\kappa$ and $\lambda$ are continuous
we say that $\varphi$ is a {\it continuous orbit equivalence} between the two actions. The actions
are said to be {\it continuously orbit equivalent} if there exists a continuous orbit equivalence between them.
Since continuity of the cocycles implies that they are locally constant, the theory of continuous orbit equivalence
is primarily of interest for spaces that are zero-dimensional, or at least far from being connected.

Suppose that $\varphi :X\to Y$ is a continuous orbit equivalence between two
free actions
$G\stackrel{\alpha}{\curvearrowright} X$ and $H\stackrel{\beta}{\curvearrowright} Y$ 
with cocycle $\kappa : G\times X\to H$ as above. Since $\kappa$ is continuous and $H$ is discrete, 
for each $g\in G$ the image $\kappa (g,X)$ is finite in $H$. If $A\subseteq X$ is closed, 
$B\subseteq X$ is open, and $F$ is a finite subset of $G$ such that $A\prec_{\alpha ,F} B$,
then setting $L = \kappa (F, X)$ (which is finite by the continuity of $\kappa$) one has $\varphi (A)\prec_{\beta ,L} \varphi (B)$, for if $\{ U_g \}_{g\in F}$ is an open
cover of $A$ such that the sets $\alpha_g U_g$ for $g\in F$ are pairwise disjoint subsets of $B$ then 
setting $D_{g,h} = \{ x\in X : \kappa (g,x) = h \}$ for $g\in G$ and $h\in H$ the collection
$\{ \varphi (D_{g,h}\cap U_g) \}_{g\in F,\,h\in \kappa (g, X) }$ is an open cover of $\varphi (A)$ such that the sets
$\beta_h \varphi(D_{g,h}\cap U_g)$ for $g\in F$ and $h\in \kappa (g,X)$ are pairwise disjoint subsets of $\varphi (B)$.
Using these observations one can show the following.

\begin{proposition}
Let $G\curvearrowright X$ and $H\curvearrowright Y$ be two free actions which are continuously orbit equivalent. Suppose that $G\curvearrowright X$ is squarely divisible. Then $H\curvearrowright Y$ is squarely divisible. 
\end{proposition}

\begin{proof}
Let $\varphi: X\to Y$ be a continuous orbit equivalence between the two actions,
with corresponding cocycle maps $\kappa: G\times X\to H$ and $\lambda : H\times Y\to G$. 
Let $O\subseteq Y$ be a nonempty open set. We want to show that $H\curvearrowright Y$ is $O$-squarely divisible.
Since $G\curvearrowright X$ is $\varphi^{-1}(O)$-squarely divisible, 
one can find nonempty open sets $O_0,O_1,O_2\subseteq X$ with disjoint closures such that $O_0\subseteq \varphi^{-1}(O)$, $\overline{O_1}\sqcup \overline{O_2}\prec O_0$, and $G\curvearrowright X$ is $(O_1,O_2,E)$-squarely divisible for all finite sets $e\in E\subseteq G$.
Considering the sets $\varphi(O_0),\varphi(O_1),\varphi(O_2)\subseteq Y$, 
it is clear, based on the discussion before the proposition statement, that $H\curvearrowright Y$ will be squarely divisible if we
manage to show that it is $(\varphi(O_1),\varphi(O_2),F)$-squarely divisible for all finite sets $e\in F\subseteq H$.

Let $e\in F\subseteq H$ be a finite set. 
Consider the finite set $E:=\lambda(F, Y)\subseteq G$. Then $e\in E$, and by assumption $G\curvearrowright X$ is $(O_1,O_2,E)$-squarely divisible.
Let $\{\tilde{V}_{i,j}\}_{i,j=1}^{n}$ be open sets in $X$ implementing this square divisibility, and denote by $R$ and $B$ the remainder and boundary, respectively (as in Definition~\ref{D-SD}). Set $V_{i,j}=\varphi(\tilde{V}_{i,j})$ for $i,j=1,\ldots,n$ and write $\tilde{R}$ and $\tilde{B}$ for the corresponding remainder and boundary. We claim that these sets witness $(\varphi(O_1),\varphi(O_2),F)$-square divisibility for $H\curvearrowright Y$.
It is a straightforward computation to verify that the $E$-disjointness of 
the sets $\{\tilde{V}_{i,j}\}_{i,j=1}^{n}$ implies that the sets $\{V_{i,j}\}_{i,j=1}^{n}$ are $F$-disjoint.
It is more subtle to check that the sets $\{V_{i,j}\}_{i,j=1}^{n}$
are pairwise equivalent. 
To do that, let $\{\tilde{C}_k\}_{k\in K}$ be Borel subsets of $X$ and $\{s_{k,i,j}\}_{k\in K,i,j=1,\ldots,n}$ elements of $G$ such that $\tilde{V}_{1,1}=\bigsqcup_{k\in K} \tilde{C}_k$ and 
$\tilde{V}_{i,j}=\bigsqcup_{k\in K} s_{k,i,j}\tilde{C}_k$ for all $i,j=1,\ldots, n$.
Fix $i,j\in \{1,\ldots,n\}$, with $(i,j)\neq (1,1)$.
As before define $D_{g,h}=\{x\in X: \kappa(g,x)=h\}$ for $g\in G$ and $h\in H$.
One checks that 
\[
V_{1,1}=\bigsqcup_{k\in K}\,\bigsqcup_{h\in \kappa(s_{k,i,j},X)}\varphi(\tilde{C}_k\cap D_{s_{k,i,j},h}), 
\]
and 
\[
V_{i,j}=\bigsqcup_{k\in K}\,\bigsqcup_{h\in \kappa(s_{k,i,j},X)}h\varphi(\tilde{C}_k\cap D_{s_{k,i,j},h}). 
\]
Thus the partition $\mathcal{C}_{i,j}:=\{\varphi(\tilde{C}_k\cap D_{s_{k,i,j},h}) : 
k\in K, \, h\in \kappa(s_{k,i,j},X) \}$ implements an equivalence between $V_{1,1}$ and $V_{i,j}$.
To obtain a pairwise equivalence between the sets $\{V_{i,j}\}_{i,j=1}^{n}$ we need however to find a partition of $V_{1,1}$ that works simultaneously for all $i,j=1,\ldots,n$.
Consider $\mathcal{C}=\bigvee_{i,j=1,\ldots,n}\mathcal{C}_{i,j}$, i.e., the common refinement.
This is a partition of $V_{1,1}$ into Borel sets which is easily seen to implement the desired pairwise equivalence. 
The discussion preceding the proposition statement now completes the proof
modulo the fact that $B\subseteq \varphi(\tilde{B})$, which follows from the inclusion
$\varphi(A^E)\subseteq \varphi(A)^F$ for every $A\subseteq X$.  
\end{proof}

We do not know whether weak square divisibility for free actions on the Cantor set is preserved under continuous orbit equivalence.

\section{Stable rank one}\label{S-stable rank one}

Our goal here is to show that, for topologically free minimal actions, 
stable rank one is a consequence of square divisibility (Theorem~\ref{T-SR1}) and also, in the Cantor set setting,
of weak square divisibility (Theorem~\ref{T-SR1 weak SD}). This we do by recasting some of the main
ideas of \cite{LiNiu20}, which in turn have precedents in \cite{Ror91,Ror04,Suz20,AlbLut22,Lin25}. The common strategy 
originates in the work of R{\o}rdam in \cite{Ror91}. 
R{\o}rdam's first insight was that if a simple unital C$^*$-algebra is finite 
then in order to approximate a non-invertible element $a$ by invertible elements one may assume,
by performing a unitary rotation and perturbation,
that $ad = da = 0$ for some nonzero positive element $d$.
Under favourable circumstances (such as $\cZ$-stability \cite{Ror04} or, as we will show, square divisibility)
the element $a$ may then be rotated by unitaries to produce a nilpotent element, which can then be 
perturbed to something invertible by adding a small scalar multiple of the identity. One can then rotate 
such a perturbation back to an (invertible) element that approximates the original non-invertible element,
yielding stable rank one. In the dynamical framework one furthermore needs the element $d$ above to be of 
a special form, and for this we will appeal to a result in \cite{LiNiu20}.

The argument in the Cantor setting is much simpler as we can build unitaries and partial isometries directly
from clopen sets and group elements, without having to negotiate boundaries. The reader seeking to get
the quickest grasp of the basic mechanisms at play is thus advised to start with the proof of Theorem~\ref{T-SR1 weak SD}.

Given an action $G\curvearrowright X$,
the embedding of $C(X)$ into the C$^*$-algebra $B(X)$ of bounded Borel functions on $X$ extends canonically to an embedding of $C(X)\rtimes_\lambda G$ into $B(X)\rtimes_\lambda G$. When convenient we will accordingly view $C(X)\rtimes_\lambda G$ as sitting inside the larger C$^*$-algebra $B(X)\rtimes_\lambda G$.
To be specific, we will find it useful to be able to express certain multiplicative relationships using indicator functions of Borel sets.

The following is a version of Lemma~3.2 in \cite{LiNiu20} that is tailored to the setting of dynamical subequivalence
and relies on the same functional calculus construction.

\begin{lemma}\label{L-unitary}
Let $G\curvearrowright X$ be an action. Let $A\subseteq X$ be a closed set and $B\subseteq X$ an open set such that 
$A\cap B = \emptyset$ and $A\prec B$. Let $A_0 \subseteq X$ be an open set with $A\subseteq A_0$.
Then there are a closed set $B^- \subseteq B$, an open set $A^+ \subseteq X$ satisfying $A\subseteq A^+\subseteq A_0$ and $A^+\cap B^-=\emptyset$, and a unitary $u\in C(X)\rtimes_\lambda G$ such that the following hold:
\begin{enumerate}
\item $u\unit_{A^+\cup D}=\unit_{B^-\sqcup D}u\unit_{A^+\cup D}$ for every Borel set $D\subseteq X$ with $D\cap B^-=\emptyset$, and in particular $u\unit_{A^+} = \unit_{B^-}u \unit_{A^+}$,

\item the same as (i) with $u^*$ in place of $u$,

\item $\unit_{A_0\cup B^-}(u-1)=u-1=(u-1)\unit_{A_0\cup B^-}$, and in particular $u1_{C}=1_{A_0\cup B^-\cup C}u1_{C}$ and $1_{C}u=1_{C}u1_{A_0\cup B^-\cup C}$ for every Borel set $C\subseteq X$,
\item $\unit_C u = u\unit_C = \unit_C$ where $C = (A_0\cup B^- )^\comp$.
\end{enumerate}
\end{lemma}

\begin{proof}
Since $A\prec B$ we can find open sets $U_1 , \dots , U_n \subseteq X$ and elements $s_1 , \dots , s_n \in G$
such that $A\subseteq \bigcup_{i=1}^n U_i \subseteq A_0$ and the sets $s_i U_i$ for $i=1, \dots , n$ are pairwise
disjoint and all contained in $B$. 
We may assume, by using the compactness of $A$ to shrink the sets $U_i$ slightly, that the closures of the sets $s_iU_i$ for $i=1,\ldots,n$ are pairwise disjoint and contained in $B$. Define $B^-$ to be the union of these closures.

Again using the compactness of $A$ as above, we can find open sets $W_i\subseteq \overline{W_i}\subseteq U_i$ for $i=1,\ldots,n$ such that $A\subseteq \bigcup_{i=1}^n W_i$.
Proceeding as in the usual construction of a partition of unity for the open cover $\{U_1,\ldots, U_n, (X\setminus \bigcup_{i=1}^{n} \overline{W_i})\}$ of $X$, we can 
find $h_1 , \dots , h_n\in C(X,[0,1])$ such that $h_i = 0$ off of $U_i$ and the function $h:=\sum_{i=1}^n h_i$
takes values in $[0,1]$ and is equal to $1$ on $\bigcup_{i=1}^{n} \overline{W_i}$. 

A separation argument applied to the closed disjoint sets $A$ and $B^-$ allows us to find an open set $A^+\subseteq X$ such that $A\subseteq A^+\subseteq \overline{A^+}\subseteq \bigcup_{i=1}^{n}W_i$ and $\overline{A^+}\cap B^-=\emptyset$. 
Taking $f\in C(X,[0,1])$ to be a function with $f|_{B^{-}}=0$ and $f|_{A^+}=1$ and multiplying each $h_i$ with $f$, we may assume that $h|_{B^{-}}=0$ and  $h|_{A^+}=1$.
Note that $s_i h_i=0$ off of $s_iU_i\subseteq B^-$ and therefore $s_i h_i\perp s_j h_j$ for $i\neq j$, 
and $s_i h_i\perp h$ for $i=1,\ldots,n$. 

Set $v = \sum_{i=1}^n h_i^{1/2}u_{s_i}^*$. Then
\begin{align*}
vv^* = \sum_{i,j} h_i^{1/2} u_{s_i}^* u_{s_j} h_j^{1/2} = \sum_{i,j} u_{s_i}^* (s_i h_i^{1/2} ) (s_j h_j^{1/2} )u_{s_j}
= \sum_i u_{s_i}^* (s_i h_i )u_{s_i} = h
\end{align*}
while
\begin{align*}
v^*v = \sum_{i,j} u_{s_i} h_i^{1/2} h_j^{1/2} u_{s_j}^*
= \sum_{i,j} (s_i h_i^{1/2} ) u_{s_i} u_{s_j}^* (s_j h_j^{1/2} ) ,
\end{align*}
from which we see that $vv^*\leq 1$ and $v^* v \perp vv^*$. We are now in the situation of \cite[Lemma~3.2]{LiNiu20},
which shows, writing $w$ for the self-adjoint element $vv^* + v^*v$ and $g$ for the continuous function on $[0,\infty )$ defined by $g(t)=\sin (\pi t/2)/\sqrt{t}$ 
for $t>0$ and $g(0)=0$, that the element
\begin{gather*}
u:= \cos \Big ( \frac{\pi}{2} w\Big) + g(vv^*)v - g(v^* v)v^* 
\end{gather*}
is unitary. 

We now verify (i). Let $D\subseteq X$ be a Borel set satisfying $D\cap B^-=\emptyset$. 
By construction each $s_i h_i^{1/2}$ vanishes off of $B^-$, and so we have 
\begin{align*}
v1_{A^+\cup D}=\sum_i h_i^{1/2}u_{s_i}^*1_{A^+\cup D}=\sum_i u_{s_i}^* (s_i h_i^{1/2})1_{A^+\cup D}=0 .
\end{align*}
The vanishing of each $s_i h_i^{1/2}$ off of $B^-$ also yields
\begin{align*}
1_{B^-}v^*=\sum_i 1_{B^-}u_{s_i}h_i^{1/2}=\sum_i 1_{B^-} (s_i h_i^{1/2})u_{s_i}=v^* 
\end{align*}
and hence $1_{B^-\sqcup D}v^*v=v^*v$, which shows, using that $g(0) = 0$, that
\begin{align*}
1_{B^-\sqcup D}g(v^*v)v^*1_{A^+\cup D}=g(v^*v)v^*1_{A^+\cup D}.
\end{align*}
Finally, since $v^*v1_{A^+\cup D}=0$, $vv^*\perp v^*v$, and $h|_{A^+} = 1$ we have 
\begin{align*}
\cos \Big ( \frac{\pi}{2} w\Big)1_{A^+\cup D}=\cos \Big ( \frac{\pi}{2} vv^*\Big)1_{A^+\cup D}=\cos \Big ( \frac{\pi}{2} h\Big)1_{A^+\cup D}=\cos \Big ( \frac{\pi}{2} h\Big)1_{D} ,
\end{align*}
while the fact that $1_D$ commutes with $\cos ( \pi h/2)$ implies
\begin{align*}
1_{B^-\sqcup D}\cos \Big ( \frac{\pi}{2} w\Big)1_{A^+\cup D}=\cos \Big ( \frac{\pi}{2} w\Big)1_{A^+\cup D}.
\end{align*}
From these computations we obtain (i).

To see that (i) holds also for
\[
u^*=\cos \Big(\frac{\pi}{2}w\Big)+v^*g(vv^*)+vg(v^*v) ,
\]
we observe that $1_{B^-}v^*=v^*$ and hence $1_{B^-\sqcup D}v^*=v^*$ so that
\[
1_{B^-\sqcup D}v^*g(vv^*)1_{A^+\cup D}=v^*g(vv^*)1_{A^+\cup D},
\]
while from the equalities $v^*v1_{A^+\cup D}=0$ and $g(0)=0$ we obtain
\[
1_{B^-\sqcup D}vg(v^*v)1_{A^+\cup D}=0=vg(v^*v)1_{A^+\cup D}.
\]
This yields $1_{B^-\sqcup D}u^*1_{A^+\cup D}=u^*1_{A^+\cup D}$, as desired.

To establish (iii) one first checks that $1_{A_0\cup B^-}v=v$ (using that each $h_i$ vanishes off of $A_0$) and $1_{A_0\cup B^-}v^*=v^*$ (using that each $s_i h_i$ is supported in $B^-$). 
It follows that $1_{A_0\cup B^-}$ acts like a unit on $w$. Since $\cos(0)=1$ we have
\begin{align*}
1_{A_0\cup B^-}\Big( \cos \Big ( \frac{\pi}{2} w\Big) -1\Big)=\cos \Big ( \frac{\pi}{2} w\Big)-1
\end{align*}
and so altogether we obtain $1_{A_0\cup B^-}(u-1)=u-1=(u-1)1_{A_0\cup B^-}$.
In particular, if $C\subseteq X$ is any Borel subset then $1_{A_0\cup B^-\cup C}$ acts like a unit on $u-1$.
Therefore $u= (u-1)1_{A_0\cup B^-\cup C}+1$, and so
$1_{C}u=1_C u1_{A_0\cup B^-\cup C}.$
Similarly, $u1_{C}= 1_{A_0\cup B^-\cup C}u1_C$. This verifies condition~(iii).

It remains to show that $1_Cu=u1_C=1_C$ where $C=(A_0\cup B^-)^c$. Again following the supports of the functions involved, it is immediate that $vv^*\perp 1_C$ and $v^*v\perp 1_C$, from which we obtain $v\perp 1_C$ using the C$^*$-identity. This in turn implies $1_Cu=u1_C=\cos(0)1_C=1_C$.
\end{proof}

We make the following two observations in the framework of Lemma~\ref{L-unitary}.

\begin{remark}\label{rem:zero-dim-unitary}
Suppose that $X$ is zero-dimensional and that $A,B\subseteq X$ are disjoint clopen sets with $A\prec B$. Then 
the picture in Lemma~\ref{L-unitary} becomes much simpler and we can avoid the use of the functional calculus to construct the unitary $u$. 
By \cite[Proposition~3.5]{Ker20} there exist a clopen partition $A=\bigsqcup_{i=1}^{n} A_i$ and elements $s_1,\dots, s_n\in G$ such that $B^- := \bigsqcup_{i=1}^{n}s_iA_i\subseteq B$.
The element
\[u := 1_{X\setminus (A\sqcup B^-)}+\sum_{i=1}^{n}1_{s_iA_i}u_{s_i}+\sum_{i=1}^{n}1_{A_i}u_{s_i}^*\in C(X)\rtimes_\lambda G\]
is a self-adjoint unitary such that 
$u1_{A}=1_{B^{-}}u$ and $1_Cu=u1_C=1_C$ where $C= (A\sqcup B^-)^c$.
\end{remark}

\begin{remark}\label{rem:commutation_relation}
Suppose that $A_1,\dots,A_n \subseteq X$ are closed sets and $A_{1,0} , \dots , A_{n,0} , B_1,\dots,B_n\subseteq X$ are open sets
such that $A_i\prec B_i$ and $A_i \subseteq A_{i,0}$ for all $i$ and $(A_{i,0}\cup B_i)\cap (A_{j,0}\cup B_j)=\emptyset$
for all $i\neq j$. For every $i$ let $u_i$ be a unitary as given by Lemma~\ref{L-unitary} with respect to the 
subequivalence $A_i\prec B_i$, with the set $A_0$ in the statement of the lemma taken to be $A_{i,0}$.
Then $u_iu_j=u_ju_i$, for all $i,j$. Indeed, setting $Y_i=A_{i,0}\cup B_i^-$ we have $Y_i\cap Y_j=\emptyset$ for $i\neq j$,
and by condition~(iv) in Lemma~\ref{L-unitary} we have $u_i=1_{Y_i}u_i1_{Y_i}+1_{Y_i^c}$. One can now easily check that $u_iu_j=1_{Y_i}u_i1_{Y_i}+1_{Y_j}u_j1_{Y_j}+1_{Y_i^c}1_{Y_j^c}=u_ju_i$.
\end{remark}

The following lemmas record a couple of simple facts that will be of use in 
the homotopy constructions in the proof of Theorem~\ref{T-SR1}.

\begin{lemma}\label{lemma:characterCts}
Let $C\subseteq V\subseteq X$ be open sets such that $C$ is relatively clopen in $V$. 
Let $h\in C(X)$ be a function which vanishes off of $V$. Then $h1_{C}\in B(X)$ actually belongs to $C(X)$.
\end{lemma}

\begin{proof}
Let $(x_n)_{n\in\mathbb{N}}$ be a sequence in $X$ converging to some point $x\in X$.  If $x\in C$, then since $C$ is open in $X$ one has $x_n\in C$ for all sufficiently large $n$, so that
\[
\lim_{n\to\infty} h(x_n) 1_{C}(x_n)=\lim_{n\to\infty} h(x_n)= h(x)=h(x)1_C(x).
\]
If $x\notin C$ but $x\in \partial C$, then from the fact that $V$ is open and the complement of $C$ is relatively open in $V$ we infer that $x\notin V$, so that
$\lim_{n\to\infty} h(x_n)=h(x)=0$ and hence
$\lim_{n\to\infty} h(x_n) 1_{C}(x_n)=0= h(x)1_C(x)$. Finally, if $x\in X\setminus\overline{C}$ then $x_n\in X\setminus\overline{C}$
for all sufficiently large $n$, in which case $\lim_{n\to\infty} h(x_n) 1_{C}(x_n)=0= h(x)1_C(x)$.
\end{proof}

\begin{lemma}\label{lemma:FuncMatrix}
Let $A$ be a unital C$^*$-algebra such that $C(X)$ is a C$^*$-subalgebra of $A$ sharing the same unit. Let $\{e_{i,j}\}_{i,j=1}^{n}$ be matrix units in $A$ and let $\{g_{i,j}\}_{i,j=1}^{n}$ be functions in $C(X)$ such that $g_{i,j}e_{k,\ell}=e_{k,\ell}g_{i,j}$ for all $i,j,k,\ell=1,\ldots,n$. Suppose that for every $x\in X$ the element $\sum_{i,j=1}^{n}g_{i,j}(x)e_{i,j}$ is a unitary in $D := C^*(\{e_{i,j} \}_{i,j=1}^n) \cong M_n$. 
Then the element $a:=\sum_{i,j=1}^{n}g_{i,j}e_{i,j}$ is a unitary in $C^*(\{g_{i,j}e_{i,j} \}_{i,j=1}^n )$. 
\end{lemma}

\begin{proof}
Thanks to the commutation relations in the hypotheses, one has a $*$-homomorphism
\[
\iota\times\iota : C^*(1,\{g_{i,j} \}_{i,j=1}^n)\otimes D\to A
\]
defined on elementary tensors by $(\iota\times\iota)(a\otimes b)=ab$. 
Recall that there is a natural $*$-isomorphism
\[
C(X)\otimes D\cong C(X,D)
\]
which maps $f\otimes e_{i,j}$ to the function $x\mapsto f(x)e_{i,j}$ for all $f\in C(X)$ and $i,j=1,\ldots,n$. In particular, the element $\sum_{i,j=1}^{n}g_{i,j}\otimes e_{i,j}$ is mapped to the function $x\stackrel{h}{\mapsto} \sum_{i,j=1}^{n}g_{i,j}(x)e_{i,j}$. By our hypotheses, 
both $h^*h$ and $hh^*$ are equal to the constant function $\sum_{i=1}^{n}e_{i,i}$. It follows that the element $b := \sum_{i,j=1}^{n}g_{i,j}\otimes e_{i,j}$ satisfies 
$b^* b = b b^* = 1\otimes (\sum_{i=1}^{n}e_{i,i} )$, 
and applying the $*$-homomorphism $\iota\times \iota$ to this equation yields
$a^* a = aa^* = \sum_{i=1}^{n}e_{i,i}$,
as desired.
\end{proof}

\begin{theorem}\label{T-SR1}
Let $G\curvearrowright X$ be a topologically free minimal action with $M_G (X) \neq\emptyset$, 
and suppose that it is squarely divisible. 
Then the reduced crossed product $C(X)\rtimes_\lambda G$ has stable rank one.
\end{theorem}

\begin{proof}
By topological freeness and minimality, $C(X)\rtimes_\lambda G$ is simple \cite{ArcSpi94}.
It is also stably finite, and in particular finite, since any invariant Borel probability measure gives rise to a faithful tracial state
via composition with the canonical conditional expectation onto $C(X)$, with faithfulness being a
consequence of simplicity.

Let $a$ be a non-invertible element of $C(X)\rtimes_\lambda G$. 
Following the strategy of R{\o}rdam \cite{Ror91} as described in the beginning of this section, we will first apply
perturbation and unitary rotation to $a$ so as to produce a two-sided zero divisor of a suitable type
and then perform a series of further perturbations and unitary rotations in order to produce a nilpotent element. 
Since a nilpotent element can be approximated arbitrarily well by invertible elements through the addition of 
nonzero scalar multiples of the unit, this will establish stable rank one.

In the general setting of simple unital finite C$^*$-algebras, R{\o}rdam shows that a non-invertible element $b$
can be perturbed and unitarily rotated to an element $b'$ satisfying $b'd = db' = 0$ for some 
nonzero positive element $d$, which is the strengthened form of zero division that is enables one,
under favourable circumstances, to unitarily rotate to a nilpotent element \cite[Proposition~3.2]{Ror91}. 
In our case we need a refinement of this principle for reduced crossed products that holds
under the conditions of finiteness and simplicity that we observed in the first paragraph and allows
us to assume, through perturbation and unitary rotation,
that $a\unit_O =\unit_O a=0$ for some nonempty open set $O\subseteq X$ \cite[Proposition~6.2]{LiNiu20}. 
We note that Proposition~6.2 of \cite{LiNiu20} assumes freeness but the proof still
works if one relaxes this to topological freeness, which is what we need here.

With the set $O$ now at hand, by the definition of square divisibility we can find an open set $O_0\subseteq O$ 
and nonempty open sets $O_1,O_2\subseteq X$ with $\overline{O_1}\cap \overline{O_2}=\emptyset$, $O_0\cap (\overline{O_1}\sqcup\overline{O_2})=\emptyset$, and $\overline{O_1}\sqcup \overline{O_2}\prec O_0$ such that the action is $(O_1,O_2,E)$-squarely divisible for all finite sets 
$e\in E\subseteq G$.

Set $Q=\overline{O_1}\sqcup \overline{O_2}$ for brevity. Since $Q\cap O_0=\emptyset$ we may apply Lemma~\ref{L-unitary} to find a unitary $r\in C(X)\rtimes_\lambda G$, a closed set $O_0^-\subseteq O_0$, and an open set $Q^+\supseteq Q$
such that $r1_{Q^+}=1_{O_0^-}r1_{Q^+}$ and $1_{Q^+}r=1_{Q^+}r1_{O_0^-}$.

Take a $t\in C(X,[0,1])$ such that $t(x)=0$ for all $x\in Q$ and $t(x)=1$ for all $x\in X\setminus Q^+$. Then $rar(1-t)=rar1_{Q^+}(1-t)=ra1_{O_0^-}r1_{Q^+}(1-t)=0$ and, similarly, $(1-t)rar=0$. We may therefore replace $a$ with $rar$ and assume that $a(1-t)=(1-t)a=0$.

Now if $a'$ is any element in the algebraic crossed product $C(X)\rtimes_{\alg}G$ then
the element $a'' := ta't\in C(X)\rtimes_{\mathrm{alg}}G$ satisfies $a''1_Q=1_Qa''=0$ and
\begin{align*}
\|a''-a\|=\|t(a'-a)t\|\leq \|a'-a\| .
\end{align*}
Since we may choose such an $a'$ so that $\|a'-a\|$ is as small as we wish,
by replacing $a$ once again, this time with $a''$, we may assume that it satisfies $a1_Q=1_Qa=0$ and has the form $a=\sum_{s\in E} f_su_s$ for some finite symmetric set $e\in E\subseteq G$. We now aim to show that $a$ can be multiplied on the left and right by suitable unitaries
in $C(X)\rtimes_\lambda G$ so as to obtain a nilpotent element.

By our choice of $O_1$ and $O_2$ as given by the definition of square divisibility, there exist an $n\in\Nb$, a collection $\{ V_{i,j} \}_{i,j=1}^n$ 
of pairwise equivalent and pairwise $E$-disjoint open subsets of $X$, and, writing $V = \bigsqcup_{i,j=1}^n V_{i,j}$,
an open set $U\subseteq X$ with $\partial V \subseteq U$ such that,
defining $V_1 = \bigsqcup_{i=1}^n V_{i,1}$, $R = V^\comp$,
and $B = \overline{V}\cap ((V\cap \overline{U}^\comp )^E )^\comp$, we have
\begin{enumerate}
\item $\overline{V_{i,1}} \prec O_1 \cap \bigsqcup_{j=2}^n V_{i,j} \cap B^\comp$ for every $i=1,\dots ,n$,

\item $R\prec O_2 \cap V \cap (V_1 \cup B)^\comp$, and

\item $B\cup (\overline{U}\cap R)\prec O_2 \cap \overline{V\cup U}^\comp$.
\end{enumerate}
Note that in each of the subequivalences in (i), (ii), and (iii) the set on the left hand side is closed and disjoint from the open set on the right hand side.

Using a separation argument and the fact that the sets $\overline{V_{i,j}}$ are pairwise disjoint, we can find mutually disjoint open neighbourhoods $\mathcal{V}_i \supseteq \overline{V_{i,1}}$ for $i=1,\dots , n$
such that for each $i$ we have $\mathcal{V}_i \subseteq V_{i,1}\cup U$ and $\mathcal{V}_i\cap (\overline{V}\setminus \overline{V_{i,1}})=\emptyset$, that is, $\mathcal{V}_i\subseteq (\overline{V}^c\cap U)\sqcup \overline{V_{i,1}}$. Applying Lemma~\ref{L-unitary} with respect to these neighbourhoods for the subequivalences in (i), for each $i=1,\ldots,n$ we obtain a unitary $u_i\in C(X)\rtimes_\lambda  G$ and a closed set $Y_{i,1}\subseteq O_1 \cap \bigsqcup_{j=2}^n V_{i,j} \cap B^\comp$ such that $u_i1_{V_{i,1}\cup D}=1_{Y_{i,1}\sqcup D}u_i1_{V_{i,1}\cup D}$ whenever $D\subseteq X$ is a Borel set satisfying $D\cap Y_{i,1}=\emptyset$. Moreover, $1_{(\overline{V}^c\cap U)\sqcup \overline{V_{i,1}}\sqcup Y_{i,1}}$ acts like a unit on $u_i-1$, and $1_Cu_i=u_i1_C=1_C$ for Borel sets $C\subseteq (\overline{V}\cup U^c)\cap \overline{V_{i,1}}^c \cap Y_{i,1}^c$.
By the choice of the open neighbourhoods $\mathcal{V}_i \supseteq\overline{V_{i,1}}$ and Remark~\ref{rem:commutation_relation}, we have $u_iu_j=u_ju_i$, for all $1\leq i,j\leq n$. 
Writing $Y_1=\bigsqcup_{i=1}^{n} Y_{i,1}$ and $u=u_1u_2\cdots u_n$, we obtain:
\begin{itemize}
\item[(u.1)] $u1_{V_1\cup D}=1_{Y_1\cup D}u1_{V_1\cup D}$ whenever $D\subseteq X$ is a Borel subset such that $D\cap Y_1=\emptyset$,
\item[(u.2)] $1_{(\overline{V}^c\cap U)\sqcup \overline{V_1}\sqcup Y_1}$ acts like a unit on $u-1$, and
\item[(u.3)] $u1_C=1_C=1_Cu$ where $C= (\overline{V}\cup U^c)\cap \overline{V_1}^c\cap Y_1^c$.
\end{itemize}

In the case of the subequivalence (ii), we apply Lemma~\ref{L-unitary} with respect to some open neighbourhood $R_0$ of $R$ that is contained in both $R\cup U$ and the complement of the closed set $Y_1$, which is disjoint from $R$. We thereby obtain a closed subset $Y_2\subseteq O_2\cap V\cap (V_1\cup B)^c$, an open set $R^+\subseteq X$ satisfying $R\subseteq R^+\subseteq R_0$, and a unitary element $v\in C(X)\rtimes_\lambda G$ such that 
\begin{itemize}
\item[(v.1)] $v1_{{R^+}\cup D}=1_{Y_2\sqcup D}v1_{R^+\cup D}$, whenever $D\subseteq X$ is a Borel set such that $D\cap Y_2=\emptyset$,
\item[(v.2)] $1_{R_0\cup Y_2}$ acts like a unit on $(v-1)$, and
\item[(v.3)] $1_Cv=1_C=v1_C$ where $C=(R_0\cup Y_2)^c$.
\end{itemize}

In the case of the subequivalence (iii), we apply Lemma~\ref{L-unitary} with respect to some open neighbourhood $B_0$ of $B\cup (\overline{U}\cap R)$ that is disjoint from $(Y_1\cup Y_2)$. 
We thereby obtain an open set $B^+$ with $B\cup (\overline{U}\cap R)\subseteq B^+\subseteq B_0$, a closed set $Y_3\subseteq O_2\cap \overline{V\cup U}^c$, and a unitary $w$ such that
\begin{itemize}
\item[(w.1)] $w1_{B^+\cup D}=1_{Y_3\sqcup D}w1_{B^+\cup D}$
and  $1_{B^+\cup D}w=1_{B^+\cup D}w1_{Y_3\sqcup D}$ (since the first equation also holds for $w^*$) whenever $D\subseteq X$ is a Borel set such that $D\cap Y_3=\emptyset$,
\item[(w.2)] $1_{B_0\cup Y_3}$ acts like a unit on $(w-1)$, and
\item[(w.3)] $1_Cw=1_C=w1_C$ where $C=(B_0\cup Y_3)^c$.
\end{itemize}

Define $b = u^* wauv$. 
By partitioning the unit in $B(X)\rtimes_\lambda G$ as $\unit_R + \sum_{i,j=1}^n \unit_{V_{i,j}}$ 
and multiplying $b$ on either side by the projections in this sum we can represent
$b$ as an $(n^2+1)\times (n^2+1)$ ``matrix''.
We will verify that, with respect to this matrix decomposition, $b$ takes the following form 
(assuming $n=3$ for the purpose of illustration), where the first row and column correspond to the set
$R$ and the remaining rows and columns correspond to the sets $V_{i,j}$ as ordered 
lexicographically with respect to the pairs $i,j$:
\begin{gather*}
\mbox{\small 
$\left[
\begin{array}{c | c c c | c c c | c c c}
0 & 0 & * & * & 0 & * & * & 0 & * & * \\
\hline
0 & 0 & 0 & 0 & 0 & 0 & 0 & 0 & 0 & 0 \\
0 & 0 & * & * & 0 & 0 & 0 & 0 & 0 & 0 \\
0 & 0 & * & * & 0 & 0 & 0 & 0 & 0 & 0 \\
\hline
0 & 0 & 0 & 0 & 0 & 0 & 0 & 0 & 0 & 0 \\
0 & 0 & 0 & 0 & 0 & * & * & 0 & 0 & 0 \\
0 & 0 & 0 & 0 & 0 & * & * & 0 & 0 & 0 \\
\hline
0 & 0 & 0 & 0 & 0 & 0 & 0 & 0 & 0 & 0 \\
0 & 0 & 0 & 0 & 0 & 0 & 0 & 0 & * & * \\
0 & 0 & 0 & 0 & 0 & 0 & 0 & 0 & * & * 
\end{array}
\right]$} .
\end{gather*}
A matrix over $\Cb$ of this form can be multiplied on the left and right by permutation
matrices with $1$ in the upper left entry so as to produce a strictly upper triangular (and hence nilpotent) matrix, and indeed
the proof will ultimately be completed by finding proxies
$z_1$ and $z_2$ for these permutation matrices in the unitary group of $C(X)\rtimes_\lambda G$ so that 
$z_1 b z_2$ has a similar strictly upper triangular matrix representation and hence is nilpotent.
Note that if all of the sets in the definition of square divisibility and in the subequivalences and equivalences at play therein
were clopen, then we could directly interpret these permutation matrices as elements of $C(X)\rtimes_\lambda G$
(using the equivalences between the sets $V_{i,j}$, via Remark~\ref{rem:zero-dim-unitary},
to define off-diagonal matrix units in the lower right $n^2 \times n^2$ block, as we will do in the proof
of Theorem~\ref{T-SR1 weak SD}) and thereby avoid
the homotopy constructions that will occupy a large part of the reminder of the proof.

These homotopies, which will be used to define the unitaries $z_1$ and $z_2$ that rotate $b$ to a nilpotent element, 
will need to be ``invisible'' to $b$, and so what we need in fact is a slightly more refined representation
of $b$ that takes into account the left zero division of $b$ across the boundary of $V$,
namely
\begin{gather}\label{E-b}
b=1_Rb1_V+\sum_{\substack{1 \le i,j,j' \le n \\ j,j'\neq 1 }}1_{V_{i,j}\setminus R^+}b1_{V_{i,j'}\setminus R^+} ,
\end{gather}
the point being that the slightly smaller projection $1_{V_{i,j}\setminus R^+}$ appears here instead of $\unit_{V_{i,j}}$.
To verify (\ref{E-b}), we begin by noting, using that $a1_{Y_2}=0$, that
\begin{align*}
auv\unit_{R^+} \stackrel{\text{(v.1)}}{=} au\unit_{Y_2}v\unit_{R^+}\stackrel{\text{(u.3)}}{=} a\unit_{Y_2}v\unit_{R^+}=0 ,
\end{align*}
which shows that $b\unit_{R^+} = 0$. This in particular yields the zero leftmost column in the matrix representation of $b$,
for which we just need $b\unit_R = 0$. 
Using the fact that $V_1\cap Y_2=Y_2\cap Y_1 =\emptyset$ and that $a1_{Y_1\sqcup Y_2}=0$, we also have
\[
auv \unit_{V_1}\stackrel{\text{(v.1)}}{=}au \unit_{V_1\sqcup Y_2} v \unit_{V_1}\stackrel{\text{(u.1)}}{=}a\unit_{Y_1\sqcup Y_2} u \unit_{V_1\sqcup Y_2} v \unit_{V_1}=0,
\]
which accounts for the other $n$ columns of zeros.
Moreover, since $Y_1\subseteq (B_0\cup Y_3)^c$ and $\unit_{Y_1} a = 0$ we see that
\[
\unit_{V_1} u^* wa \stackrel{\text{(u.1)}}{=}\unit_{V_1} u^* \unit_{Y_1} wa \stackrel{\text{(w.3)}}{=} \unit_{V_1} u^* \unit_{Y_1} a  =  0
\]
and hence $\unit_{V_1} b = 0$, which accounts for the $n$ rows of zeros in the lower right $n^2 \times n^2$ submatrix.

Now let $1\leq i,i' \leq n$ with $i\neq i'$ and let $2\leq j,j'\leq n$. Writing $p=(i,j)$ and $q=(i',j')$ with $V_p$ and $V_q$ interpreted accordingly as $V_{i,j}$ and $V_{i',j'}$, let us check that $\unit_{V_p} b \unit_{V_q} = 0$, which will in particular account for the remaining zeros in the
matrix representation of $b$. We do this by verifying the equalities
(\ref{E-equality one}),  (\ref{E-equality two}), (\ref{E-equality three}), and (\ref{E-equality four})
and then applying these in a chain-like way to obtain (\ref{E-combine}).

First observe, since $\unit_{R_0 \cup Y_2\cup V_q}$ acts like a unit on $(v-1)$ by (v.2), that
\begin{align}\label{E-equality one}
v\unit_{V_q} = \unit_{R_0 \cup Y_2\cup V_q} v\unit_{V_q} .
\end{align}

Next, setting $A = R_0\cup Y_2\cup V_q\cup V_{i',1}$ for brevity,
we note that the fact that $1_{(\overline{V}^c\cap U)\sqcup \overline{V_{i',1}}\sqcup Y_{i',1}}$ acts like a unit on $(u_{i'}-1)$ means that so does $1_{A\cup Y_{i',1}}$ due to the inclusion $(\overline{V}^c\cap U)\sqcup \overline{V_{i',1}}\sqcup Y_{i',1}\subseteq A\cup Y_{i',1}$, and thus
\begin{align*}
u_{i'}\unit_{R_0\cup Y_2\cup V_q} = \unit_{A\cup Y_{i',1}} u_{i'}\unit_{R_0\cup Y_2\cup V_q}.
\end{align*}
Given that the sets $Y_{k,1}$ are pairwise disjoint and $Y_{k,1}\cap A=\emptyset $ for all $k\neq i'$ (since $Y_{k,1}\subseteq (V_{k,2}\cup\cdots\cup V_{k,n})\cap B^c$), we therefore obtain
\begin{align*}
u\unit_{R_0\cup Y_2\cup V_q}&=u_1\cdots u_{i'-1} u_{i'+1}\cdots u_{n-1} u_n  \unit_{A\cup Y_{i',1}} u_{i'}\unit_{R_0\cup Y_2\cup V_q}\\
&=u_1\cdots u_{i'-1} u_{i'+1}\cdots u_{n-1} \unit_{A\cup Y_{i',1}\cup Y_{n,1}} u_n  \unit_{A\cup Y_{i',1}} u_{i'}\unit_{R_0\cup Y_2\cup V_q}.
\end{align*}
Proceeding recursively in this way, we conclude that 
\begin{align}\label{E-equality two}
u \unit_{R_0\cup Y_2\cup V_q}=\unit_{A\cup Y_1}u \unit_{R_0\cup Y_2\cup V_q}.
\end{align}

Next we reexpress $a\unit_{A\cup Y_1}$. Note first that since the sets $\{V_{i,j}\}_{i,j=1}^{n}$ are $E$-disjoint, we clearly have $s(V_q\cup V_{i',1})\subseteq R\cup V_q\cup V_{i',1}$ for all $s\in E$. Moreover, $(R\cup B)^c=(V\cap \overline{U}^c)^E\subseteq s(V\cap U^c)$ for all $s\in E$, using that $E=E^{-1}$. In other words, $s(R\cup U)\subseteq R\cup B$ for all $s\in E$. Putting these facts together, and recalling that $R_0\subseteq R\cup U$, we have 
\begin{align*}
s(R_0\cup V_q\cup V_{i',1})\subseteq R\cup B\cup V_q\cup V_{i',1}
\end{align*} 
for all $s\in E$. Since $a1_{Y_i}=1_{Y_i}a=0$ for $i=1,2,3$, it follows that
\begin{align}\label{E-equality three}
a\unit_{A\cup Y_1}= a\unit_{R_0 \cup V_q\cup V_{i',1}}
&= \sum_{s\in E}f_su_s\unit_{R_0 \cup V_q\cup V_{i',1}}\\
&=\sum_{s\in E}\unit_{s(R_0\cup V_q\cup V_{i',1})}f_su_s\notag\\
&= \unit_{R\cup B\cup V_q\cup V_{i',1}}a\unit_{R_0 \cup V_q\cup V_{i',1}}\notag\\
&= \unit_{(R\cup B\cup V_q\cup V_{i',1})\setminus Y_3}a\unit_{A\cup Y_1}.\notag
\end{align}

Finally, we reexpress $u^*\unit_{(R\setminus U)\sqcup V_q\sqcup V_{i',1}}$.
Using (w.1) along with the fact that $(U\cap R)\cup B\subseteq B^+$, and $Y_3\subseteq R\setminus U$, we compute that
\begin{align*}
w \unit_{(R\cup B\cup V_q\cup V_{i',1})\setminus Y_3}&= w \unit_{((R\setminus U) \cup (U\cap R) \cup B\cup V_q\cup V_{i',1})\setminus Y_3}\\
&=\unit_{(R\setminus U)\sqcup V_q\sqcup V_{i',1}}w \unit_{((R\setminus U) \cup (U\cap R) \cup B\cup V_q\cup V_{i',1})\setminus Y_3} .
\end{align*}
Note that for every $k\neq i'$, we have $u_k^* \unit_{V_q\sqcup V_{i',1}} = \unit_{V_q\sqcup V_{i',1}}$. Also, $\unit_{(\overline{V}^c\cap U)\cup \overline{V_{i',1}}\cup Y_{i',1}}$ acts as a unit on $(u_{i'}-1)$, and so does $\unit_{(R\cap U)\cup \bigsqcup_{j=1}^{n}\overline{V_{i',j}}}$ since $(\overline{V}^c\cap U)\cup \overline{V_{i',1}}\cup Y_{i',1}\subseteq (R\cap U)\cup \bigsqcup_{j=1}^{n}\overline{V_{i',j}}$. It follows that
\begin{align*}
u^*\unit_{V_q\cup V_{i',1}}=u_{i'}^*\unit_{V_q\cup V_{i',1}}=\unit_{(R\cap U)\cup \bigsqcup_{j=1}^{n}\overline{V_{i',j}}}u_{i'}^*\unit_{V_q\cup V_{i',1}}.
\end{align*}
On the other hand, since $R\setminus U \subseteq (\overline{V}\cup U^c)\cap \overline{V_1}^c\cap Y_1^c$ condition (u.3) implies that $u^* 1_{R\setminus U} = 1_{R\setminus U}$. In conjunction with the display above, this gives us
\begin{align}\label{E-equality four}
u^*\unit_{(R\setminus U)\sqcup V_q\sqcup V_{i',1}}=\unit_{R\cup \bigsqcup_{j=1}^{n}\overline{V_{i',j}}} u^*\unit_{(R\setminus U)\sqcup V_q\sqcup V_{i',1}}.
\end{align}

We can now apply (\ref{E-equality one}),  (\ref{E-equality two}), (\ref{E-equality three}), and (\ref{E-equality four}) in sequence to obtain
\begin{align}\label{E-combine}
b1_{V_q}&=u^* wauv \unit_{V_q} \\
&=  \unit_{R\cup \bigsqcup_{j=1}^{n}\overline{V_{i',j}}} u^* \unit_{(R\setminus U)\sqcup V_q\sqcup V_{i',1}} w \unit_{(R\cup B\cup V_q\cup V_{i',1})\setminus Y_3} a \unit_{A\cup Y_1} u \unit_{R_0 \cup Y_2\cup V_q} v \unit_{V_q}. \notag
\end{align}
As $V_p\cap (R\cup \bigsqcup_{j=1}^{n}\overline{V_{i',j}})=\emptyset$, we conclude that $1_{V_p}b1_{V_q}=0$.
We have thus verified that the matrix representation of $b$ has the desired form, i.e., that (\ref{E-b}) holds with
$\unit_{V_{i,j}}$ in place of $\unit_{V_{i,j}\setminus R^+}$.

To obtain (\ref{E-b}) itself, it remains now to check, given $1\leq i,j \leq n$ and writing $p = (i,j)$,
that $1_{R^+\cap V_p}b=b 1_{R^+\cap V_p}=0$. The second equality follows from $b 1_{R^+}=0$, 
which we observed at the outset. To verify that $1_{R^+\cap V_p}b = 0$, we will compute that 
$1_{R^+\cap V_p}u^*wa=0$. We saw earlier that 
$\unit_{V_1} b = 0$, and so we may assume $j\neq 1$.
Since $R^+\cap V_p\subseteq (R\cup U)\cap V= U\cap V\subseteq B$ and hence $R^+\cap V_p\subseteq \overline{V}\cap \overline{V_1}^c\cap Y_1^c$, from (u.3) we obtain 
$u 1_{R^+\cap V_p} = 1_{R^+\cap V_p}$ and hence
$1_{R^+\cap V_p}u^*=1_{R^+\cap V_p}$.
By (w.1) we have $1_{B}w=1_{B}w1_{Y_3}$ and thus, since $R^+\cap V_p\subseteq B$ as just observed, 
\begin{align*}
1_{R^+\cap V_p}w=1_{R^+\cap V_p}w1_{Y_3}.
\end{align*}
It follows that
\begin{align*}
1_{R^+\cap V_p}u^*wa=1_{R^+\cap V_p}wa=1_{R^+\cap V_p}w1_{Y_3}a=0,
\end{align*}
as desired.

To set up the homotopies involved in the construction of $z_1$ and $z_2$,
we first define some matrix units $e_{p,q}$ in $B(X)\rtimes_\lambda G$ and bump functions $f$ and $h$ 
in $C(X)\rtimes_\lambda G$.
Set $\Omega = \{ 1,\dots , n\}^2$.
By the equivalence of the sets $V_{i,j}$, which as above we also write as $V_p$ where $p=(i,j)\in \Omega$, 
we can find a finite partition $\{ C_k \}_{k\in K}$ of $V_{1,1}$ which is relatively clopen in $V_{1,1}$
and tuples $s_k = (s_{k,p} )_{p\in \Omega}$ of elements of $G$ for $k\in K$ such that 
$V_p = \bigsqcup_{k\in K} s_{k,p} C_k$ for all $p\in \Omega$, with $s_{k, (1,1)}$ equal to the identity of $G$ for all $k\in K$,
and the closures in $X$ of the sets $s_{k,p}C_k$ for $p\in\Omega$ and $k\in K$ are pairwise disjoint.
We can then find open neighbourhoods $W_{k,p}\supseteq s_{k,p}\overline{C_k}$ with pairwise disjoint closures. For each $k\in K$ define the open set $\tilde{C}_k =\bigcap_{p\in\Omega}s_{k,p}^{-1}W_{k,p}\supseteq \overline{C_k}$. Then the sets $s_{k,p}\tilde{C}_k$ for $k\in K$ and $p\in \Omega$ have pairwise disjoint closures in $X$. Define $\tilde{V}_p = \bigsqcup_{k\in K} s_{k,p}\tilde{C}_k$. Note that this gives a partition of $\tilde{V}_p$ which is clopen in the relative topology on $\tilde{V}_p$. Moreover, the sets $\tilde{V}_p$ for $p\in \Omega$ have pairwise disjoint closures in $X$. In particular, by Lemma~\ref{lemma:characterCts} it follows that
\begin{itemize}
\item[($\bullet$)]if $q\in C(X)$ is any function that vanishes off of $\tilde{V}_{1,1}$ then $q1_{\tilde{C}_k}\in C(X)$ for all $k\in K$.
\end{itemize}

For $p\in \Omega$ define $d_p= \sum_{k\in K} u_{s_{k,p}}1_{\tilde{C}_k}$. Then $d_p^*d_p=1_{\tilde{V}_{1,1}}$ and $d_pd_p^*=1_{\tilde{V}_p}$, so that $d_p$ is a partial isometry in $B(X)\rtimes_\lambda G$. Observe also that $d_p^*d_q=0$ whenever $p\neq q$.

For $p,q\in\Omega$ define $e_{p,q}=d_pd_q^*\in B(X)\rtimes_\lambda G$. Note that $e_{p,q}^*=e_{q,p}$ and $\sum_{p\in \Omega} e_{p,p}=\sum_{p\in\Omega} 1_{\tilde{V}_p}=1_{\tilde{V}}$ where $\tilde{V}=\bigsqcup_{p\in\Omega} \tilde{V}_p$. Moreover, using Kronecker delta notation,
\begin{align*}
e_{p,q}e_{r,s}
= d_pd_q^*d_rd_s^*
=\delta_{q,r}d_pd_r^*d_rd_s^*
=\delta_{q,r}d_p1_{\tilde{V}_{1,1}}d_s^*
=\delta_{q,r}d_pd_s^*
=\delta_{q,r}e_{p,s}.
\end{align*}
We conclude that the elements $e_{p,q}$ for $p,q\in\Omega$ are partial isometries in 
$B(X)\rtimes_\lambda G$ forming a set of matrix units.

Next we use the partial isometries $d_p$ to build a bump function $h$ which will be equal to $1$ on $V$ 
and $0$ off of $\tilde{V}$ and a second bump function $f$ which will be equal to $1$ on $V\setminus R^+$ and $0$
off of $V$. 
To define $z_1$ and $z_2$, we will perform homotopies over the second of these two sets using $f$ 
and then a patching operation over the first one using $h$ to achieve global unitarity.

To construct $h$,
choose an $h_1\in C(X, [0,1])$ satisfying $h_1|_{V_{1,1}}=1$ and $h_1|_{X\setminus \tilde{V}_{1,1}}=0$, which is possible since $\overline{V_{1,1}}\subseteq \tilde{V}_{1,1}$. We claim that $\{d_ph_1d_p^*\}_{p\in\Omega}$ is a collection of pairwise orthogonal functions in $C(X)$. Indeed
\begin{align*}
d_ph_1d_p^*=\sum_{k,j\in K}u_{s_{k,p}}1_{\tilde{C}_k}h_11_{\tilde{C}_j}u_{s_{j,p}}^*=\sum_{k\in K}u_{s_{k,p}}1_{\tilde{C}_k}h_1u_{s_{k,p}}^*=\sum_{k\in K}1_{s_{k,p}\tilde{C}_k}(s_{k,p}h_1 ),
\end{align*}
and the last sum, which is clearly supported on $\tilde{V}_p$, 
belongs to $C(X)$ seeing that $1_{\tilde{C}_k}h_1\in C(X)$ by ($\bullet$). 
Set $h=\sum_{p\in \Omega} d_ph_1d_p^*\in C(X)$, which is a positive contraction supported on $\tilde{V}$. For all $p,q\in \Omega$ we have
\begin{align*}
he_{p,q}=\sum_{r\in\Omega} d_rh_1d_r^* d_pd_q^*
&=d_ph_1d_p^*d_pd_q^* \\
&=d_ph_1d_q^* \\
&= d_pd_q^*d_qh_1d_q^* \\
&= d_pd_q^* \cdot\sum_{r\in\Omega}d_rh_1d_r^* =  e_{p,q}h ,
\end{align*}
and also
\begin{align*}
he_{p,q}=d_ph_1d_q^*=\sum_{k,j\in K} u_{s_{k,p}}1_{\tilde{C}_k}h_11_{\tilde{C}_j}u_{s_{j,q}}^*=\sum_{k\in K} u_{s_{k,p}}1_{\tilde{C}_k}h_1u_{s_{k,q}}^* ,
\end{align*}
which by ($\bullet$) shows that $he_{p,q}$ belongs to $C(X)\rtimes_\lambda G$.
Observe moreover that, for every $x\in V$, taking $p\in \Omega$ and $k\in K$
such that $x\in s_{k,p}C_k$ and using the fact that
$s_{k,p}^{-1}(x)\in C_k\subseteq V_{1,1}$ and $h_1|_{V_{1,1}}=1$, we have 
\[h(x)=(d_ph_1d_p^*)(x)= (s_{k,p} h_1)(x)=1, \]
that is, $h|_V = 1$.

Now we construct $f$.
Let $p\in \Omega$ and $k\in K$. Seeing that $\partial (s_{k,p}C_k)\subseteq \partial V_p\subseteq R^+$, we have
\begin{align*}
\overline{C_k\setminus s_{k,p}^{-1}R^+}
&\subseteq \overline{C_k}\setminus s_{k,p}^{-1}R^+\\
&=s_{k,p}^{-1}(s_{k,p}\overline{C_k}\setminus R^+)\\
&\subseteq s_{k,p}^{-1}(s_{k,p}\overline{C_k}\setminus \partial (s_{k,p}C_k))
=\mathrm{int}(C_k)\subseteq C_k\subseteq V_{1,1}.  
\end{align*}
Thus we can find a function $f_1\in C(X,[0,1])$ such that $f_1|_{C_k\setminus s_{k,p}^{-1}R^+}=1$ for all $p\in\Omega$ and $k\in K$ and $f_1|_{X\setminus V_{1,1}}=0$. 
An argument essentially identical to the one given above for $h_1$ shows that
$\{d_pf_1d_p^*\}_{p\in \Omega}$ is a set of pairwise orthogonal functions in $C(X)$. 
Moreover, since $f_1$ is supported on $V_{1,1}$ the function
$d_pf_1d_p^*=\sum_{k\in K} 1_{s_{k,p}C_k}(s_{k,p} f_1)$ is supported on $V_p$, so that $f=\sum_{p\in\Omega} d_pf_1d_p^*$ is a function in $C(X)$ which is supported on $V$, i.e., $f|_R = 0$. Exactly as was done above for $h$, one can show that for all $p,q\in\Omega$ the function $f$ commutes with every $e_{p,q}$ and the element $fe_{p,q}$ belongs to $C(X)\rtimes_\lambda G$. Observe finally that, for every $x\in V\setminus R^+$, taking $p\in \Omega$ and $k\in K$ such that $x\in s_{k,p}C_k\setminus R^+$ and using that
$s_{k,p}^{-1}(x)\in C_k\setminus s_{k,p}^{-1}R^+$ we have
\[f(x)=(d_pf_1d_p^*)(x)= (s_{k,p} f_1)(x)=1, \] 
that is, $f|_{V\setminus R^+} = 1$.

We are ready to construct $z_1$.
Choose some permutation $\kappa_1$ of $\{ 1,\dots , n^2 \}$ that for each $i=1,\dots ,n$ 
shifts the numbers in the interval $\{ (i-1)(n-1) +2 , (i-1)(n-1) +3 , \dots , (i-1)(n-1)+n\}$ by $i-1$ in the positive direction. Write $S$ for the permutation matrix in $M_{n^2}$ corresponding to $\kappa_1^{-1}$, i.e., the matrix whose $\ell$-th row for a given $1\leq \ell\leq n^2$ is the same as the $\kappa_1(\ell)$-th row of the identity matrix.
We could interpret the $(1 + n^2 )\times (1 + n^2 )$ block diagonal matrix $\diag (1,S)$ 
as an element of $B(X)\rtimes_\lambda G$ using the family of 
Murray--von Neumann equivalences between $\unit_{V_{1,1}}$ and the other indicator functions $\unit_{V_{i,j}}$
induced from our equivalences between the sets $V_{i,j}$. But this element would not 
lie in $C(X)\rtimes_\lambda G$ unless all of the sets at play were clopen, and so it is not directly of use.
The unitary $z_1$ will be a proxy for $\diag (1,S)$ inside the C$^*$-algebra $C(X)\rtimes_\lambda G$ 
such that the product $z_1 b$ has the following matrix form (in the illustrative case $n=3$), which is the same form
as if we multiplied $b$ on the left by $\diag (1,S)$ (interpreted as above as an element of $B(X)\rtimes_\lambda G$) instead of $z_1$:
\begin{gather*}
\mbox{\small 
$\left[
\begin{array}{c | c c c | c c c | c c c}
0 & 0 & * & * & 0 & * & * & 0 & * & * \\
\hline
0 & 0 & 0 & 0 & 0 & 0 & 0 & 0 & 0 & 0 \\
0 & 0 & * & * & 0 & 0 & 0 & 0 & 0 & 0 \\
0 & 0 & * & * & 0 & 0 & 0 & 0 & 0 & 0 \\
\hline
0 & 0 & 0 & 0 & 0 & * & * & 0 & 0 & 0 \\
0 & 0 & 0 & 0 & 0 & * & * & 0 & 0 & 0 \\
0 & 0 & 0 & 0 & 0 & 0 & 0 & 0 & * & * \\
\hline
0 & 0 & 0 & 0 & 0 & 0 & 0 & 0 & * & * \\
0 & 0 & 0 & 0 & 0 & 0 & 0 & 0 & 0 & 0 \\
0 & 0 & 0 & 0 & 0 & 0 & 0 & 0 & 0 & 0 
\end{array}
\right]$} .
\end{gather*}

Since the unitary group $\sU (M_{n^2})$ is path-connected, there is a continuous map $W : [0,1]\to \sU(M_{n^2})$ such that $W(0)=1_{M_{n^2}}$ and $W(1)=S$. 
Let $\beta  : \Omega \to \{ 1,\dots , n^2 \}$ be the bijection $(i,j) \mapsto (i-1)n+j$ giving
the lexicographic ordering of $\Omega$. Define $W^\dagger (t) = \sum_{p,q\in\Omega}W(t)_{\beta(p),\beta(q)}e_{p,q}$, for $t\in [0,1]$. Given that $W^\dagger$ is $W$ composed with the isomorphism $M_{n^2}\cong C^*(\{e_{p,q} \}_{p,q\in\Omega})$ that identifies $e_{p,q}$ with the standard matrix unit $e_{\beta(p),\beta(q)}$ in $M_{n^2}$, we have
\[
W^\dagger(t)W^\dagger (t)^*=\sum_{p\in\Omega}e_{p,p}=\unit_{\tilde{V}} = W^\dagger (t)^*W^\dagger (t)
\] 
for all $t\in [0,1]$.

For all $p,q\in \Omega$ define a map $F_{p,q} : [0,1]\to\mathbb{C}$ by $F_{p,q}(t) = W(t)_{\beta(p),\beta(q)}$ (that is, $F_{p,q}(t)$ is the $(\beta(p),\beta(q))$-coordinate of the matrix $W(t)$). Note that $F_{p,q}$ is continuous. Since $f\in C(X)$ is a positive contraction, we may apply the functional calculus to define $g_{p,q} := F_{p,q}(f)\in C(X)$. Since $g_{p,q}$ belongs to $C^*(f,1)$ it commutes with $e_{r,s}$ for all $r,s\in \Omega$, and we have $g_{p,q}(x)=W(f(x))_{\beta(p),\beta(q)}$ for all $x\in X$. 
Define 
\[z_1 = (1-h)+h\cdot \sum_{p,q\in \Omega} g_{p,q} e_{p,q}\in C(X)\rtimes_\lambda G.\]
Since $C(X)\rtimes_\lambda G$ is stably finite, to show that $z_1$ is unitary it suffices to verify that it is an isometry. This we do as follows, using in the third step the equality $(1-h)1_R=(1-h)$ (which follows from $h|_{V}=1$) and an application of Lemma~\ref{lemma:FuncMatrix} (given the unitarity of $\sum_{p,q\in\Omega}g_{p,q}(x)e_{p,q}=W^\dagger(f(x))$ in the C$^*$-subalgebra $C^*(\{e_{p,q} \}_{p,q\in\Omega})$ for all $x\in X$), in the fourth step
that  $g_{p,q}1_{R}=\delta_{p,q}1_R$ (since $f|_R=0$), and in the fifth step
that $h1_{\tilde{V}}=h$ and $(1-h)1_R=(1-h)$:
\begin{align*}
z_1^*z_1&=\Big((1-h)+h\cdot \sum_{p,q\in \Omega} g_{p,q} e_{p,q}\Big)^*\Big((1-h)+h\cdot \sum_{p,q\in \Omega} g_{p,q} e_{p,q}\Big)\\
&=(1-h)^2+(1-h)h\Big(\sum_{p,q\in\Omega}g_{p,q}e_{p,q}+\sum_{p,q\in\Omega}\overline{g_{p,q}}e_{q,p}\Big) \\
&\hspace*{35mm} \ +h^2\Big(\sum_{p,q\in\Omega}g_{p,q}e_{p,q}\Big)^*\Big(\sum_{p,q\in\Omega}g_{p,q}e_{p,q}\Big)\\
&=(1-h)^2+(1-h)h1_R\Big(\sum_{p,q\in\Omega}g_{p,q}e_{p,q}+\sum_{p,q\in\Omega}\overline{g_{p,q}}e_{q,p}\Big)+h^21_{\tilde{V}}\\
&=(1-h)^2+(1-h)h1_{R}\Big(2\cdot\sum_{p\in\Omega}e_{p,p}\Big)+h^2\\
&=(1-h)^2+2(1-h)h+h^2\\
&=1.
\end{align*}

Next we construct $z_2$.
Choose a permutation $\kappa_2$ of $\{ 1,\dots , n^2 \}$ that for each $i=1,\dots ,n$ 
shifts the numbers in the interval $\{ (i-1)n +2 , (i-1)n +3 , \dots , in\}$ by $n-i$ in the positive direction. Write $T$ for the permutation matrix in $M_{n^2}$ corresponding to $\kappa_2$.
The unitary $z_2$ will be a proxy
for $\diag(1,T^*)$ inside $C(X)\rtimes_\lambda G$ such that, in the illustrative case $n=3$, the product $z_1 b z_2$ takes the matrix form
\begin{gather*}
\mbox{\small 
$\left[
\begin{array}{c | c c c | c c c | c c c}
0 & * & * & * & * & * & * & * & * & * \\
\hline
0 & 0 & 0 & 0 & 0 & 0 & 0 & 0 & 0 & 0 \\
0 & 0 & 0 & 0 & * & * & 0 & 0 & 0 & 0 \\
0 & 0 & 0 & 0 & * & * & 0 & 0 & 0 & 0 \\
\hline
0 & 0 & 0 & 0 & 0 & 0 & * & * & 0 & 0 \\
0 & 0 & 0 & 0 & 0 & 0 & * & * & 0 & 0 \\
0 & 0 & 0 & 0 & 0 & 0 & 0 & 0 & * & * \\
\hline
0 & 0 & 0 & 0 & 0 & 0 & 0 & 0 & * & * \\
0 & 0 & 0 & 0 & 0 & 0 & 0 & 0 & 0 & 0 \\
0 & 0 & 0 & 0 & 0 & 0 & 0 & 0 & 0 & 0 
\end{array}
\right]$} .
\end{gather*}
For general $n$ the matrix will have a similar strictly upper triangular form.

Take a continuous map $Z : [0,1]\to \sU (M_{n^2})$ such that $Z(0)=1_{M_{n^2}}$ and $Z(1)=T^*$. As before, for all $p,q\in \Omega$ define $H_{p,q} : [0,1]\to\mathbb{C}$ by $H_{p,q}(t)= Z(t)_{\beta(p),\beta(q)}$ and set $w_{p,q}= H_{p,q}(f)\in C(X)$. Then $w_{p,q}$ commutes with the matrix units $e_{r,s}$ and $w_{p,q}(x)=Z(f(x))_{\beta(p),\beta(q)}$ for all $x\in X$. 
Set
\[
z_2= (1-h)+h\cdot \sum_{p,q\in \Omega} w_{p,q} e_{p,q}\in C(X)\rtimes_\lambda G.
\]
By the same arguments as for $z_1$, the element $z_2$ is a unitary.

Let us now finally verify that the matrix representation of $z_1 b z_2$ is strictly upper diagonal.
We write $z_1 b z_2 = \unit_R z_1 b z_2 + \unit_V z_1 b z_2$ and show separately that these two summands 
take a certain form.

We show that $1_R z_i= 1_R=z_i 1_R$ for $i=1,2$.
Since $\unit_R g_{p,q} = \delta_{p,q} \unit_R$ for all $p,q\in\Omega$ and $h1_{\tilde{V}} = h$, 
we have
\begin{align*}
\unit_R z_1 &= \unit_R (1-h)+h\cdot\sum_{p,q\in\Omega} \unit_R g_{p,q} e_{p,q} \\
&=\unit_R(1-h)+h\cdot\sum_{p\in\Omega} \unit_R e_{p,p}\\
&=(1-h)\unit_R+h1_R \unit_{\tilde{V}}
=1_R .
\end{align*}
Since $z_1$ is a unitary, it follows that $1_Rz_1=1_R=z_11_R$. A similar argument shows that $1_Rz_2=1_R=z_21_R$.
We also observe that $z_2$ commutes with $1_V$ given that for all $p,q\in\Omega$ we have
\begin{align*}
e_{p,q}1_{V}&=d_pd_q^*1_{\tilde{V}_q}1_{V}
=d_pd_q^*1_{V_q}\\
&=\Big(\sum_{k\in K}u_{s_{k,p}}1_{\tilde{C}_k}u_{s_{k,q}}^*\Big)1_{V_q} \\
&=\Big(\sum_{k\in K}u_{s_{k,p}}u_{s_{k,q}}^*1_{s_{k,q}\tilde{C}_k}\Big)1_{V_q}\\
&=\sum_{k\in K}u_{s_{k,p}}u_{s_{k,q}}^*1_{s_{k,q}C_k} \\
&=\sum_{k\in K}1_{s_{k,p}C_k}u_{s_{k,p}}u_{s_{k,q}}^* 
=1_{V_p}e_{p,q}
=1_Ve_{p,q} .
\end{align*}
A similar argument shows that $z_1$ commutes with $1_{V}$.
Since $b\unit_R =  0$ and hence $b = b \unit_V$, we thereby obtain
\begin{align}\label{E-one}
\unit_R z_1 b z_2 = \unit_R b z_2 = \unit_R b \unit_V z_2 = \unit_R b z_2 \unit_V .
\end{align}

Next we compute $\unit_V z_1 b z_2$. 
Given $r=(i,j)\in\Omega$ with $j\neq 1$ we have, 
applying the fact that $g_{p,q}$ and $h$ commute with the matrix units $e_{p,q}$ and using the
equalities $h1_V=1_V$ and $g_{p,q}1_{V\setminus R^+}=S_{\beta(p),\beta(q)}1_{V\setminus R^+}$,
\begin{align}\label{E-z1}
z_1 1_{V_r\setminus R^+}
&=(1-h)1_R 1_{V_r\setminus R^+}+\Big(h\cdot \sum_{p,q\in \Omega} g_{p,q} e_{p,q}\Big)  1_{V_r\setminus R^+}\\
&=\Big(h\cdot \sum_{p,q\in \Omega} g_{p,q} e_{p,q}\Big) 1_{V_r\setminus R^+}\notag\\
&=\Big(\sum_{p,q\in \Omega} S_{\beta(p),\beta(q)} e_{p,q}\Big) 1_{V_r\setminus R^+}\notag\\
&=\Big(\sum_{p,q\in \Omega} S_{\beta(p),\beta(q)} e_{p,q}\Big) 1_{\tilde{V}_r} 1_{V_r\setminus R^+}\notag\\
&=\Big(\sum_{p,q\in \Omega} S_{\beta(p),\beta(q)} e_{p,q}\Big) e_{r,r} 1_{V_r\setminus R^+}\notag\\
&=\Big(\sum_{p\in \Omega} S_{\beta(p),\beta(r)} e_{p,r}\Big) 1_{V_r\setminus R^+}.\notag
\end{align}
Set $\Upsilon = \{ ((i,j),(i,j'))\in\Omega\times\Omega : 1\leq i\leq n,\ 2\leq j,j'\leq n \}$. Combining with (\ref{E-b}), it follows that
\begin{align*}
\unit_V z_1 b =z_1 \unit_V b= \sum_{p\in\Omega ,\hspace*{0.4mm} (q,r)\in\Upsilon}S_{\beta(p),\beta(q)}e_{p,q}1_{V_q\setminus R^+}b1_{V_r\setminus R^+}.
\end{align*}
Now given $r=(i,j')\in\Omega$ with $j'\neq 1$, a computation similar to (\ref{E-z1}) shows that
\begin{align*}
1_{V_r \setminus R^+}z_2=\sum_{t\in\Omega}T^*_{\beta(r),\beta(t)}1_{V_r\setminus R^+}e_{r,t}=\sum_{t\in\Omega}T_{\beta(t),\beta(r)}1_{V_r\setminus R^+}e_{r,t},
\end{align*}
and so combining this with the previous display we obtain
\begin{align*}
\unit_V z_1bz_2
&=\sum_{p,t\in\Omega ,\hspace*{0.4mm} (q,r)\in\Upsilon}S_{\beta(p),\beta(q)}e_{p,q}1_{V_q\setminus R^+}b1_{V_r\setminus R^+}T_{\beta(t),\beta(r)}e_{r,t} .
\end{align*}
Since $T_{\beta(t),\beta(r)}=1$ if $\beta(t)=\kappa_2(\beta(r))$ and is $0$ otherwise,
and $S_{\beta(p),\beta(q)}=1$ if $\beta(p)=\kappa_1^{-1}(\beta(q))$ and is $0$ otherwise,
this formula can be rewritten as
\begin{align}\label{E-two}
\unit_V z_1bz_2
= \sum_{(q,r)\in\Upsilon}e_{(\beta^{-1}\circ \kappa_1^{-1}\circ \beta )(q),q}1_{V_q\setminus R^+}b1_{V_r\setminus R^+}e_{r,(\beta^{-1}\circ \kappa_2\circ\beta )(r)} . 
\end{align}

Formulas (\ref{E-one}) and (\ref{E-two}), together with the fact that $1_{V_p}e_{p,q}=e_{p,q}1_{V_q}$ for all $p,q\in\Omega$,
show us that we can write
\[
z_1bz_2
= 1_Rw1_V+\sum_{(q,r)\in\Upsilon}1_{V_{(\beta^{-1}\circ \kappa_1^{-1}\circ \beta )(q)}}y_{q,r}1_{V_{(\beta^{-1}\circ \kappa_2\circ \beta )(r)}} 
\]
for some elements $w,y_{q,r}\in B(X)\rtimes_\lambda G$. In the matrix representation of $z_1 b z_2$, 
the term $1_Rw1_V$ accounts for the top row of possibly nonzero terms 
excluding the top left diagonal entry (which is zero), while the sum over $\Upsilon$ produces a $n^2 \times n^2$ submatrix
which is strictly upper diagonal seeing that 
$\kappa_1^{-1}\circ\beta(q)< \kappa_2\circ \beta(r)$ for all $(q,r)\in\Upsilon$.
Thus the matrix representation of 
$z_1bz_2$ is strictly upper diagonal, and so $(z_1bz_2 )^{n^2+1} = 0$,
which establishes the desired nilpotence.
\end{proof}

\begin{theorem}\label{T-SR1 weak SD}
Suppose that $G$ is countably infinite.
Let $G\curvearrowright X$ be a topologically free minimal action on the Cantor set
with $M_G (X) \neq\emptyset$, and suppose that it is weakly squarely divisible. 
Then the reduced crossed product $C(X)\rtimes_\lambda G$ has stable rank one.
\end{theorem}

\begin{proof}
Let $a$ be a non-invertible element of $C(X)\rtimes_\lambda G$. 
By \cite[Proposition~6.2]{LiNiu20} (which assumes freeness although the proof also works for topologically free actions) we may assume, by rotating $a$ by suitable unitaries and perturbing, 
that $a\unit_O =\unit_O a=0$ for some nonempty clopen set $O\subseteq X$ and that $a$ has the form $\sum_{s\in E} g_s u_s$ for some finite symmetric set $e\in E\subseteq G$
and functions $g_s\in C(X)$. 

As in the proof of Theorem~\ref{T-SR1}, we now aim to show that $a$ can be multiplied on the left and right
by suitable unitaries in $C(X)\rtimes_\lambda G$ so as to obtain a nilpotent element.

Since the action is weakly squarely divisible there are pairwise disjoint
clopen sets $O_0, O_1 , O_2 \subseteq X$ with $O_0\subseteq O$ and a finite symmetric set $e\in F\subseteq G$ such that
$O_1 \sqcup O_2 \prec_F O_0$ and the action is $(O_1 , O_2 , FEF)$-squarely divisible. 
The subequivalence $O_1\sqcup O_2 \prec_F O_0$ means that we can find a partition of $O_1\sqcup O_2$ 
into clopen sets $A_s$ for $s\in F$ such that the sets $sA_s$ for $s\in F$ are pairwise disjoint 
and contained in $O_0$.
Writing $C$ for the complement of $(O_1 \sqcup O_2)\sqcup \bigsqcup_{s\in F} sA_s$ in $X$,
we can then define, by Remark~\ref{rem:zero-dim-unitary}, a self-adjoint unitary in $C(X)\rtimes_\lambda G$ by
\[
u = \unit_C + \sum_{s\in F} u_s \unit_{A_s} + \sum_{s\in F} u_{s^{-1}} \unit_{sA_s} .
\]
Setting $a' = uau^{-1}=uau$ we have $a' \unit_{O_1 \sqcup O_2} = \unit_{O_1 \sqcup O_2} a' = 0$,
and so replacing $a$ by $a'$ we may assume that $a$ has the form $\sum_{s\in FEF} f_s u_s$
for some functions $f_s \in C(X)$ (using the fact that $e\in F= F^{-1}$)
and satisfies $a\unit_{O_1 \sqcup O_2} = \unit_{O_1 \sqcup O_2} a = 0$.

Since the action is $(O_1 , O_2 , FEF)$-squarely divisible and $FEF$ contains the support of $a$
as an element of the algebraic crossed product, we can now proceed as in the proof of Theorem~\ref{T-SR1} 
to find unitaries $z_1 , z_2\in C(X)\rtimes_\lambda G$ such that $z_1 az_2$ is nilpotent (this is the key point where the difference between
definitions of square divisibility and weak square divisibility plays out and why we need to work in the Cantor setting here,
where we can control the support of $a$ when replacing it with $a'$).
In fact we can construct $z_1$ and $z_2$ much more easily than in the proof of Theorem~\ref{T-SR1} by proceeding as follows 
using the zero-dimensional characterization of $(O_1 , O_2 , FEF)$-squarely divisible given by
Proposition~\ref{P-SD0Dim}, which yields an $n\in\Nb$ and a collection $\{ V_{i,j} \}_{i,j=1}^n$ 
of pairwise equivalent and pairwise $FEF$-disjoint clopen subsets of $X$
such that, defining $V = \bigsqcup_{i,j=1}^n V_{i,j}$, $V_1 = \bigsqcup_{i=1}^n V_{i,1}$, $R = V^\comp$,
and $B = V\cap (V^{FEF} )^\comp$, one has the following:
\begin{enumerate}
\item $V_{i,1} \prec O_1 \cap \bigsqcup_{j=2}^n V_{i,j} \cap B^\comp$ for every $i=1,\dots ,n$,

\item $R\prec O_2 \cap V \cap (V_1 \cup B)^\comp$, and

\item $B\prec O_2 \cap R$.
\end{enumerate}
In all of the above subequivalences the two clopen sets are disjoint, and so by Remark~\ref{rem:zero-dim-unitary} we 
can construct associated (involutive) unitaries $u_i$ for $i=1,\dots , n$, $v$, and $w$ in $C(X)\rtimes_\lambda G$.
By Remark~\ref{rem:commutation_relation}, the unitaries $u_1, \dots , u_n$ pairwise commute, and 
also commute with $v$. 
Then there exist clopen sets $Y_{i,1} \subseteq O_1 \cap \bigsqcup_{j=2}^n V_{i,j} \cap B^\comp$ for $i=1,\ldots,n$,
$Y_2 \subseteq O_2 \cap V\cap (V_1\cup B)^\comp$, and $Y_3 \subseteq O_2 \cap R$
such that 
\begin{itemize}
\item $u_i\unit_{V_{i,1}} u_i^* = \unit_{Y_{i,1}}$ and $u_i \unit_{C}=\unit_{C}u_i=\unit_C$ for $C=(V_{i,1}\sqcup Y_{i,1})^\comp$ with $1\leq i\leq n$,
\item $v\unit_R v^* = \unit_{Y_2}$ and  $v \unit_{C}=\unit_{C}v=\unit_C$ for $C=(R\sqcup Y_2)^\comp$,
\item $w\unit_B w^* = \unit_{Y_3}$ and $w \unit_{C}=\unit_{C}w=\unit_C$ for $C=(B\sqcup Y_3)^\comp$.
\end{itemize}
Set $u = u_1 \cdots u_n$ and $Y_1=\bigsqcup_{i=1}^{n}Y_{i,1}$. Then $u\unit_{V_1}u^*=1_{Y_1}$ and $u\unit_C=\unit_Cu=\unit_C$
for $C=(V_1\sqcup Y_1)^\comp$.

Define $b = u^* wauv$. 
By partitioning the unit in $C(X)\rtimes_\lambda G$ as $\unit_R + \sum_{i,j=1}^n \unit_{V_{i,j}}$ 
and multiplying $b$ on either side by the projections in this sum we represent
$b$ as an $(n^2+1)\times (n^2+1)$ matrix, which we will now argue takes the form 
(assuming $n=3$ for the purpose of illustration)
\begin{gather*}
\mbox{\small 
$\left[
\begin{array}{c | c c c | c c c | c c c}
0 & 0 & * & * & 0 & * & * & 0 & * & * \\
\hline
0 & 0 & 0 & 0 & 0 & 0 & 0 & 0 & 0 & 0 \\
0 & 0 & * & * & 0 & 0 & 0 & 0 & 0 & 0 \\
0 & 0 & * & * & 0 & 0 & 0 & 0 & 0 & 0 \\
\hline
0 & 0 & 0 & 0 & 0 & 0 & 0 & 0 & 0 & 0 \\
0 & 0 & 0 & 0 & 0 & * & * & 0 & 0 & 0 \\
0 & 0 & 0 & 0 & 0 & * & * & 0 & 0 & 0 \\
\hline
0 & 0 & 0 & 0 & 0 & 0 & 0 & 0 & 0 & 0 \\
0 & 0 & 0 & 0 & 0 & 0 & 0 & 0 & * & * \\
0 & 0 & 0 & 0 & 0 & 0 & 0 & 0 & * & * 
\end{array}
\right]$},
\end{gather*}
where the first row and column correspond to the set
$R$ and the remaining rows and columns correspond to the sets $V_{i,j}$ as ordered 
lexicographically with respect to the pairs $i,j$.

Since $Y_2$ is disjoint from $V_1 \sqcup Y_1$ and contained in $O_2$, we have
\[
auv\unit_R = au\unit_{Y_2} v = a\unit_{Y_2} v =0 ,
\]
so that $b\unit_R = 0$,
which accounts for the zero leftmost column in the matrix representation of $b$.

Next, since $V_1$ is disjoint from $R\sqcup Y_2$ and $Y_1$ is a subset of $O_1$, we also have
\[
auv \unit_{V_1} = au\unit_{V_1} = a\unit_{Y_1} u = 0,
\]
so that $b\unit_{V_1}=0$, which accounts for the other $n$ columns of zeros.

Since $Y_1$ is disjoint from $B\sqcup Y_3$ and contained in $O_1$, we have
\[
\unit_{V_1} u^* wa = u^* \unit_{Y_1} wa= u^*\unit_{Y_1}a=0
\]
and hence $\unit_{V_1} b = 0$, which accounts for the $n$ rows of zeros in the lower right $n^2 \times n^2$ submatrix.

Now let $2\leq i,i' \leq n$ with $i \neq i'$ and let $1\leq j,j' \leq n$. Writing $p = (i,j)$ and $q=(i',j')$
with $V_p$ and $V_q$ interpreted accordingly as $V_{i,j}$ and $V_{i',j'}$,
let us check that $\unit_{V_p} u^* wauv \unit_{V_q} = 0$, which will account for the remaining zeros in the
matrix representation of $b$.
Since $R\setminus Y_3$ is disjoint from the supports of $u$ and $w$, $Y_1$ is contained in $O_1$, 
and $V_p$ is disjoint from $R$, we have
\begin{align*}
\unit_{V_p}u^* w\unit_R a
= \unit_{V_p}u^*w \unit_{R\setminus Y_3}a
=\unit_{V_p}u^*\unit_{R\setminus Y_3}a
= \unit_{V_p}\unit_{R\setminus Y_3}a
= 0 .
\end{align*}
On the other hand, since $Y_3$ is disjoint from the support of $u$ and $V_p$ is disjoint from $Y_3$ we have
\begin{align*}
\unit_{V_p} u^* w\unit_B a
= \unit_{V_p} u^* \unit_{Y_3} wa
= \unit_{V_p} \unit_{Y_3} wa
= 0 .
\end{align*}
For every $s\in FEF$, using that $FEF$ is symmetric we have $V^{FEF}\subseteq sV$ so that $sR\subseteq B\sqcup R$.
It follows that $a\unit_R = \unit_{R\sqcup B} a\unit_R$ by the representation of $a$ via the set $FEF$. The above two facts yield
\begin{align*}
\unit_{V_p} u^* wa\unit_R 
= \unit_{V_p} u^* w(\unit_R + \unit_B) a\unit_R 
= 0 .
\end{align*}
Since $v$ is self-adjoint and $R$ is disjoint from the support of $u$, it follows that
\begin{align}\label{E-intersection}
\unit_{V_p} u^* wauv \unit_{V_q \cap Y_2} 
= \unit_{V_p} u^* wau\unit_{R} v \unit_{V_q \cap Y_2} 
= \unit_{V_p} u^* wa\unit_{R} v\unit_{V_q \cap Y_2} 
= 0 .
\end{align}
Next, set $Z = \bigsqcup_{k=1}^n V_{i',k}$.
By the $FEF$-disjointness of the sets $V_{i,j}$, 
we have that $sZ\subseteq R\sqcup Z$ for every $s\in FEF$.
The representation of $a$ as a sum indexed
by $FEF$, along with the inclusion $Y_3 \subseteq O_2$,
then implies
\[
a\unit_{Z} = \unit_{R\sqcup Z} a\unit_{Z} = \unit_{(R\setminus Y_3 )\sqcup Z} a\unit_{Z} ,
\]
and since the support of $w$ is $B\sqcup Y_3$ we have 
\[
w \unit_{(R\setminus Y_3 )\sqcup Z} = \unit_{R\sqcup Z} w\unit_{(R\setminus Y_3 )\sqcup Z}.
\]
Moreover $u^* \unit_{R\sqcup Z} = \unit_{R\sqcup Z} u^*\unit_{R\sqcup Z}$
by the definition of $u$ and the fact that it is self-adjoint.
Putting these facts together, and using the disjointness of $V_p$ and $R\sqcup Z$, we obtain
\begin{align*}
\unit_{V_p} u^* wa\unit_{Z} 
&= \unit_{V_p} u^* w\unit_{(R\setminus Y_3 )\sqcup Z} a\unit_{Z} \\
&= \unit_{V_p} u^* \unit_{R\sqcup Z} w\unit_{(R\setminus Y_3 )\sqcup Z} a\unit_{Z} \\
&= \unit_{V_p} \unit_{R\sqcup Z} u^* \unit_{R\sqcup Z} w\unit_{(R\setminus Y_3 )\sqcup Z} a\unit_{Z} \\
&= 0
\end{align*}
and thus, since $V_q \setminus Y_2$ is disjoint from the support of $v$ and 
$u\unit_{V_q \setminus Y_2} = \unit_{Z} u\unit_{V_q \setminus Y_2}$ by the definition of $u$ and the fact that it is self-adjoint,
\begin{align*}
\unit_{V_p} u^* wauv \unit_{V_q \setminus Y_2}
= \unit_{V_p} u^* wau\unit_{V_q \setminus Y_2} 
= \unit_{V_p} u^* wa\unit_{Z} u\unit_{V_q \setminus Y_2} 
&= 0 .
\end{align*}
Combining with (\ref{E-intersection}) then yields
\begin{align*}
\unit_{V_p} u^* wauv \unit_{V_q}
= \unit_{V_p} u^* wauv (\unit_{V_q \setminus Y_2} + \unit_{V_q \cap Y_2} ) 
= 0 .
\end{align*}
We have thus verified that the matrix representation of $b$ has the desired form.

Set $\Omega = \{ 1,\dots , n\}^2$.
By the equivalence of the sets $V_{i,j}$, which as above we also write as $V_p$ where $p=(i,j)\in \Omega$, 
we can find a clopen partition $\{ C_k \}_{k\in K}$ of $V_{1,1}$ 
and tuples $s_k = (s_{k,p} )_{p\in \Omega}$ of elements of $G$ for $k\in K$ such that 
$V_p = \bigsqcup_{k\in K} s_{k,p} C_k$ for all $p\in \Omega$.
For $p,q\in\Omega$ we define the partial isometry 
\[
e_{p,q} = \sum_{k\in K} \unit_{s_{k,p} C_k} u_{s_{k,p}} u_{s_{k,q}}^{-1} \unit_{s_{k,q}C_k} 
\]
These define matrix units in $C(X)\rtimes_\lambda G$. 
The diagonal matrix units $e_{p,p}$ are the indicator functions $\unit_{V_p}$, with the identity 
matrix in the resulting copy of $M_{n^2}$ in $C(X)\rtimes_\lambda G$ equal to $\unit_V$. 

Let $\beta : \Omega \to \{ 1,\dots , n^2 \}$ be the bijection $(i,j) \mapsto (i-1)n+j$ giving the lexicographic ordering of $\Omega$.
Choose some permutation $\kappa_1'$ of $\{ 1,\dots , n^2 \}$ that for each $i=1,\dots ,n$ 
shifts the numbers in the interval $\{ (i-1)(n-1) +2 , (i-1)(n-1) +3 , \dots , (i-1)(n-1)+n\}$ by $i-1$ in the positive direction.
Choose a permutation $\kappa_2'$ of $\{ 1,\dots , n^2 \}$ that for each $i=1,\ldots,n$ shifts the numbers in the interval $\{(i-1)n+2,(i-1)n+3,\ldots,in\}$ by $n-i$ in the 
positive direction. Define the permutations $\kappa_1 = \beta^{-1} \circ \kappa_1' \circ \beta$
and $\kappa_2 = \beta^{-1} \circ \kappa_2' \circ \beta$ of $\Omega$.

We now define the two ``permutation matrices''
\begin{align*}
z_1 = \unit_R + \sum_{p\in \Omega} e_{p , \kappa_1 (p)}  \hspace*{7mm} \text{and} \hspace*{7mm}
z_2 = \unit_R + \sum_{p\in \Omega} e_{p, \kappa_2 (p) } .
\end{align*}
In the illustrative case $n=3$, multiplying $b$ by $z_1$ on the left yields a matrix of the form
\begin{gather*}
\mbox{\small 
$\left[
\begin{array}{c | c c c | c c c | c c c}
0 & 0 & * & * & 0 & * & * & 0 & * & * \\
\hline
0 & 0 & 0 & 0 & 0 & 0 & 0 & 0 & 0 & 0 \\
0 & 0 & * & * & 0 & 0 & 0 & 0 & 0 & 0 \\
0 & 0 & * & * & 0 & 0 & 0 & 0 & 0 & 0 \\
\hline
0 & 0 & 0 & 0 & 0 & * & * & 0 & 0 & 0 \\
0 & 0 & 0 & 0 & 0 & * & * & 0 & 0 & 0 \\
0 & 0 & 0 & 0 & 0 & 0 & 0 & 0 & * & * \\
\hline
0 & 0 & 0 & 0 & 0 & 0 & 0 & 0 & * & * \\
0 & 0 & 0 & 0 & 0 & 0 & 0 & 0 & 0 & 0 \\
0 & 0 & 0 & 0 & 0 & 0 & 0 & 0 & 0 & 0 
\end{array}
\right]$}
\end{gather*}
and then multiplying by $z_2$ on the right converts this into a matrix of the form
\begin{gather*}
\mbox{\small 
$\left[
\begin{array}{c | c c c | c c c | c c c}
0 & * & * & * & * & * & * & * & * & * \\
\hline
0 & 0 & 0 & 0 & 0 & 0 & 0 & 0 & 0 & 0 \\
0 & 0 & 0 & 0 & * & * & 0 & 0 & 0 & 0 \\
0 & 0 & 0 & 0 & * & * & 0 & 0 & 0 & 0 \\
\hline
0 & 0 & 0 & 0 & 0 & 0 & * & * & 0 & 0 \\
0 & 0 & 0 & 0 & 0 & 0 & * & * & 0 & 0 \\
0 & 0 & 0 & 0 & 0 & 0 & 0 & 0 & * & * \\
\hline
0 & 0 & 0 & 0 & 0 & 0 & 0 & 0 & * & * \\
0 & 0 & 0 & 0 & 0 & 0 & 0 & 0 & 0 & 0 \\
0 & 0 & 0 & 0 & 0 & 0 & 0 & 0 & 0 & 0 
\end{array}
\right]$} .
\end{gather*}
For general $n$ we similarly obtain a strictly upper triangular matrix, as is readily checked,
and so $ (z_1 b z_2 )^{n^2+1} = 0$, yielding the desired nilpotence.
\end{proof}

\section{Square divisibility and amenable groups}\label{S-amenable}

Our aim here is to establish Theorem~\ref{T-amenable}. 

\begin{lemma}\label{L-invariance}
Let $E$ be a nonempty finite subset of $G$ containing $e$ and let $0 < \delta \leq 2$. 
Let $F$ be an $(E,\delta /(2|E|))$-invariant finite subset of $G$. Then every set $F_0 \subseteq F$
with $|F_0 | \geq (1-\delta /(4|E|) )|F|$ is $(E,\delta )$-invariant.
\end{lemma}

\begin{proof}
We have $F_0^E = F^E \setminus \bigcup_{s\in E} s^{-1}(F\setminus F_0 )$ and so, using the fact that $\delta\leq 2$ to obtain the last inequality below,
\begin{align*}
|F_0^E|
\geq |F^{E}| - |E||F\setminus F_0|
&\geq \bigg(1-\frac{\delta}{2|E|}\bigg)|F| - \frac{\delta}{4} |F| \\
&\geq \bigg(1-\frac{\delta}{2|E|} - \frac{\delta }{4(1-\delta /(4|E|))}\bigg)|F_0 | \\
&\geq (1-\delta )|F_0 | . \qedhere
\end{align*}
\end{proof}

We need the following lemma to get the $E$-disjointness in Lemma~\ref{L-tiling}. 

\begin{lemma}\label{L-E disjoint}
Let $e\in E$ be a finite subset of $G$ and $\delta > 0$. 
Then for every $(E^2 ,\delta )$-invariant finite set $K\subseteq G$ the set $K^E$ 
is $(E,\delta )$-invariant. 
\end{lemma}

\begin{proof}
If $K$ is a $(E^2 ,\delta )$-invariant finite subset of $G$ then 
\begin{gather*}
|K^E | \geq |(K^E )^E| = |K^{E^2}| \geq (1-\delta )|K|. \qedhere
\end{gather*}
\end{proof}

\begin{definition}\label{D-Tiling}
Let $\sT$ be a collection of finite subsets of $G$. Let $\eps > 0$ and let $E$ be a finite subset of $G$ containing $e$. 
An {\it $(\eps ,E)$-tiling} of a finite set $K\subseteq G$ by sets in $\sT$ is a collection $\{ T_i \}_{i\in I}$
of right translates of sets in $\sT$ such that
\begin{enumerate}
\item the sets $ET_i$ for $i\in I$ are pairwise disjoint subsets of $K$, and

\item $|\bigsqcup_{i\in I} T_i | \geq (1-\eps )|K|$.
\end{enumerate}
\end{definition}

We need the following version of the Ornstein--Weiss quasitiling theorem. 

\begin{lemma}\label{L-tiling}
Suppose that $G$ is amenable. Let $E$ be a finite subset of $G$ containing $e$ and let $0<\eps \leq 1$. Then there
is a finite collection $\sT$ of $(E,\eps )$-invariant finite subsets of $G$, a finite set $D\subseteq G$, and a $\beta > 0$ 
such that for every $(D,\beta )$-invariant finite set $K\subseteq G$ there is an $(\eps ,E)$-tiling of $K$ 
by sets in $\sT$.
\end{lemma}

\begin{proof}
Take an $\eps' > 0$ such that $(1-\eps' )^3 \geq 1-\eps$. Set $\eps'' = \eps' /(2|E^2|)$.
By the Ornstein--Weiss quasitiling theorem (see Theorem~6 in \cite{OW87} and Theorem~4.36 in \cite{KerLi16}) 
there are a finite collection $\sT_0$ of $(E^2,\eps'' )$-invariant finite subsets of $G$, a finite set $D\subseteq G$, 
and a $\beta > 0$ (more precisely, one can take $\beta=\eps''/8$) such that for every $(D,\beta )$-invariant finite set $K\subseteq G$ 
there is a finite collection
$\sS$ of right translates of members of $\sT_0$ satisfying $\bigcup_{S\in\sS} S \subseteq K$ and $|\bigcup_{S\in\sS} S | \geq (1-\eps''/2 )|K|$
and pairwise disjoint sets $\tilde{S} \subseteq S$ for $S\in\sS$ such that $|\tilde{S} | \geq (1-\eps'' /2)|S|$ for every $S\in\sS$.
Given such $K$ and $\sS$, by Lemma~\ref{L-invariance} the sets $\tilde{S}$ for $S\in\sS$ are $(E^2,\eps' )$-invariant. 
Then, by Lemma~\ref{L-E disjoint}, for
every $S\in\sS$ the set $\tilde{S}^E$ 
satisfies $|\tilde{S}^E | \geq (1-\eps' )|\tilde{S}|$.
Using this in the first step and the fact that $|\tilde{S}|\geq (1-\eps')|S|$ for every $S\in \sS$ at the second step, we have
\begin{align*}
\bigg|\bigsqcup_{s\in\sS} \tilde{S}^E \bigg|
\geq (1-\eps' )\bigg|\bigsqcup_{S\in\sS} \tilde{S} \bigg| 
&\geq (1-\eps' )^2\Big|\bigcup_{S\in \sS} S\Big| \\
&\geq (1-\eps' )^3 |K| 
\geq (1-\eps )|K|.
\end{align*}
Thus $\{ \tilde{S}^E \}_{s\in\sS}$ forms an $(\eps ,E)$-tiling of $K$ by sets in $\sT$, where
\[\sT:=\{T^E: T\subseteq T_0 \text{ for some } T_0\in \sT_0 \text{ for which } |T | \geq (1-\eps'' /2)|T_0| \}. \]
It remains to show that $\sT$ consists of $(E,\eps)$-invariant sets. 
Let $T\subseteq T_0$ for some $T_0\in \sT_0$ for which $|T|\geq (1-\eps'' /2)|T_0|$. 
Since $T_0$ is $(E^2,\eps'')$-invariant, Lemma~\ref{L-invariance} implies
that $T$ is $(E^2,\eps')$-invariant, and in particular $(E^2,\eps)$-invariant. 
By Lemma~\ref{L-E disjoint} the set $T^E$ is $(E,\eps)$-invariant.
\end{proof}

\begin{lemma}\label{L-open set density}
Let $G\curvearrowright X$ be a minimal action and let
$\{O_i\}_{i\in I}$ be a finite collection of nonempty open subsets of $X$. Then there are a finite set $E\subseteq G$ and $0<\theta \leq 1/2$ such that 
every Borel tower $(S,V)$ with an $(E,1/2 )$-invariant shape $S$ satisfies
\[
\mu (O_i\cap SV) \geq \theta \mu (SV)
\]
for every $i\in I$ and $\mu\in M_G(X)$.
\end{lemma}

\begin{proof}
It suffices to prove the lemma for one nonempty open set $O\subseteq X$, since then for every $i\in I$
we can find a $0<\theta_i\leq 1/2$ and a finite set $E_i\subseteq G$ satisfying the conditions in the lemma for $O_i$, in which case we may take $\theta =\min_{i\in I} \theta_i$ and $E=\bigcup_{i\in I} E_i$.

By minimality and the compactness of $X$ there is a finite set $E\subseteq G$ 
such that $\bigcup_{t\in E} t^{-1} O = X$. Set $\theta = 1/(2|E|)$. 

Let $(S,V)$ be a Borel tower whose shape $S$ is $(E,1/2 )$-invariant and let $\mu\in M_G(X)$.
For every $s\in S$ set $n(s)=|\{(t,r)\in E\times S^E: tr=s\}|$ and note that $0\leq n(s)\leq |E|$.
We have
\begin{align*}
\frac12 |S| \mu(V)\leq \big|S^E\big|\mu(V)= \sum_{r\in S^E} \mu(rV)&\leq\sum_{r\in S^E}\sum_{t\in E}\mu(rV\cap t^{-1}O)\\
&=\sum_{r\in S^E}\sum_{t\in E}\mu(trV\cap O)\\
&= \sum_{s\in E} n(s)\mu(sV\cap O)\\
&\leq |E|\mu(SV\cap O)
\end{align*}
Dividing by $|E|$ we obtain $\mu(SV\cap O)\geq \theta |S|\mu(V)=\theta\mu(SV)$, as desired.
\end{proof}

The following lemma is a consequence of the portmanteau theorem (see \cite[Proposition~3.4]{KerSza20}).
\begin{lemma}\label{L-nbhd}
Let $G\curvearrowright X$ be an action, $A\subseteq X$ a closed set, and $\delta > 0$.
Then there is an open set $U\supseteq A$ such that 
$\sup_{\mu\in M_G (X)} \mu (U) < \sup_{\mu\in M_G (X)} \mu (A) + \delta$.
\end{lemma}

\begin{theorem}\label{T-amenable}
Suppose that $G$ is infinite and amenable. Let $G\curvearrowright X$ be a minimal action 
that has the URP and comparison.
Let $O_1$ and $O_2$ be nonempty open subsets of $X$ and $E$ a finite subset of $G$ containing $e$. 
Then the action is $(O_1 , O_2 , E)$-squarely divisible. In particular, the action is squarely divisible.
\end{theorem}

\begin{proof}
By Lemma~\ref{L-open set density} there exist a finite set $F\subseteq G$ and $0<\theta\leq 1/2$ such that 
for every Borel tower $(S,V)$ with $(F,1/2 )$-invariant shape $S$ one has, for $i=1,2$,
\begin{align}\label{E-O}
\mu (O_i \cap SV) \geq \theta \mu (SV),
\end{align}
for every $\mu\in M_G(X)$. Let $n$ be an integer 
greater than $2/\theta + 1$. Set $\eps = \theta /(20n^2)$.
 
By Lemma~\ref{L-tiling} there are a finite collection $\sT$ of $(EF\cup E,\theta\eps/4 )$-invariant finite subsets of $G$,
a finite set $D\subseteq G$, and a $\beta > 0$ 
such that for every $(D,\beta )$-invariant finite set $K\subseteq G$ there is an $(\eps ,E)$-tiling
of $K$ by sets in $\sT$. 

Since the action has the URP, there exists
an open castle $\{ (S_k , V_k ) \}_{k\in K}$ with $(D,\beta )$-invariant
shapes such that, writing $R_0 = X\setminus \bigsqcup_{k\in K} S_k V_k$ for the remainder,
\begin{gather}\label{E-R0}
\sup_{\mu\in M_G(X)} \mu (R_0) < \frac{\theta\eps}{12n^2}.
\end{gather}
We will moreover assume without loss of generality that the sets $\{s\overline{V_k} : s\in S_k, k\in K\}$ are pairwise disjoint (see for example Lemma~6.3 in \cite{GarGefNarVac24}) and that $e\in S_k$ for all $k\in K$ (if $e\notin S_k$ then we can pick a $t\in S_k$ and replace $S_k$ by $S_kt^{-1}$ and $V_k$ by $t^{-1}V_k$).
Since $G$ is infinite, by improving the almost invariance that we require of the shapes we may assume, for every $k\in K$, that
\begin{gather}\label{E-cardS_k}
\eps |S_k|> n^2 |\sT| \max_{T\in\sT} |T|.
\end{gather}
Using Lemma~\ref{L-nbhd} and the inequality (\ref{E-R0}) together with the fact that $\partial R_0 \subseteq R_0$, 
we can find an open set $U\supseteq \partial R_0$ satisfying
\begin{gather}\label{E-min}
\sup_{\mu\in M_G(X)} \mu (\overline{U}) <  \frac{\theta\eps}{12n^2}.
\end{gather}
Since $\partial R_0=\bigsqcup_{k\in K}\bigsqcup_{s\in S_k} s\partial V_k$, by shrinking $U$ appropriately we may assume that it is the union of sets 
$sU_k$ for $s\in S_k$ and $k\in K$ with pairwise disjoint closures where $U_k$ is some open set containing $\partial V_k$. Moreover, we may assume that 
$s\overline{U_k} \cap t\overline{V_\ell}=\emptyset$
whenever $(s,k)\neq (t,\ell)$ for $s\in S_k$, $t\in S_\ell$ and $k, \ell\in K$.
Indeed, since $\{s\overline{V_k}: s\in S_k,\, k\in K\}$ is a disjoint collection we can use a separation argument 
to find open sets $Q_{s,k}\supseteq s\overline{V_k}$ for $s\in S_k$ and $k\in K$ with pairwise disjoint closures,
in which case the sets $U_k:=\bigcap_{s\in S_k} s^{-1}(U\cap Q_{s,k})$ for $k\in K$ satisfy the required properties. These properties in particular imply that
\begin{gather}\label{Uclosure}
\overline{U}^c\cap sV_k = s(\overline{U}^c \cap V_k) 
\end{gather}
for every $s\in S_k$ and $k\in K$ (using that $e\in S_k$ for every $k\in K$).

Let $k\in K$. By our choice of $D$ and $\beta$ there exist finite sets $C_{k,T} \subseteq G$ for $T\in\sT$ such that the sets $ETc$ for $T\in\sT$ and $c\in C_{k,T}$ are
pairwise disjoint subsets of $S_k$ and $ |\bigsqcup_{T\in\sT}\bigsqcup_{c\in C_{k,T}}Tc|\geq (1-\eps)|S_k|$, in which case
\begin{gather}\label{E-inequC_k,T}
(1-\eps)|S_k|\leq \sum_{T\in\sT} |T| |C_{k,T}|\leq |S_k|.
\end{gather}
We claim that the inequality $\eps|S_k|>\max_{T\in\sT} |T|$ guaranteed by (\ref{E-cardS_k}) permits us 
to find sets $C_{k,T}' \subseteq C_{k,T}$ for $T\in\sT$ satisfying
\begin{gather}\label{E-eps_k}
(1-3\eps)|S_k|\leq \sum_{T\in\sT}  |T||C_{k,T}' | \leq (1-2\eps )|S_k |.
\end{gather}
To see this, first write $\sT=\{T_1,\ldots, T_m\}$. Finding  sets $C_{k,T_i}' \subseteq C_{k,T_i}$ for $i=1,\ldots, m$ which satisfy (\ref{E-eps_k}) is equivalent to 
finding integers $0\leq r_i\leq |C_{k,T_i}|$ (specifying the number of elements we will remove from $C_{k,T_i}$) for $i=1,\ldots,m$ so that 
\[\sum_{i=1}^{m} |T_i||C_{k,T_i}|-(1-2\eps)|S_k|\leq \sum_{i=1}^{m}|T_i|r_i\leq \sum_{i=1}^{m} |T_i||C_{k,T_i}|-(1-3\eps)|S_k|.\]
Set $L_1=\sum_{i=1}^{m} |T_i||C_{k,T_i}|-(1-2\eps)|S_k|$ and $L_2= \sum_{i=1}^{m} |T_i||C_{k,T_i}|-(1-3\eps)|S_k|$, and note that $L_2-L_1=\eps|S_k|$.
Define $\mathcal{A}=\big\{\sum_{i=1}^{m}|T_i|r_i : 0\leq r_i\leq |C_{k,T_i}|\big\}$, which is a finite set of nonnegative integers. 
Our goal is to find an element $a\in \mathcal{A}$ with $L_1\leq a\leq L_2$.

Pick a maximal element $a\in\mathcal{A}$ satisfying $a\leq L_1$.
If $a= L_1$, then the proof of the claim is complete. 
We therefore may assume that $a< L_1$.
Then there must exist some $i\in\{1,\ldots, m\}$ with $r_i<|C_{k,T_i}|$.
Now $a+|T_i|$ is an element of $\mathcal{A}$ and by the choice of $a$ we have $a+|T_i|>L_1$.
Using that $|T_i|<\eps|S_k|$ we also have $a+|T_i|<L_1+\eps|S_k|=L_2$, which completes the proof of the claim.

Now define
\begin{align*}
Z = \bigsqcup_{k\in K} \bigsqcup_{T\in\sT} TC_{k,T} V_k \hspace*{4mm}\text{and}\hspace*{4mm}
Z' = \bigsqcup_{k\in K} \bigsqcup_{T\in\sT} TC_{k,T}' V_k .
\end{align*}
Then, for $\mu\in M_G(X)$,
\begin{align*}
\mu(Z)
=\sum_{k\in K} \sum_{T\in\sT} |T||C_{k,T} | \mu (V_k ) 
\hspace*{4mm}\text{and}\hspace*{4mm}
\mu (Z')
= \sum_{k\in K} \sum_{T\in\sT} |T||C_{k,T}' | \mu (V_k ). 
\end{align*}
In view of (\ref{E-eps_k}) we conclude that
\begin{align*}
(1-\eps )\mu (R_0^c )\leq \mu (Z ) \leq \mu(R_0^c)
\hspace*{4mm}\text{and}\hspace*{4mm}
(1-3\eps )\mu (R_0^c )\leq \mu (Z' ) \leq (1-2\eps)\mu(R_0^c).
\end{align*}
Thus
\begin{align}\label{E-lower}
\mu (Z\setminus Z' ) 
= \mu (Z) - \mu (Z' ) \geq \eps \mu (R_0^c ) 
\stackrel{(\ref{E-R0})}{>} \frac{\eps}{2} .
\end{align}

For each $k\in K$ and $T\in \sT$ choose pairwise disjoint sets $C_{k,T,i,j} \subseteq C_{k,T}'$ of equal cardinality 
indexed by $i,j=1,\dots , n$ so that their union has cardinality greater than $|C_{k,T}'| - n^2$ (for instance, take each $C_{k,T,i,j}$ to have cardinality $\lfloor |C_{k,T}'|/n^2\rfloor$). Then we have
\begin{align}\label{E-double union}
\bigg| \bigsqcup_{T\in\sT} \bigsqcup_{i,j=1}^n TC_{k,T,i,j} \bigg| 
&= \sum_{T\in\sT} \sum_{i,j=1}^n |T||C_{k,T,i,j} | \\
&> \sum_{T\in\sT} |T|(|C_{k,T}'| - n^2 ) \notag \\
&\geq  \sum_{T\in\sT} |T||C_{k,T}'| - n^2 |\sT |\max_{T\in\sT} |T| \notag \\
&\leftstackrel{(\ref{E-eps_k}), (\ref{E-cardS_k})}{>}  (1-4\eps )|S_k |. \notag
\end{align}
For $i,j=1,\dots ,n$ we now define 
\[
V_{i,j} = \bigsqcup_{k\in K} \bigsqcup_{T\in\sT} TC_{k,T,i,j} V_k . 
\]
Note that the sets $V_{i,j}$ are $E$-disjoint due to the fact that the sets $ETc$ for $T\in\sT$ and $c\in C_{k,T}$
are pairwise disjoint and contained in $S_k$.
The fact that for each $k\in K$ and $T\in\sT$ the sets $C_{k,T,i,j}$ for $i,j=1,\dots , n$ have equal
cardinality guarantees that the sets $TC_{k,T,i,j} V_k$ for $i,j=1,\dots , n$ are pairwise equivalent.
It follows that the sets $V_{i,j}$ are pairwise equivalent.
Write $V=\bigsqcup_{i,j=1}^{n}V_{i,j}$. 
Since $\partial V\subseteq \partial R_0\subseteq U$,
it remains to check that the set $\{V_{i,j}\}_{i,j=1}^{n}$ together with
$R := X\setminus V$ and $B:=\overline{V} \cap ((V\cap \overline{U}^c )^E )^c$
satisfy the three itemized conditions in the definition of square divisibility.

For $\mu\in M_G(X)$ we have, using the bound $\mu(R_0)\leq\eps\leq \eps /(1-4\eps )$ guaranteed by (\ref{E-R0}) 
to get the last inequality,
\begin{align}\label{E-footprint}
\mu \bigg(\bigsqcup_{i,j=1}^n V_{i,j} \bigg)
\stackrel{(\ref{E-double union})}{\geq} (1-4\eps )\sum_{k\in K} |S_k |\mu (V_k ) 
= (1-4\eps )\mu (R_0^c )
\geq 1-5\eps. 
\end{align}
The pairwise equivalence of the sets $\{V_{i,j}\}_{i,j=1}^{n}$ in conjunction with (\ref{E-footprint}) yields, for all $i,j=1,\dots ,n$, 
\begin{align}\label{E-n squared}
\mu (V_{i,j} ) \geq \frac{1-5\eps}{n^2}\geq \frac{1}{2n^2} .
\end{align}
Letting $1\leq i,j \leq n$ and using (\ref{E-n squared}) we observe that 
\begin{align}\label{E-U}
\mu (\overline{U}) \stackrel{(\ref{E-min})}{<} \frac{\eps}{2n^2} \leq \eps \mu (V_{i,j} ) .
\end{align}
For a set $\Omega\subseteq \{(i,j): 1\leq i,j\leq n\}$ we 
let $V_\Omega=\bigsqcup_{(i,j)\in \Omega}V_{i,j}$.
Using that the sets $\{V_{i,j}\}_{i,j=1}^{n}$ are $E$-disjoint, it is straightforward to verify that 
\begin{gather}\label{E-VOmega}
V_\Omega\cap B^c=(V_{\Omega}\cap \overline{U}^c)^E.
\end{gather}
We also have that
\begin{align}\label{E-Vomega}
(V_\Omega\cap \overline{U}^c)^E\supseteq \bigsqcup_{(i,j)\in\Omega}( V_{i,j}\cap \overline{U}^c)^E \;\; &\leftstackrel{(\ref{Uclosure})}{=}
\bigsqcup_{(i,j)\in \Omega} \Big( \bigsqcup_{k\in K}\bigsqcup_{T\in\sT}TC_{k,T,i,j}(V_k\cap \overline{U}^c) \Big)^{\!E}\\ 
&\supseteq \bigsqcup_{(i,j)\in\Omega} \bigsqcup_{k\in K}\bigsqcup_{T\in\sT} T^EC_{k,T,i,j}(V_k\cap \overline{U}^c).\notag
\end{align}
Together with the $(E,\theta\eps/4)$-invariance of the sets $T\in\sT$ and equation~(\ref{Uclosure}), 
this yields, for all $\mu\in M_G(X)$,
\begin{align}\label{E-(1-eps)}
\mu\big((V_\Omega\cap \overline{U}^c)^E\big)\geq (1-\theta\eps/4)\mu(V_\Omega\cap \overline{U}^c).
\end{align}

For $i=1,2$ we have, using at the second step that for every $T\in\sT$ the set $T^E$ is $(F,1/2)$-invariant
(by the $(EF,\eps)$-invariance of $T$) and at the second to last step
that every $T\in \sT$ is $(E,\eps)$-invariant,
\begin{align}\label{VOmegaO}
\mu(O_i\cap(V_\Omega\cap\overline{U}^c)^E)\;\;&\leftstackrel{(\ref{E-Vomega})}{\geq}
\sum_{(i,j)\in\Omega}\sum_{k\in K}\sum_{T\in\sT}\sum_{c\in C_{k,T,i,j}} \mu(O_i\cap T^Ec(V_k\cap\overline{U}^c))\\ 
&\leftstackrel{(\ref{E-O})}{\geq} \theta \sum_{(i,j)\in\Omega}\sum_{k\in K}\sum_{T\in\sT}\sum_{c\in C_{k,T,i,j}} \mu( T^Ec(V_k\cap\overline{U}^c))\notag\\ 
&= \theta \sum_{(i,j)\in\Omega}\sum_{k\in K}\sum_{T\in\sT} |T^E||C_{k,T,i,j}|\mu(V_k\cap\overline{U}^c)\notag\\ 
&\geq \theta (1-\eps)\sum_{(i,j)\in\Omega}\sum_{k\in K}\sum_{T\in\sT} |T||C_{k,T,i,j}|\mu(V_k\cap\overline{U}^c)\notag\\ 
&\leftstackrel{(\ref{Uclosure})}{=} \theta (1-\eps) \mu (V_\Omega\cap \overline{U}^c).\notag
\end{align}
Moreover, since $\overline{V_{i,j}}\subseteq \overline{U}\cup (V_{i,j}\cap \overline{U}^c)$ and $\mu(\overline{U})<\eps /(2n^2)\stackrel{(\ref{E-n squared})}{\leq}\eps \mu(\overline{V_{i,j}})$ we have 
\[\mu(V_{i,j}\cap \overline{U}^c)\geq \mu(\overline{V_{i,j}})-\mu(\overline{U})\geq (1-\eps)\mu(\overline{V_{i,j}}) \] 
and hence
\begin{gather}\label{E-V_OmegaUcomplement}
\mu(V_\Omega\cap \overline{U}^c)\geq (1-\eps)\mu(\overline{V_\Omega}).
\end{gather}

Next let $1\leq i\leq n$. Set $W_i = \bigsqcup_{j=2}^n V_{i,j}$ and
$Y_i = O_1 \cap W_i\cap B^c \stackrel{(\ref{E-VOmega})}{=} O_1 \cap (W_i \cap \overline{U}^c )^E$.
Since $(1-\eps)^2\geq 1/2$, for $\mu\in M_G(X)$ we then have
\begin{align*}
\mu (Y_i ) 
\stackrel{(\ref{VOmegaO})}{\geq} \theta (1-\eps)\mu(W_i\cap \overline{U}^c)\stackrel{(\ref{E-V_OmegaUcomplement})}{\geq}\theta(1-\eps)^2\mu(\overline{W_i})\geq \frac{\theta}{2}\mu(\overline{W_i}),
\end{align*}
and thus, since the sets $\overline{V_{i,j}}$ are pairwise disjoint and pairwise Borel equivalent, and $n > 2/\theta + 1$,
\begin{align*}
\mu (\overline{V_{i,1}}) 
= \frac{1}{n-1} \mu (\overline{W_i} ) 
\leq \frac{2}{\theta (n-1)} \mu (Y_i ) 
< \mu (Y_i ) .
\end{align*}
By hypothesis the action has comparison, and so we infer that $\overline{V_{i,1}} \prec Y_i$, for each $i=1,\ldots,n$,
which verifies condition (i) in the definition of square divisibility.

We next show that $R\prec O_2\cap V\cap V_1^c\cap B^c$, which we will again do via comparison.
If $\mu\in M_G(X)$ and $\Omega=\{(i,j): 1\leq i\leq n, 2\leq j\leq n\}$, then
$V_\Omega=V\cap V_1^c$ and we have, using at the last step below that 
$\eps=\theta /(20n^{2})\leq \theta(n^2-n)/(20n^2)$,
\begin{align*}
\mu (O_2 \cap V \cap V_1^c \cap B^c) \;\;
&\leftstackrel{(\ref{E-VOmega})}{=} \mu (O_2 \cap (V_\Omega \cap \overline{U}^c )^E )\\
&\leftstackrel{(\ref{VOmegaO}),(\ref{E-V_OmegaUcomplement})}{\geq} \theta(1-\eps)^2\mu(V_\Omega)\\
&\geq \frac{\theta}{2}\mu(V_\Omega)\\
&=\frac{\theta}{2}(n^2-n)\mu(V_{i,j})\\
&\leftstackrel{(\ref{E-n squared})}{\geq} \frac{\theta}{4n^2}(n^{2}-n)\\
&>5\eps.
\end{align*}
Since $\mu(R)\leq 5\eps$ by (\ref{E-footprint}),
comparison then yields condition (ii) in the definition
of square divisibility.

Finally we verify, again using comparison, that
$B\cup (\overline{U}\cap R)\prec O_2\cap \overline{V\cup U}^c$.
Since $\overline{V}$ is equal to the disjoint union of $\overline{V} \cap ((V\cap \overline{U}^c )^E )^c$
and $(V\cap \overline{U}^c )^E$, given $\mu\in M_G (X)$ we have
\begin{align*}
\mu (B) 
=\mu (\overline{V} \cap ((V\cap \overline{U}^c )^E )^c )
&= \mu (\overline{V} ) - \mu ((V\cap \overline{U}^c )^E ) \\
&\leftstackrel{(\ref{E-(1-eps)})}{\leq} \mu (\overline{V} ) - (1-\theta\eps/4 )\mu (V\cap \overline{U}^c ) \\
&= (1-\theta\eps/4 )(\mu (\overline{V} ) - \mu (V\cap \overline{U}^c )) + \frac{\theta\eps}{4} \mu (\overline{V} ) \\
&=(1-\theta\eps/4 )\mu (\overline{V} \cap \overline{U}) +  \frac{\theta\eps}{4} \mu (\overline{V} ) \\
&\leq \mu (\overline{U}) +  \frac{\theta\eps}{4} ,
\end{align*}
and therefore $\mu(B\cup (\overline{U}\cap R))\leq 2\mu(\overline{U})+\theta\eps/4$.
On the other hand, 
\[\overline{V\cup U}=\overline{V}\cup \overline{U}\subseteq \bigg(\bigsqcup_{k\in K}\bigsqcup_{T\in\sT}\bigsqcup_{c\in C_{k,T}\setminus C_{k,T}'} TcV_k\bigg)^{\!\!c} \cup \overline{U}, \]
so that, combining with the equation (\ref{Uclosure}), 
\[\overline{V\cup U}^c \supseteq \bigg(\bigsqcup_{k\in K}\bigsqcup_{T\in\sT}\bigsqcup_{c\in C_{k,T}\setminus C_{k,T}'} Tc(V_k\cap \overline{U}^c)\bigg). \]
We thus have, using the $(F,1/2)$-invariance of the sets in $\sT$ for the second inequality,
\begin{align*}
\mu (O_2 \cap \overline{V\cup U}^c )
&\geq \sum_{k\in K} \sum_{T\in\sT} \sum_{c\in C_{k,T} \setminus C_{k,T}'} \mu (O_2 \cap Tc(V_k\cap \overline{U}^c) ) \\
&\leftstackrel{(\ref{E-O})}{\geq} \theta \sum_{k\in K} \sum_{T\in\sT} \sum_{c\in C_{k,T} \setminus C_{k,T}'} \mu (Tc(V_k\cap \overline{U}^c) ) \\
&= \theta \sum_{k\in K}\sum_{T\in\sT}|T|(|C_{k,T}|-|C_{k,T}'|)(\mu(V_k)-\mu(V_k\cap\overline{U}))\\
&\geq \theta \mu (Z\setminus Z' ) -\mu(\overline{U})\stackrel{(\ref{E-lower})}{>} \frac{\theta\eps}{2} -\mu(\overline{U}).
\end{align*}
Since $\mu(\overline{U})< \theta\eps/12$, we conclude that
\[\mu(B\cup (\overline{U}\cap R))\leq 2\mu(\overline{U})+\frac{\theta\eps}{4}\leq \frac{\theta\eps}{2}-\mu(\overline{U})<\mu(O_2\cap \overline{V\cup U}^c), \]
in which case comparison gives
$B\cup (\overline{U}\cap R)\prec O_2 \cap \overline{V\cup U}^c$, yielding (iii) in the definition of square divisibility.

The last statement in the theorem follows from Proposition~\ref{P-stronger}.
\end{proof}

\section{Square divisibility and product actions}\label{S-product}

\begin{theorem}\label{T-SD product}
Let $G\curvearrowright X$ and $H\curvearrowright Y$ be minimal actions 
with $G$ infinite, and suppose that the first action has the URP and comparison.
Then the product action $G\times H\curvearrowright X\times Y$
given by $(g,h)(x,y) = (gx,hy)$ is squarely divisible.
\end{theorem}

\begin{proof}
Let $O_X \subseteq X$ and $O_Y \subseteq Y$ be nonempty open sets. It is enough to verify that 
the product action is $O_X \times O_Y$-squarely divisible.

Since $H\curvearrowright Y$ is minimal, for every $y\in Y$ there are an open neighbourhood $V_y\subseteq Y$ of $y$
and an $h\in H$ such that $hV_y \subseteq O_Y$. A compactness argument then shows that 
there is a finite open cover $\{ V_h\}_{h\in F}$ of $Y$ indexed by a finite set $F\subseteq H$
such that $hV_h \subseteq O_Y$ for every $h\in F$.

Since $G\curvearrowright X$ has the URP it is essentially free, i.e., for every $\mu \in M_G(X)$ 
the action $G\curvearrowright (X,\mu)$ is free. 
The minimality of $G\curvearrowright X$ then implies that $G\curvearrowright X$ is topologically free.
There thus exists an $x\in X$ such that $sx\neq tx$ for all distinct $s,t\in G$.
Since $X$ has no isolated points (as ensues from the minimality and topological freeness of the action 
of $G$ together with the infiniteness of $G$) and the action of $G$ is minimal, it follows that we can find  
elements $g_h\in G\setminus \{e\}$ indexed by $h\in F$ such that the points $g_hx$ for $h\in F$ are distinct
and contained in $O_X$. A separation argument
then yields an open neighbourhood $O_0$ of $x$ such that the sets
$\{g_h O_0\}_{h\in F}$ are pairwise disjoint subsets of $O_X$ with $\overline{O_0}\cap g_h O_0=\emptyset$ for every $h\in F$.

Again using the nonexistence of isolated points in $X$, we can find nonempty open sets
$O_1 , O_2 \subseteq X$ such that $\overline{O_1} \cap \overline{O_2} = \emptyset$
and $\overline{O_1} \sqcup \overline{O_1} \subseteq O_0$.
Set $O_0' = (O_X\setminus \overline{O_0})  \times O_Y$ and $O_i' = O_i \times Y$ for $i=1,2$. 
Clearly $O_0', O_1'$, and $O_2'$ are nonempty disjoint open subsets of $X\times Y$ with 
$\overline{O_1'}\cap \overline{O_2'}=\emptyset$ and $O_0'\cap (\overline{O_1'}\sqcup \overline{O_2'})=\emptyset$.

To conclude that the product action is $O_X\times O_Y$-squarely divisible it remains to show that $\overline{O_1'}\sqcup \overline{O_2'}\prec O_0'$ and that $G\times H\curvearrowright X\times Y$ is $(O_1',O_2', E)$-squarely divisible for a given finite set $e\in E\subseteq G\times H$, which we may assume to be a product set $K\times L$.
Since
\begin{align*}
\overline{O_1'} \sqcup \overline{O_2'} = (\overline{O_1} \sqcup \overline{O_2}) \times Y
\subseteq \bigcup_{h\in F} (O_0\times V_h)
\end{align*}
we have
\begin{align*}
\overline{O_1'} \sqcup \overline{O_2'}
\prec \bigsqcup_{h\in F} (g_h, h)( O_0\times V_h ) 
\subseteq (O_X\setminus \overline{O_0}) \times O_Y  = O_0' .
\end{align*}
On the other hand, by Theorem~\ref{T-amenable} the action $G\curvearrowright X$ is $(O_1 , O_2 , K)$-squarely divisible,
and if we consider all of the sets that appear in conditions (i)--(iii) in Definition~\ref{D-SD} as witnesses to this 
$(O_1 , O_2 , K)$-square divisibility and take their products with $Y$, then it is readily checked
that these product sets witness the desired $(O_1' , O_2' , E)$-squarely divisibility of the product action.
\end{proof}

\section{A diagonal action machine}\label{S-diagonal machine}

Our goal here is to develop a mechanism for constructing diagonal actions of free products
satisfying certain properties that will be employed in later sections to various ends
and with varying degrees of power,
specifically in the proofs of Propositions~\ref{P-top free} and \ref{P-trivial factor},
Lemma~\ref{L-SD free product} (on the way to Theorem~\ref{T-SD free products}), 
and Theorem~\ref{T-SD free groups}.
First we carry out a series of lemmas that help us build a machine, namely Proposition~\ref{P-min WM}, 
that outputs diagonal actions that are minimal and weakly mixing
(in the case of Theorem~\ref{T-SD free groups} we will actually 
only rely separately on some of these lemmas). The specific inputs to this machine
that we will need in our applications are furnished by Proposition~\ref{P-extending}
and involve local freeness and (as is relevant 
to Lemma~\ref{L-SD free product}) square divisibility. 

We will require some tools from ergodic theory for this program. 
By the notation $(Z,\zeta )$ we always mean
an atomless standard probability space, of which there is only one up to measure isomorphism.
We write $\Act (G,Z,\zeta )$ for the space of all measure-preserving 
actions $G\curvearrowright (Z,\zeta )$ equipped with the Polish topology that has as a basis the open sets
\[
U_{\alpha , \Omega , E,\delta}
= \{ \beta \in\Act (G,Z,\zeta ) : \mu (\beta_s A \Delta \alpha_s A ) < \delta\text{ for all } s\in E \text{ and } A\in\Omega \}
\]
where $\alpha$ is an element of $\Act (G,Z,\zeta )$, 
$\Omega$ is a finite collection of measurable subsets of $Z$, $E$ is a finite subset of $G$, and $\delta > 0$.
For definitions and background on the aspects of ergodic theory at play in this and the next section,
including weak mixing, the Koopman representation, and disjointness,
see \cite{KerLi16} and \cite{Gla03}. Although we use the terminology ``weak mixing'' and ``disjoint''
in both the measure-dynamical and topological-dynamical senses, the context will always make it clear which is intended.

Recall that two actions $G\curvearrowright X$ and $G\curvearrowright Y$ are {\it disjoint} if there is no proper nonempty closed subset $Z\subseteq X\times Y$ which is invariant under the diagonal action
$s(x,y) = (sx,sy)$ and satisfies $\pi_X(Z)=X$ and $\pi_Y(Z)=Y$ for the canonical projection maps onto $X$ and $Y$, respectively.

\begin{lemma}\label{L-min disjoint}
Suppose that $H$ is a subgroup of $G$. Let $G \stackrel{\alpha}{\curvearrowright} X$ be a minimal action and 
$G\stackrel{\beta}{\curvearrowright} Y$ an action whose restriction $H\stackrel{\tilde{\beta}}{\curvearrowright} Y$ is minimal.
Suppose there exists a nonempty closed $H$-invariant set $A\subseteq X$
such that the action $H \curvearrowright A$ restricting $\alpha$ is minimal and disjoint from $\tilde{\beta}$.
Then the diagonal action $G \stackrel{\alpha\times\beta}{\curvearrowright} X\times Y$ is minimal. In particular, $\alpha$ and $\beta$ are disjoint.
\end{lemma}

\begin{proof}
For brevity set $\gamma = \alpha\times\beta$.
Let $(x,y)\in X\times Y$ and let $U\subseteq X$ and $V\subseteq Y$ be nonempty open sets.
It suffices to show the existence of an $r\in G$ such that $\gamma_r (x,y) \in U\times V$.

By the minimality of $\alpha$ there exists an $s\in G$ such that $\alpha_s U \cap A \neq\emptyset$,
in which case $\gamma_s (U\times V) \cap (A\times Y) \neq\emptyset$. 
Since the actions $H \stackrel{\alpha}{\curvearrowright} A$ and $H\stackrel{\tilde{\beta}}{\curvearrowright} Y$ are minimal and disjoint, their product $H\curvearrowright A\times Y$ is minimal. It follows, using the 
compactness of $A\times Y$, that there are finitely many $H$-translates of the open set $\gamma_s(U\times V)$ which cover $A\times Y$, which means that we can find a finite collection $\{ W_t \}_{t\in F}$ of open 
subsets of $X\times Y$ indexed by some finite set $F\subseteq H$ such that 
$A\times Y\subseteq \bigcup_{t\in F} W_t$ 
and $\gamma_t W_t \subseteq \gamma_s (U\times V)$ for every $t\in F$.
By the tube lemma and the compactness of $A$ we can then find an open set $O\subseteq X$ such that $A\subseteq O$
and $O\times Y \subseteq \bigcup_{t\in F} W_t$.

Again using the minimality of $\alpha$, there exists a $k\in G$ such that $\alpha_k x \in O$.
Then $\gamma_k (x,y) \in O\times Y$, and so
there is a $t_0 \in F$ for which $\gamma_k (x,y) \in W_{t_0}$.
It follows that
\[
\gamma_{s^{-1} t_0 k} (x,y)
\in \gamma_{s^{-1}} \gamma_{t_0} W_{t_0}
\subseteq \gamma_{s^{-1}} \gamma_s (U\times V)
= U\times V ,
\]
which completes the proof.
\end{proof}

Recall that a {\it joining} of two p.m.p.\ actions $G\curvearrowright(Z_1,\zeta_1)$ and $G\curvearrowright (Z_2,\zeta_2)$ is a probability measure $\zeta$ on $Z_1\times Z_2$ which is invariant under
the diagonal action $G\curvearrowright Z_1\times Z_2$, given by $s(z_1 , z_2 ) = (sz_1 , sz_2 )$, and projects factorwise onto $\zeta_1$ and $\zeta_2$. The p.m.p.\ actions $G\curvearrowright(Z_1,\zeta_1)$ and $G\curvearrowright (Z_2,\zeta_2)$ are {\it disjoint} if their only joining is the product measure $\zeta_1\times\zeta_2$.

\begin{lemma}\label{L-product disjoint}
Suppose $G$ is amenable. Let $G\curvearrowright X$ and $G\curvearrowright Y$ be minimal actions
such that $M_G (Y)$ contains a unique measure $\nu$.
Suppose that there exists an ergodic $\mu\in M_G (X)$ such that the p.m.p.\ actions
$G\curvearrowright (X,\mu )$ and $G\curvearrowright (Y,\nu )$ are disjoint. 
Then the diagonal action $G\curvearrowright X\times Y$ is minimal.
In particular, the actions $G\curvearrowright X$ and $G\curvearrowright Y$ are disjoint.
\end{lemma}

\begin{proof}
Let $(x_0,y_0) \in X\times Y$, and let $A$ denote the orbit closure of $(x_0,y_0)$ under $G\curvearrowright X\times Y$. We will show that $A = X\times Y$.

By the pointwise ergodic theorem \cite{Shu88, Tem92} (see also \cite[Theorem~1.2]{Lin01}), 
if we fix a tempered F{\o}lner sequence $(F_n )_{n\in\Nb}$ for $G$ (temperedness meaning that $|\bigcup_{k < n} F_k^{-1}F_n| \le C|F_n|$ for all $n\in \Nb$ where $C>0$ is a constant, a property that can be arranged by passing to a suitable subsequence of any given F{\o}lner sequence), there exists an $x\in X$ such that 
\begin{equation}
\int_X f\, d\mu = \lim_{n\to \infty} \frac{1}{|F_n|} \sum_{s\in F_n} f(s^{-1}x)  \label{E-X marginal}
\end{equation}
for all $f\in C(X)$
(actually this holds for $\mu$-a.e.\ $x\in X$). Since $G\curvearrowright X$ is minimal, the projection of $A$ onto $X$, being $G$-equivariant, is surjective. Thus there exists a $y\in Y$ such that $(x,y) \in A$. 

Since $\nu$ is the unique measure in $M_G(Y)$, we have 
\begin{equation}
\int_Y f\,d\nu = \lim_{n\to\infty} \frac{1}{|F_n|}\sum_{s\in F_n} f(s^{-1}y)  \label{E-Y marginal}
\end{equation}
for all $f\in C(Y)$, a well-known phenomenon that was originally observed by Oxtoby \cite{Oxt52} in the case $G = \Zb$. 
Indeed, if this were not true then one could construct an invariant measure $\rho \in M(Y)$ as a weak$^*$ cluster point of point mass averages $\nu_n := |F_n|^{-1}\sum_{s\in F_n} \delta_{s^{-1}y}$ for which $\int_Y f\,d\rho$ is bounded away from $\int_Y f\,d\nu$, in which case $\rho\ne \nu$.

Applying the same convergence-to-invariance principle using the F{\o}lnerness of the sets $F_n$, 
the point mass averages
$\lambda_n := |F_n|^{-1} \sum_{s\in F_n} \delta_{s^{-1}(x,y)} \in M(X\times Y)$
have a weak$^*$ cluster point $\lambda\in M_G (X\times Y)$.
By passing to a subsequence of $( F_n )_{n\in\Nb}$ and relabelling, we may assume that $\int_{X\times Y} f\,d\lambda = \lim_{n\to\infty} |F_n|^{-1}\sum_{s\in F_n} f(s^{-1}x,s^{-1}y)$ for all $f\in C(X\times Y)$. Since $(x,y) \in A$ and $A$ is closed and $G$-invariant, each $\lambda_n$ is supported on $A$ and hence $\lambda$ is supported on $A$. By (\ref{E-X marginal}) and (\ref{E-Y marginal}), $\lambda$ projects factorwise onto $\mu$ and $\nu$ hence by the disjointness of $G\curvearrowright (X,\mu)$ and $G\curvearrowright (Y,\nu)$
we infer that $\lambda = \mu \times \nu$. In particular, $\mu \times \nu$ is supported on $A$. By the minimality of $G\curvearrowright X$ and $G\curvearrowright Y$, each of $\mu$ and $\nu$ have full support and hence $A = X\times Y$.
\end{proof}

A unitary representation $\pi : G \to \cB(\cH)$ is {\it weakly mixing} if it has no nonzero finite-dimensional invariant subspaces, or, equivalently, if for every finite set $\Omega\subseteq \cH$ and $\eps>0$ there exists  $s\in G$ such that $|\langle\pi(s)\xi,\zeta\rangle|<\eps$ for all $\xi,\zeta\in \Omega$ (see \cite[Theorem~2.23]{KerLi16}).

\begin{lemma}\label{L-atomless}
The set of atomless measures in $M(\Tb )$ is a dense $G_\delta$.
\end{lemma}

\begin{proof}
For a finite set $F\subseteq C(\Tb)$ and $\delta>0$ write $W_{F,\delta}$ for the set of all $\mu\in M(\Tb )$
such that there exists an $m\in\Zb$ for which $|\int_{\Tb} z^m f\bar{g} \, d\mu| < \delta$ for all $f,g\in F$. This set
is weak$^*$ open. Indeed if $\mu_n\notin W_{F,\delta}$ and $\mu_n$ converges to $\mu$ 
in the weak$^*$ topology then for every $m\in\Zb$ we can find $n_1<n_2<n_2<\ldots$ and $f,g\in F$ so that $|\int_{\Tb} z^m f\bar{g} \, d\mu_{n_k}|\geq\delta$ for all $k\in\Nb$. 
Convergence gives $|\int_{\Tb} z^m f\bar{g} \, d\mu|\geq\delta$, so that $\mu\notin W_{F,\delta}$.
To see that $W_{F,\delta}$ is dense, let $\mu\in M(\Tb )$. For each $k\in\Nb$ define $\mu_k$ to be the measure 
in $M(\Tb )$ which for each $j=0,\ldots , k-1$ is equal on the half-open arc $I_{j,k}$ from 
$e^{2\pi i j/k}$ (inclusive) to $e^{2\pi i (j+1)/k}$ (exclusive) to Lebesgue measure,
normalized to have the same mass on the arc as $\mu$. 

We check that $\mu_k \to \mu$ as $k\to\infty$. Let $h\in C(\Tb)$ and $\eps>0$. By uniform continuity, there exists $k_0\in\Nb$ such that $|h(z)-h(w)|<\eps /2$ whenever $z,w\in I_{j,k}$ for some $k\geq k_0$ and $0\leq j \leq k-1$. Let $k\geq k_0$ and pick an $x_{j,k}\in I_{j,k}$ for each $j=0,\ldots,k-1$. Then for each $j=0,\ldots, k-1$ we have
\begin{align*}
\left|\int_{I_{j,k}} h \,d\mu-\int_{I_{j,k}} h\, d\mu_k\right|&\leq \left|\int_{I_{j,k}} (h-h(x_{j,k}))\, d\mu\right|+\left|\int_{I_{j,k}} (h(x_{j,k})-h)\, d\mu_k\right|\\
&\leq \frac{\eps}{2}(\mu(I_{j,k})+\mu_k(I_{j,k}))=\eps\mu(I_{j,k})
\end{align*}
and hence
\begin{align*}
\left|\int_\Tb h\, d\mu-\int_\Tb h\, d\mu_k\right|&\leq \sum_{j=0}^{k-1}\left|\int_{I_{j,k}} h\, d\mu-\int_{I_{j,k}} h\, d\mu_k\right|\\
&\leq \sum_{j=0}^{k-1}\eps \mu(I_{j,k})=\eps.
\end{align*}
This proves that $\mu_k\to \mu$ as $k\to \infty$.
For $f,g\in C(\Tb )$, uniform continuity and the fact that $\int_{I_{j,k}}z^k\,dz=0$ for all $k\in\Nb$ and $j=0,\ldots,k-1$ imply that $\int_{\Tb} z^k f\bar{g} \, d\mu_k \to 0$ as $k\to\infty$.
This convergence shows that $\mu_k$ is eventually always in $W_{F,\delta}$, and so we 
obtain the density of $W_{F,\delta}$.

Since $C(\Tb )$ is separable, it has a countable dense subset $\Upsilon$. 
Set $W = \bigcap_{F\in\mathcal{P}_\mathrm{fin}(\Upsilon)} \bigcap_{n\in\Nb} W_{F,1/n}$, where $\mathcal{P}_\mathrm{fin}(\Upsilon)$ denotes the countable collection of all finite subsets of $\Upsilon$. By the Baire category theorem,
the set $W$ is a dense $G_\delta$. Since for every $\mu\in M(\Tb )$ the elements of $\Upsilon$ 
modulo $\mu$-a.e.\ equality form a dense subset of $L^2 (\Tb , \mu )$, 
a simple approximation argument
shows that $\mu$ belongs to $W$ precisely when the associated unitary representation
of $\Zb$ on $L^2 (\Tb , \mu )$ by the multiplication operators $z^n$ is weakly mixing.
Weak mixing is equivalent in this case to the atomlessness of the measure $\mu$, as is 
well known and easy to see.
\end{proof}

Recall that two positive bounded Borel measures $\mu$ and $\nu$ on $X$ are said to be \emph{mutually singular}, written $\mu\perp \nu$, if there exists a Borel set $A\subseteq X$ for which $\mu(X\setminus A)=0$ and $\nu(A)=0$.

Let $K$ be a compact group and let $\mu$ and $\nu$ be positive bounded Borel measures on $K$. 
The \emph{convolution} of $\mu$ and $\nu$ is defined by $(\mu*\nu )(E)=\int_K \mu(Ez^{-1})\,d\nu(z)$ for Borel sets $E\subseteq K$. The convolution of $n$ copies of $\mu$ is written $\mu^{*n}$.
Note that if $\mu=\sum_{i=1}^{n}\alpha_i\delta_{x_i}$ and $\nu=\sum_{j=1}^{m}\beta_j\delta_{y_j}$ are atomic measures then $\mu*\nu=\sum_{i,j}\alpha_i\beta_j\delta_{x_iy_j}$.

The \emph{conjugate} $\overline{\nu}$ of a measure $\nu\in M(\Tb )$ is defined
by $\overline{\nu}(B)=\nu(\{z\in\Tb: \overline{z}\in B\})$ for Borel sets $B\subseteq \Tb$.

\begin{lemma}\label{L-disjoint powers}
Let $\mu\in M(\Tb )$. Then the set of all $\nu\in M(\Tb )$ such that $(\nu + \overline{\nu} )^{*n}\perp \mu$
for all $n\in\Nb$ is a dense $G_\delta$.
\end{lemma}

\begin{proof}
Let $\omega\in M(\Tb )$, let $\Upsilon$ be a finite subset of $C(\Tb )$, and let $\eps > 0$. Let $m\in\Nb$ and $\eps' > 0$, to be determined. 
For $k=1,\dots , m$ write $I_{k,m}$ for the half-open arc from
$e^{2\pi i(k-1)/m}$ (inclusive) to $e^{2\pi i k/m}$ (exclusive).
For every $k=1,\ldots,m$ find $p_k,q_k\in \Nb\cup \{0\}$ such that 
$\left|\omega(I_{k,m})-p_k /q_k\right|<\eps'$ and $\sum_{k=1}^{m} p_k /q_k=1$. Take a set
$F\subseteq \Tb$, to be further specified below, 
whose intersection with the arc $I_{k,m}$ for a given $1\leq k\leq m$ 
contains exactly $p_k\cdot\prod_{j\neq k}q_j$ points.
Letting $\kappa\in M(\Tb)$ be the uniform probability measure supported on $F$, one can verify that
\begin{enumerate}
\item $|\omega (I_{k,m}) - \kappa (I_{k,m})| < \eps'$ for every $k=1,\dots ,m$.
\end{enumerate}
Write $C$ for the (countable) set
of all $z\in\Tb$ such that $\mu ( \{ z \} ) > 0$. 
We claim that $F$ can be chosen so that, writing $\overline{F}=\{ \overline{z} : z\in F \}$,
\begin{enumerate}
\item[(ii)] the sets $F^a\overline{F}^b$
are disjoint from $C$ for all integers $a,b\geq 0$.
\end{enumerate}
Set $N_k=p_k\cdot\prod_{j\neq k}q_j$. Then the configuration space $U$ for the possible sets $F$ 
satisfying our original requirement that $\left|F\cap I_{k,m}\right|=p_k\cdot\prod_{j\neq k}q_j$ 
for every $k=1,\ldots,m$ can be parameterized as $\prod_{k=1}^{m} (I_{k,m}^{N_k}\setminus \Delta_{k,m} )\subseteq \prod_{k=1}^{m}\Tb^{N_k}$,
where $\Delta_{k,m}:=\{(x_1,\ldots,x_{N_k})\in \Tb^{N_k}:x_i=x_j \text{ for some } i\neq j\}$. Note that $\lambda(U)=\prod_{k=1}^{m} m^{-N_k}>0$, where $\lambda$ denotes normalized Lebesgue measure, since each $\Delta_{k,m}$ can be viewed as a subset of a finite union of hyperplanes in $\mathbb{R}^{N_k}$ via the identification of $\Tb$ with $[0,1)$ and hence has measure zero.
On the other hand, writing $N=N_1+\ldots+N_m$, the ``bad'' set is contained in the countable union
\begin{align*}
A:=\bigcup_{c\in C}\,\bigcup_{a,b=0}^{\infty}\,\bigcup_{\pi\in\{1,\ldots,N\}^a} \,\bigcup_{\sigma\in\{1,\ldots,N\}^b} A_{c,a,b,\pi , \sigma}
\end{align*}
where for $\pi = (i_1,\ldots,i_a)$ and $\sigma = (j_1,\ldots,j_b)$ the set $A_{c,a,b,\pi ,\sigma}$
is defined as
\begin{align*}
\bigg\{(p_1,\ldots,p_{N})\in \prod_{k=1}^{m}\mathbb{T}^{N_k}: p_{i_1}\cdots p_{i_a}\cdot\overline{p}_{j_1}\cdots\overline{p}_{j_b}=c\bigg\}.
\end{align*}
Again identifying $\Tb$ with $[0,1)$, the set $A_{c,a,b,\pi ,\sigma}$ can be viewed 
as a subset of a finite union of hyperplanes in $\prod_{k=1}^{m}\mathbb{R}^{N_k}$, namely the
intersection of $\prod_{k=1}^{m}[0,1)^{N_k}$ with
\begin{align*}
\bigcup_{\ell=-b-1}^{a-1}\bigg\{(p_1,\ldots,p_{N})\in \prod_{k=1}^{m}\mathbb{R}^{N_k}: 
\sum_{k=1}^{a}p_{i_k}-\sum_{k=1}^{b}p_{j_k}=c+\ell\bigg\} ,
\end{align*}
and hence has measure zero. We conclude that $U\setminus A$ has nonzero measure, 
and so $F$ can be chosen to satisfy condition (ii) above.

By the uniform continuity of functions in $C(\Tb )$, we can take $m$ large enough and $\eps'$ small enough so that 
condition (i) implies that $\kappa$ belongs
to the weak$^*$ open neighbourhood 
\begin{equation}
W_{\omega , \Upsilon , \eps} := \bigg\{ \nu\in M(\Tb ) : \bigg|\int_{\mathbb{T}} f\, d\omega - \int_{\mathbb{T}} f\, d\nu \bigg| < \eps \text{ for every }f\in\Upsilon \bigg\}. \label{E-nbhd}
\end{equation}
of $\omega$.

Now let $\delta > 0$ and $n\in\Nb$. Define $Q_{\delta,n}$ to be the set of all $\nu\in M(\Tb)$ for which there exists an open set $V\subseteq \Tb$ such that $\mu(V)<\delta$ and $(\nu+\overline{\nu})^{*j}(\Tb\setminus V)<\delta$ for every $j=1,\ldots,n$.
By the portmanteau theorem, if $\rho_n \to \rho$ weak$^*$ inside $M(Z)$ for some compact metrizable space 
$Z$ then $\limsup_{n\to\infty} \rho_n(Y)\geq \rho(Y)$ for every closed set $Y\subseteq Z$,
a fact which, when combined with the weak$^*$ continuity of addition and convolution of bounded positive measures on compact groups \cite[Theorem~1.2.2]{Hey77}, shows that $Q_{\delta,n}$
is open.

We will also verify that $Q_{\delta, n}$ is dense. 
Since the sets $W_{\omega,\Upsilon,\eps}$ defined in (\ref{E-nbhd}) form a basis for the weak$^*$ topology on $M(\Tb)$, it suffices to show that each of these sets intersects $Q_{\delta,n}$. For given
$\Upsilon$ and $\eps$ we have our $F\subseteq\Tb$ and $\kappa\in W_{\omega,\Upsilon,\eps}$ as above.
We claim that $\kappa$ belongs to $Q_{\delta,n}$. By the disjointness condition in (ii) we can find, for every $z\in \bigcup_{j=1}^{n}\bigcup_{k=0}^{j} F^k \overline{F}^{j-k}$, a small enough open arc
$U_z$ in $\Tb$ containing $z$ so that the union $U = \bigcup_{j=1}^{n}\bigcup_{k=0}^{j}\bigcup_{z\in F^k\overline{F}^{j-k}} U_z$ satisfies $\mu (U) < \delta$. On the other hand, for every $j=1,\ldots,n$ the atomic measure $(\kappa+\overline{\kappa})^{*j}$ is supported on $\bigcup_{k=0}^{j} F^k \overline{F}^{j-k}$. Thus $(\kappa+\overline{\kappa})^{*j}(\Tb\setminus U)=0$ for all $j=1,\ldots,n$, which shows that $\kappa\in Q_{\delta,n}$ and hence that $W_{\omega,\Upsilon,\eps}$ intersects $Q_{\delta,n}$.

Now $\bigcap_{m=1}^\infty Q_{1/m,m}$ is a dense $G_\delta$ set,
and we claim that it is equal to the set in the statement of the proposition.
First we let $\nu \in M(\Tb)$ be such that $(\nu+\overline{\nu})^{*n}\perp \mu$ for all $n\in\Nb$ and check that $\nu\in \bigcap_{m=1}^\infty Q_{1/m,m}$. Given $m\in\Nb$, we want to find an open set $V\subseteq \Tb$ such that $\mu(V)<1/m$ and $(\nu+\overline{\nu})^{*j}(\Tb\setminus V)<1/m$ for every $j=1,\ldots,m$. 
Since $\sum_{j=1}^{m}(\nu+\overline{\nu})^{*j}\perp \mu$ from our assumption,
there is a Borel set $B\subseteq \Tb$ such that $\mu(B)=0$ and $(\nu+\overline{\nu})^{*j}(\Tb\setminus B)=0$ for every $j=1,\ldots,m$. It follows by the regularity of $\mu$ that there is an open set $V\subseteq \Tb$ satisfying $V\supseteq B$ and $\mu(V)<1/m$. Then $(\nu+\overline{\nu})^{*j}(\Tb\setminus V)=0$ for every $j=1,\ldots,m$ and so $\nu$ belongs to $\bigcap_{m=1}^\infty Q_{1/m,m}$.

For the reverse inclusion, let $\nu \in \bigcap_{m=1}^{\infty} Q_{1/m,m}$.
For every $m\in\Nb$ we can find an open set $V_m\subseteq \Tb$ for which $\mu(V_m)<1/m$ and $(\nu+\overline{\nu})^{*j}(\Tb\setminus V_m)<1/m$ for every $j=1,\ldots,m$. 
Consider the Borel set $V=\bigcap_{k=1}^{\infty}\bigcup_{m=k}^{\infty}V_{m^2}$. We have $\mu(V)\leq \sum_{m=k}^{\infty}\mu(V_{m^2})<\sum_{m=k}^{\infty} m^{-2}$ for every $k\in\Nb$, and hence $\mu(V)=0$. Finally, writing $\Tb \setminus V = \bigcup_{k=1}^\infty W_k$ where $W_k = \bigcap_{m=k}^\infty \Tb \setminus V_{m^2}$, we have $(\nu+\overline{\nu})^{*n}(W_k)\leq (\nu+\overline{\nu})^{*n}(\Tb\setminus V_{m^2})<1/m^2$ for all $m\ge \max\{k,\sqrt{n}\}$. Hence $(\nu+\overline{\nu})^{*n}(W_k)=0$ for all $k\in\Nb$, which yields $(\nu+\overline{\nu})^{*n}(\Tb\setminus V)=0$.
Thus $(\nu + \overline{\nu})^{*n} \perp \mu$ for all $n\in \Nb$.
\end{proof}

\begin{notation}
Given an action $\alpha\in \Act (G,X)$,
for a pair of open sets $U_1 , U_2\subseteq X$ we write $N_\alpha (U_1 , U_2 )$ for the set of all $s\in G$
such that $\alpha_s U_1 \cap U_2 \neq\emptyset$. 
Given an action $\alpha\in\Act(G,Z,\zeta)$, for a pair of measurable sets $U, V \subseteq Z$ 
and $\eps > 0$ 
we write $N_{\alpha ,\eps} (U , V )$ for the set of all $s\in G$ such that 
$\zeta (\alpha_s U \cap V ) > \zeta (U ) \zeta (V ) - \eps$. 
\end{notation}

For convenient reference we record the following well-known fact. One applies
the Ornstein--Weiss quasitower theorem (see \cite[Theorem~4.45]{KerLi16})
to deduce that for infinite amenable $G$ the conjugacy class of any free action in $\Act (G,Z,\zeta )$ is dense 
in the set of all free actions in $\Act (G,Z,\zeta )$, which in turn is dense in $\Act (G,Z,\zeta )$ by a result of 
Glasner and King \cite[Appendix (2)]{GlaKin98}.

\begin{lemma}\label{L-conj class}
Suppose that $G$ is infinite and amenable. Then the conjugacy class of every free action in $\Act (G,Z,\zeta )$ is dense.
\end{lemma}

By definition, an action $\alpha\in \Act (G,X)$ is weakly mixing if and only if the intersection of $N_\alpha (U_1 , U_2 )$ and $N_\alpha (V_1 , V_2 )$ is nonempty for all nonempty open sets $U_1 , U_2 ,V_1 , V_2 \subseteq X$. 

\begin{lemma}\label{L-Gamma}
Suppose $G$ is infinite and amenable. Let $L$ be an infinite subset of $G$.
Let $\Gamma$ be the set of all $\alpha\in\Act (G, Z,\zeta )$ such that 
$N_{\alpha ,\eps} (A_1 , A_2 ) \cap N_{\alpha ,\eps} (B_1 , B_2 ) \cap L \neq\emptyset$ for all measurable sets $A_1 , A_2 , B_1 , B_2 \subseteq Z$ and $\eps > 0$. 
Then $\Gamma$ contains a dense $G_\delta$ subset of $\Act(G,Z,\zeta)$.
\end{lemma}

\begin{proof}
As explained in the proof of Theorem~1.7 in \cite{KerLi16}, we can find a countable collection $\sD$ of measurable subsets of $Z$ so that for every $\eta>0$ and every measurable set $A\subseteq Z$ there exists $\tilde{A}\in \sD$ with $\zeta(A\Delta \tilde{A})<\eta$. For a tuple $T = (A_1 , A_2 ,B_1 , B_2 )\in\sD^4$
and $\eps > 0$ write $\sW_{T,\eps}$ for the set of all $\alpha\in \Act (G, Z,\zeta )$ such that 
$N_{\alpha ,\eps} (A_1 , A_2 ) \cap N_{\alpha ,\eps} (B_1 , B_2 ) \cap L \neq\emptyset$.
Let $\alpha\in \sW_{T,\eps}$ and choose an $s\in N_{\alpha ,\eps} (A_1 , A_2 ) \cap N_{\alpha ,\eps} (B_1 , B_2 ) \cap L$. It is straightforward to verify that there exists a sufficiently small $\eps'>0$ so that whenever $\beta\in U_{\alpha, \{A_1,A_2,B_1,B_2\}, s,\eps'}$, then $\beta$ also belongs to  $\sW_{T,\eps}$, and so $\sW_{T,\eps}$ is open.

There exist free mixing actions in $\Act (G,Z,\zeta )$, e.g., the Bernoulli action $G\curvearrowright (Y^G,\nu^G)$ 
for some nontrivial standard probability space $(Y,\nu)$ (see Section~2.3.1 of \cite{KerLi16}), identifying
$(Z,\zeta)$ with $(Y^G,\nu^G)$. Now since the conjugacy class of any free action in $\Act (G,Z,\zeta )$ is dense
by Lemma~\ref{L-conj class}, the set of mixing actions in $\Act (G,Z,\zeta )$ is dense.
Since $L$ is infinite, every mixing action in $\Act (G,Z,\zeta )$ is contained in $\sW_{T,\eps}$,
and so $\sW_{T,\eps}$ is dense in $\Act (G,Z,\zeta )$. It follows that the set 
$\sW := \bigcap_{T\in\sD^4} \bigcap_{n=1}^\infty \sW_{T,1/n}$ is a dense $G_\delta$.

To complete the proof we will verify that $\sW\subseteq\Gamma$. Let $\alpha\in\sW$ and $\eps > 0$,
and let $A_1 , A_2 ,B_1 , B_2$ be measurable subsets of $Z$. By the density of $\sD$ we can find
$\tilde{A}_1 , \tilde{A}_2 , \tilde{B}_1 , \tilde{B}_2 \in\sD$ such that 
$\zeta (\tilde{A}_1 \Delta A_1 )$, $\zeta (\tilde{A}_2 \Delta A_2 )$, $\zeta (\tilde{B}_1 \Delta B_1 )$, and
$\zeta (\tilde{B}_2 \Delta B_2 )$ are all less than $\eps /8$.
By the definition of $\sW$ there exists 
an $s\in N_{\alpha ,\eps /2} (\tilde{A}_1 , \tilde{A}_2 ) \cap N_{\alpha ,\eps /2} (\tilde{B}_1 , \tilde{B}_2 ) \cap L$.
We then have
\begin{align*}
|\zeta (A_1 ) \zeta (A_2 ) - \zeta (\tilde{A}_1 )\zeta (\tilde{A}_2 )|
&\leq |\zeta (A_1 ) - \zeta (\tilde{A}_1 )|\zeta (A_2 ) + \zeta (\tilde{A}_1 )|\zeta (A_2 ) - \zeta (\tilde{A}_2 )| \\
&<\frac{\eps}{8} + \frac{\eps}{8} = \frac{\eps}{4}
\end{align*}
and hence, using the fact that $A\cap B\subseteq (C\cap D)\cup (A\Delta C) \cup (B\Delta D)$ for any sets $A,B,C,D$, 
\begin{align*}
\zeta (\alpha_sA_1 \cap A_2 )
&\geq \zeta (\alpha_s\tilde{A}_1 \cap \tilde{A}_2 ) - \zeta (\alpha_s(A_1 \Delta \tilde{A}_1 )) - \zeta (A_2 \Delta \tilde{A}_2 ) \\
&> \bigg( \zeta (\tilde{A}_1 ) \zeta (\tilde{A}_2 ) - \frac{\eps}{2} \bigg) - \frac{\eps}{8} - \frac{\eps}{8} \\
&\geq \zeta (A_1 ) \zeta (A_2 ) - \eps.
\end{align*}
Similarly, 
\[
\zeta (\alpha_sB_1 \cap B_2 ) > \zeta (B_1 ) \zeta (B_2 ) - \eps .
\]
Thus $s\in N_{\alpha ,\eps} (A_1 , A_2 ) \cap N_{\alpha ,\eps} (B_1 , B_2 ) \cap L$, showing that $\alpha\in\Gamma$.
\end{proof}

\begin{lemma}\label{L-WM coset}
Let $G\stackrel{\alpha}{\curvearrowright} X$ be an action for which $M_\alpha (X)$ contains a measure of full support. Let $H$ be an infinite subgroup of $G$.
Let $U_1 , U_2 , V_1 , V_2 \subseteq X$ be nonempty open sets and 
let $s\in N_\alpha (U_1 , U_2 ) \cap N_\alpha (V_1 , V_2 )$. Then 
the set $Hs \cap N_\alpha (U_1 , U_2 ) \cap N_\alpha (V_1 , V_2 )$ is infinite.
\end{lemma}

\begin{proof}
Let $\mu\in M_\alpha (X)$ be of full support.
Then the product measure $\mu\times\mu$ on $X\times X$ has full support.
By assumption the open sets $\alpha_s U_1 \cap U_2$ and $\alpha_s V_1 \cap V_2$ are nonempty,
and so by Poincar{\'e} recurrence (see for example \cite[Theorem~2.10]{KerLi16}) applied to the diagonal
action $H \curvearrowright (X\times X,\mu \times\mu )$ we deduce that the set of all $h\in H$
such that 
\[
(\alpha_h (\alpha_s U_1 \cap U_2 ) \times \alpha_h (\alpha_s V_1 \cap V_2)) \cap 
((\alpha_s U_1 \cap U_2 ) \times (\alpha_s V_1 \cap V_2)) \neq\emptyset
\] 
is infinite. For such $h$ we have in particular $\alpha_{hs} U_1 \cap U_2 \neq\emptyset$
and $\alpha_{hs} V_1 \cap V_2 \neq\emptyset$, i.e., 
$hs \in N_\alpha (U_1 , U_2 ) \cap N_\alpha (V_1 , V_2 )$.
\end{proof}

A unitary operator $U\in \cB(\cH)$ is said to be {\it weakly mixing} if the associated unitary representation $\mathbb{Z}\to \cB(\cH)$ given by $n\mapsto U^n$ is weakly mixing.
This is equivalent to the nonexistence of eigenvectors, i.e., the nonexistence 
of nonzero $\xi\in \cH$ such that $U\xi=\lambda\xi$ for some $\lambda\in \mathbb{C}$.

We additionally require some facts from harmonic analysis which we summarize here. 
Let $\cH$ be a separable Hilbert space.
Given a unitary operator $U\in \cB(\cH)$, its \textit{spectral measures} are finite Borel measures $\{\sigma_\xi\}_{\xi\in \cH}$ on $\mathbb{T}$ with the property that 
$\langle U^n\xi,\xi\rangle = \int_{\mathbb{T}}z^n d\sigma_\xi(z)$, for all $n\in\Zb$.
Moreover, there exists a $\sigma_U\in M(\mathbb{T})$, uniquely determined up to membership in the same measure class,
with the property that $\sigma_\xi\ll\sigma_U$ for all $\xi\in \cH$ and for every finite Borel measure $\nu$ on $\mathbb{T}$ with $\nu\ll\sigma_U$ there exists $\xi\in \cH$ for which $\nu\sim \sigma_\xi$. 
This measure, or more accurately its measure class, is called the \textit{maximal spectral type} of $U$.
For $\nu\in M(\Tb)$, the multiplication operator $M_z\in \cB(L^2(\Tb,\nu))$, given by $f\mapsto zf$, has $\nu$ as its maximal spectral type. The maximal spectral type of a tensor product $U_1\otimes \cdots\otimes U_n$ of unitary operators is the convolution measure $\sigma_{U_1}*\cdots*\sigma_{U_n}$. Finally, the maximal spectral type of a countable direct sum $\bigoplus_{n=1}^{\infty} U_n$ of unitary operators is given by $\sum_{n=1}^{\infty}c_n\sigma_{U_n}$, where $c_n$ are any positive numbers satisfying $\sum_{n=1}^{\infty} c_n=1$.

For a transformation $T\in \Aut(Z,\zeta)$ we consider its \textit{reduced maximal spectral type} $\sigma_{T,0}\in M(\Tb)$, defined as the maximal spectral type of the unitary operator $U_{T,0}\in \cB(L_0^2(Z,\zeta))$ that we obtain by restricting the Koopman operator 
$U_T \xi = \xi\circ T^{-1}$ on $L^2(Z,\zeta)$ to the closed invariant subspace $L_0^2(Z,\zeta) = L^2(Z,\zeta)\ominus \Cb1$. Given $\gamma\in \Act(\Zb,Z,\zeta)$ we denote by $\sigma_{\gamma,0}\in M(\Tb)$ the reduced maximal spectral type of the generating transformation. 
Note also that $\sigma_{T^{-1},0}=\overline{\sigma_{T,0}}$ (and the same holds for the ``usual'', i.e., non-reduced maximal spectral type). Two transformations $T,S\in \Aut(Z,\zeta)$ are \textit{spectrally disjoint} if  $\sigma_{T,0}\perp \sigma_{S,0}$, a property known to imply disjointness (e.g., see \cite[Theorem~6.28]{Gla03}).

Let $T_1\in \Aut(Z_1,\zeta_1)$ and $T_2\in \Aut(Z_2,\zeta_2)$ be two transformations and consider the associated
Koopman operators $U_{T_1}\in \cB(L^2(Z_1,\zeta_1))$ and $U_{T_2}\in \cB(L^2(Z_2,\zeta_2))$. 
Using that unitarily equivalent unitary operators have the same spectral measure, one checks that the reduced maximal spectral type $\sigma_{T_1\times T_2,0}$ is equivalent to $\sigma_{T_1,0} + \sigma_{T_2,0} + (\sigma_{T_1,0} * \sigma_{T_2,0})$. More generally, the reduced maximal spectral type of a (possibly infinite) product is equivalent to the (normalized) sum of all possible convolutions of reduced maximal spectral types of the factors in the product.

For a unitary operator $U\in \cB(\cH)$, we denote by $\overline{U}$ the induced unitary operator 
on the conjugate Hilbert space $\overline{\cH}$. For $\nu\in M(\mathbb{T})$ we have a unitary isomorphism from $\overline{L^2(\mathbb{T},\nu)}$ to $L^2(\mathbb{T},\overline{\nu})$ given by $f\mapsto [z\mapsto \overline{f(\bar{z})}]$, which sends $\overline{M_z}$ to $M_z$.

Finally, we recall the construction of  symmetric tensor products. Let $\cH$ be a Hilbert space, let $n\in\Nb$, and let $\sigma\in \mathrm{Sym}(n)$. Write
$U_\sigma\in \cB(\cH^{\otimes n})$ for the unitary operator determined by
$U_\sigma(\xi_1\otimes\cdots\otimes \xi_n)=\xi_{\sigma(1)}\otimes\cdots\otimes\xi_{\sigma(n)}$ on elementary tensors.
The \textit{$n$-fold symmetric tensor product} of $\cH$ is the closed Hilbert subspace
\[\cH^{\odot n}:=\{\xi\in \cH^{\otimes n}: U_\sigma(\xi)=\xi \text{ for all } \sigma\in \mathrm{Sym}(n)\}. \]
Every unitary operator $U\in \cB(\cH)$ determines a unitary operator $U^{\odot n} \in \cB(\cH^{\odot n})$ via the restriction of $U^{\otimes n}$ to $\cH^{\odot n}$.

\begin{lemma}\label{L-nontorsion}
Suppose that $G$ is amenable and contains a normal infinite cyclic subgroup. Let $\Omega$ be a finite subset of $\Act (G, Z,\zeta )$. 
Then the set of all free weakly mixing actions in $\Omega^\perp$ (i.e., those that are disjoint from all actions in $\Omega$) is a dense $G_\delta$ subset of $\Act (G,Z,\zeta )$.
\end{lemma}

\begin{proof}
By a result of \cite{Jun81} (which is stated and proved there for ergodic actions of $G=\Zb$, although the argument 
works more generally) the set $\{\alpha\}^\perp$ is a $G_\delta$ for every $\alpha\in\Omega$.
The last paragraph in the proof of Theorem~5.21 in \cite{KerLi16} shows that the set of all weakly mixing actions
in $\Act (G,Z,\zeta )$ is a $G_\delta$, while \cite[Appendix (2)]{GlaKin98} shows that set of all free actions
in $\Act (G,Z,\zeta )$ is a $G_\delta$.
The intersection of all of these $G_\delta$ sets is again a $G_\delta$ set,
and so we conclude that the set of all free weakly mixing actions in $\Omega^\perp$ 
is a $G_\delta$ subset of $\Act (G,Z,\zeta )$. We note that the set of all free weakly mixing actions
in $\Act (G,Z,\zeta )$ is actually a dense $G_\delta$, as can be seen by Lemma~\ref{L-conj class} 
and the existence of Bernoulli actions, which are weakly mixing, or by applying the general result in \cite{KerPic08}.
It would thus now be sufficient to show that $\Omega^\perp$ contains at least one free action, 
since again using Lemma~\ref{L-conj class} we could then conclude that $\Omega^\perp$ is dense.
However, the disjointness argument below is already contingent on weak mixing via the atomlessness of $\nu$,
and so we will get weak mixing anyway, in addition to freeness.

By assumption $G$ has a normal infinite cyclic subgroup $H$, which we identify with $\Zb$.  First we show the existence of a free  weakly mixing action in $\Omega^\perp$.
For $\alpha \in \Omega$, let $\sigma_{\alpha',0} \in M(\Tb)$ denote the reduced maximal spectral type of the restriction 
$\alpha'$ of $\alpha$ to $H$.
By Lemmas~\ref{L-atomless} and \ref{L-disjoint powers} there is an
atomless measure $\nu\in M(\Tb )$ such that
$(\nu + \overline{\nu})^{*n} \perp \sigma_{\alpha',0}$ for all $\alpha\in\Omega$ and $n\in \Nb$.
Let $\pi : \Zb \to \cB (L^2 (\Tb , \nu ))$ be the unitary representation
associated to the multiplication operator $M_z\in \cB (L^2 (\Tb , \nu ))$ and let $\Zb \stackrel{\gamma}{\curvearrowright} (M,\mu)$ denote its associated Gaussian action, in which case $(M,\mu )$ is standard and atomless (see Appendix~E of \cite{KerLi16}).
Since the measure $\nu$ is atomless the representation $\pi$ is weakly mixing, which implies that
the action $\gamma$ is weakly mixing \cite[Theorem~2.38]{KerLi16}. 
The transformation obtained by restricting $\gamma$ to the generator $1\in\Zb$ gives rise to
a Koopman operator $U_\gamma \in\cB(L^2(M,\mu))$ along with its restriction $U_{\gamma,0} \in\cB(L_0^2(M,\mu))$. 
As shown in Theorem~E.19 in \cite{KerLi16}, there exists a unitary isomorphism $L^2(M,\mu)\simeq \bigoplus_{n=0}^{\infty} (L^2(\Tb,\nu)\oplus \overline{L^2(\Tb,\nu)})^{\odot n}$ (with the direct summand for $n=0$ equal to the complex numbers) which carries the unitary operator $U_\gamma$ to the unitary operator $\bigoplus_{n=0}^{\infty}(M_z\oplus \overline{M_z})^{\odot n}$.
Since $\gamma$ is weakly mixing, $U_{\gamma,0}$ is weakly mixing (see \cite[Proposition~2.7]{KerLi16}), so that $L_0^2(M,\mu)$ is exactly the orthogonal complement of vectors fixed by $U_\gamma$. Since $\nu$ is atomless, the unitary operators $(M_z\oplus \overline{M_z})^{\otimes n}$ for $n\in \Nb$ are weakly mixing (see \cite[Theorem~2.23]{KerLi16}), and so likewise $\bigoplus_{n=1}^{\infty} (L^2(\Tb,\nu)\oplus \overline{L^2(\Tb,\nu)})^{\odot n}$ is the orthogonal complement of vectors fixed by $\bigoplus_{n=0}^{\infty}(M_z\oplus \overline{M_z})^{\odot n}$. It follows that the isomorphism $L^2(M,\mu)\simeq \bigoplus_{n=0}^{\infty} (L^2(\Tb,\nu)\oplus \overline{L^2(\Tb,\nu)})^{\odot n}$ restricts to a unitary isomorphism $L_0^2(M,\mu)\simeq \bigoplus_{n=1}^{\infty} (L^2(\Tb,\nu)\oplus \overline{L^2(\Tb,\nu)})^{\odot n}$. Consequently, $U_{\gamma,0}$ and $\bigoplus_{n=1}^{\infty}(M_z\oplus \overline{M_z})^{\odot n}$ are unitarily equivalent via this isomorphism.

For each $n\in\mathbb{N}$, the unitary operator $(M_z\oplus \overline{M_z})^{\odot n}$ is the restriction of the unitary operator $(M_z\oplus \overline{M_z})^{\otimes n}$ to an invariant subspace. 
Hence, the maximal spectral type of the former is absolutely continuous with respect to the maximal spectral type of $(M_z\oplus \overline{M_z})^{\otimes n}$, which, by the discussion preceding this lemma, is exactly $(\nu+\overline{\nu})^{*n}$.
Combining these facts, we conclude that the reduced maximal spectral type $\sigma_{\gamma,0}\in M(\Tb)$ of $\gamma$ is absolutely continuous with respect to $\sum_{n=1}^\infty 2^{-n} (\nu+\overline{\nu})^{*n}$.

Suppose for a moment that $G$ is equal to the subgroup $H\cong\mathbb{Z}$. Then $\mathbb{Z}\stackrel{\gamma}{\curvearrowright}(M,\mu)$ (identifying $(M,\mu)\simeq (Z,\zeta)$ via a Borel isomorphism) would be the desired free weakly mixing action in $\Omega^\perp$. Indeed, since $\mathbb{Z}\stackrel{\gamma}{\curvearrowright}(M,\mu)$ is an ergodic action of $\Zb$ it must be free, and since $\sigma_{\gamma,0}\perp \sigma_{\alpha,0}$ for all $\alpha\in\Omega$ it follows from the discussion before this lemma that $\gamma$ is disjoint from each $\alpha\in\Omega$.

Returning to the general case, our aim is to coinduce $\gamma$ to an action of $G$ that preserves the desired properties. Let $h\in G$ denote a generator for $H$. Take a transversal $e\in T\subseteq G$ for $G/H$, i.e., a set such that $G=\bigsqcup_{t\in T}tH$, and consider the action $G\stackrel{\rho}{\curvearrowright} (M^T,\mu^T )$ coinduced from $\gamma$, 
defined for $s\in G$ and $t\in T$ and $x\in M^T$ by $(\rho_sx )(t) = \gamma_{k^{-1}} (x(r))$ where $k$ and $r$ are the unique elements
of $H$ and $T$, respectively, that satisfy $rk=s^{-1}t$.

Since $\gamma$ is free and weakly mixing, so is $\rho$ (see Lemmas~2.1 and 2.2 in \cite{Ioa11}, respectively).
Note that since $H=\langle h\rangle$ is normal in $G$, we have $(\rho_hx)(t) = \gamma_{t^{-1}ht} (x(t))$ for $x\in X^T$ and $t\in T$.
Since the conjugation actions $\Ad t: H\to H$ for $t\in G$ are automorphisms, they must map the generator $h$ to $h$ or to $h^{-1}$. Hence the restriction $\rho'$ of the coinduced action $\rho$ to $H$ is generated by the diagonal product automorphism of $(M^T,\mu^T)$ given by either $\gamma_h$ or $\gamma_{h^{-1}}$ at each coordinate.
Given the facts summarized before the lemma, it follows that the reduced maximal spectral type $\sigma_{\rho',0}$ is absolutely continuous with respect to the (countable) infinite normalized sum of all possible convolutions of $\sigma_{\gamma,0}$ and $\overline{\sigma_{\gamma,0}}$.
Write $\eta=\sum_{n=1}^{\infty}2^{-n}(\nu+\bar{\nu})^{*n}$. Then $\bar{\eta}=\eta$ (since the convolution operation commutes with the conjugation operation for measures, as follows directly from the definitions) and $\eta^{*m}\ll \eta$ for all $m\in\Nb$. Combining these observations together with the fact that $\sigma_{\gamma,0}\ll \eta$, we deduce that $\sigma_{\rho',0}\ll \eta$, i.e., $\sigma_{\rho',0}\ll \sum_{n=1}^{\infty}2^{-n}(\nu+\bar{\nu})^{*n}$.

We have thus shown that for every $\alpha\in \Omega$ the reduced spectral types $\sigma_{\alpha',0}$ and $\sigma_{\rho',0}$ are mutually singular.
Therefore the actions $\alpha'$ and $\rho'$ are (spectrally) disjoint
(note that $\alpha'$ need not
be ergodic, and thus $\sigma_{\alpha',0}$ can have an atom at $1$, but 
since $\rho'$ is ergodic, being a product of weakly mixing actions, its reduced maximal spectral type $\sigma_{\rho',0}$ does not have an atom at $1$). The disjointness of $\alpha'$ and $\rho'$ implies that
$\alpha$ and $\rho$ are themselves disjoint (we identify $\rho$ as an action on $(Z,\zeta)$ using a Borel isomorphism $(M^T,\mu^T)\simeq (Z,\zeta)$). Since $\rho$ is free, by Lemma~\ref{L-conj class} its conjugacy class is dense in $\Act (G,Z,\zeta )$. 
Since weak mixing and disjointness are preserved under conjugacy, we conclude that the free weakly mixing actions which belong to $\Omega^\perp$ form a dense $G_\delta$ subset of $\Act(G,Z,\zeta)$.
\end{proof}

\begin{proposition}\label{P-min WM}
Suppose that $H$ is an amenable subgroup of $G$ that contains a normal infinite cyclic subgroup. 
Let $G \stackrel{\alpha}{\curvearrowright} X$ be a weakly mixing minimal action on the Cantor set $X$ with $M_\alpha (X)\neq\emptyset$. 
Then there is a free strictly ergodic action $H \stackrel{\gamma}{\curvearrowright} X$ 
such that if $G \stackrel{\beta}{\curvearrowright} X$ 
is any action extending $\gamma$ then the diagonal action $G \stackrel{\alpha\times\beta}{\curvearrowright} X\times X$
is minimal and weakly mixing.
\end{proposition}

\begin{proof}
Writing $H \stackrel{\alpha'}{\curvearrowright} X$ for the restriction of $\alpha$ to $H$,
 by Zorn's lemma there exists a nonempty closed set $A\subseteq X$ such that the restriction 
$\tilde{\alpha}$ of $\alpha'$ to $A$ is minimal. 
By the amenability of $H$ there exists a $\mu\in M_H (A)$, which we may take to be ergodic.
This $\mu$ is either atomic (which, in view of minimality, occurs precisely when $A$ is finite) or atomless.

Suppose first that $\mu$ is atomless. 
Write $\sC$ for the countable collection of nonempty clopen subsets of $X$.
Since $\alpha$ is minimal, each measure in $M_\alpha(X)$ has full support.
Since $\alpha$ is weakly mixing, 
it follows by Lemma~\ref{L-WM coset} that for every 
$T = (U_1 , U_2 , V_ 1, V_2 )\in\sC^4$ there exists an $s_T\in G$ such that the set $H_T$ of all $h\in H$ for which $hs_T\in N_\alpha (U_1 , U_2 ) \cap N_\alpha (V_1 , V_2 )$ is infinite.
Note that $\sC^4$ is countable and a countable intersection of dense $G_\delta$ sets in a Polish space is again a dense $G_\delta$.
With $(Z,\zeta )$ denoting as usual a fixed standard atomless probability space,
Lemmas~\ref{L-Gamma} and \ref{L-nontorsion} imply that the set of free ergodic p.m.p.\ actions $\gamma\in\Act(H,Z,\zeta)$ that are disjoint 
from $H\stackrel{\tilde{\alpha}}{\curvearrowright} (A,\mu )$ and satisfy 
$N_{\gamma , \eps} (W_1 , W_2 ) \cap N_{\gamma , \eps} (Z_1 , Z_2 ) \cap H_T \neq\emptyset$
for all nonnull measurable sets $W_1,W_2,Z_1,Z_2\subseteq A$, $\eps>0$, and $T = (U_1 , U_2 , V_ 1, V_2 ) \in\sC^4$ contains a dense $G_\delta$ set (the weak mixing aspect of Lemma~\ref{L-nontorsion} is not needed here but will be used
later in Section~\ref{S-SD II}). In particular there exists such an action $H\stackrel{\gamma}{\curvearrowright} (Z,\zeta )$.
By the Jewett--Krieger theorem \cite{Ros87,Wei25} we may assume
that $Z$ is equal to our Cantor set $X$ and that $\gamma$ is a free strictly ergodic action on $X$ with $\zeta$ being the unique element in $M_\gamma (X)$ (note that the compact metrizable space produced by the construction in \cite{Wei25} is zero-dimensional and by the infiniteness of $H$ cannot have isolated points, so that it must be the Cantor set). This implies by Lemma~\ref{L-product disjoint} that the actions $H\stackrel{\tilde{\alpha}}{\curvearrowright} A$ and $H\stackrel{\gamma}{\curvearrowright} X$
are disjoint.
It then follows by Lemma~\ref{L-min disjoint} that if $G \stackrel{\beta}\curvearrowright X$ is any action
extending $\gamma$ then the diagonal action $G\stackrel{\alpha\times \beta}{\curvearrowright} X\times X$ 
is minimal. 

Let us also check that such a diagonal action $\alpha\times\beta$ with $\beta$ extending $\gamma$ is weakly mixing.
Let $U_1 , U_2 , V_1 , V_2, W_1 , W_2 , Z_1 , Z_2 \subseteq X$
be nonempty clopen sets. 
Write $T=(U_1,U_2,V_1,V_2)$. By our choice of $\gamma$ we can find an $h\in H_T$, i.e., $hs_T\in N_\alpha(U_1,U_2)\cap N_\alpha(V_1,V_2)$, such that for all $\eps >0$ we have
$h\in N_{\gamma,\eps}(\beta_{s_T}W_1,W_2)\cap N_{\gamma,\eps}(\beta_{s_T}Z_1,Z_2)$, 
that is, $\zeta(\beta_{hs_T}W_1\cap W_2)>\zeta(\beta_{s_T}W_1)\zeta(W_2)-\eps$ and $\zeta(\beta_{hs_T}Z_1\cap Z_2)>\zeta(\beta_{s_T}Z_1)\zeta(Z_2)-\eps$.
Since the measure $\zeta$ has full support by the minimality of $\gamma$, 
it follows that $hs_T\in N_\beta(W_1,W_2)\cap N_\beta(Z_1,Z_2)$. Hence
$hs_T\in N_{\alpha\times \beta}(U_1\times W_1, U_2\times W_2)\cap N_{\alpha\times \beta}(V_1\times Z_1, V_2\times Z_2)$.
Since products of clopen sets form a basis for the topology on $X\times Y$, we conclude that $\alpha\times\beta$ is weakly mixing.

Suppose now that $\mu$ is atomic, in which case $A$ is finite. We can no longer
apply the above argument to $\tilde{\alpha}$ itself, but we can apply it instead to 
the diagonal action $H\stackrel{\tilde{\alpha}\times \mathrm{tr}}{\curvearrowright} (A\times X, \mu\times \rho)$, where $\mathrm{tr}$ denotes the trivial action and $X$ is equipped with some atomless Borel probability measure $\rho$. This yields a free strictly ergodic action $H\stackrel{\gamma}{\curvearrowright} X$ which, with respect to its unique invariant Borel probability measure $\zeta$, is disjoint from $H\stackrel{\tilde{\alpha}\times \mathrm{tr}}{\curvearrowright} (A\times X, \mu\times \rho)$ 
and hence also from $H\stackrel{\tilde{\alpha}}{\curvearrowright} (A, \mu)$, and, with $H_T$ as above, 
satisfies $N_{\gamma , \eps} (W_1 , W_2 ) \cap N_{\gamma , \eps} (Z_1 , Z_2 ) \cap H_T \neq\emptyset$
for all nonnull measurable sets $W_1,W_2,Z_1,Z_2\subseteq X$, $\eps>0$, and $T = (U_1 , U_2 , V_ 1, V_2 ) \in\sC^4$.
We can now conclude, arguing as before, that for every action $G \stackrel{\beta}\curvearrowright X$ 
extending $\gamma$ the diagonal action $G\stackrel{\alpha\times \beta}{\curvearrowright} X\times X$ is minimal
and weakly mixing.
\end{proof}

We remark that the action $\beta$ extending $\gamma$ that is universally quantified in Proposition~\ref{P-min WM} need not 
admit any invariant Borel probability measures, even though $\gamma$ does admit one.
In our applications in Proposition~\ref{P-top free} and Lemma~\ref{L-SD free product}
we will be handling $\beta$ that do satisfy $M_\beta (X)\neq\emptyset$,
specifically ones that will be given to us through the use of Proposition~\ref{P-extending}.
In this case $M_{\alpha\times\beta }(X\times X)$ will be nonempty as it will contain at least one product measure.
What is interesting in the proof of Proposition~\ref{P-min WM}, especially in the context of these applications,
is the auxiliary nature of the use of p.m.p.\ ergodic theory
via the the amenability of the subgroup $H$. It could very well happen that the $\alpha'$-minimal set 
$A$ on which we apply this ergodic theory is null for all $\alpha$-invariant 
Borel probability measures. 

If $H$ is itself an infinite cyclic subgroup of $G$ in Proposition~\ref{P-min WM}, for example if we replace
the original $H$ with the hypothesized normal infinite cyclic subgroup, then the conclusion is valid with
no extra assumptions on $H$. One may wonder whether this special form of the proposition is sufficient 
for all practical purposes, but the problem in our applications in Sections~\ref{S-spaces of actions} and \ref{S-SD I}
is that we do not have a mechanism for extending an
action of the normal infinite cyclic subgroup of $H$ to all of $G$, or more specifically of first extending the action to $H$.
So the general version will be necessary.

We now turn to the proof of Proposition~\ref{P-extending}, which as mentioned above provides the dynamical data
that will be fed into Proposition~\ref{P-min WM} in subsequent sections.
We will need the following asymptotic freeness result involving random permutation matrices
due to Collins and Dykema \cite{ColDyk11}. It has the consequence that if
we have actions of two groups $G$ and $H$ on a common finite set of large cardinality 
then we can randomly conjugate one of the actions to obtain, with high probability, an action of $G*H$ that is almost
free on a given set of words whose individual letters act almost freely. 
This principle also works more generally for sofic approximations, in which the axioms for a group action are 
allowed to be corrupted on a proportionally small part of the set. In fact the original motivation in \cite{ColDyk11} was to establish
the soficity of certain amalgamated free products.
Our application in Lemma~\ref{L-RandMatApp} below has a similar sofic spirit.

\begin{theorem}\cite[Theorem~2.1]{ColDyk11}\label{T-CD}
Write $S_n$ for the set of $n\times n$ permutation matrices in $M_n$ and $\tr$ for the 
normalized trace on $M_n$. Then for all $n,d\in \Nb$ 
and $A_1 , \dots , A_{2d} \in S_n$ one has
\begin{align*}
\lefteqn{\frac{1}{n!} \sum_{U\in S_n}
\tr (A_1 (UA_2 U^* ) A_3 (U A_4 U^* )\cdots A_{2d-1} (U A_{2d} U^* ))} \hspace*{50mm} \\
&\hspace*{10mm} < r_d \max_{1\leq i\leq 2d} \tr (A_i ) + \frac{t_d}{n} .
\end{align*}
where $r_d , t_d > 0$ are constants depending only on $d$.
\end{theorem}

Note that if we view an $n\times n$ permutation matrix $A$ as a permutation of a set with $n$ elements then its normalized trace 
is exactly the number of fixed points of this permutation divided by $n$.

\begin{lemma}\label{L-RandMatApp}
Suppose that $G$ is residually finite and that $H$ is countably infinite and amenable.
Let $\eps>0$ and let $F\subseteq G*H$ be a nonempty finite set. 
Define $F_H$ to be the set of all elements in $H$ (including $e\in H$) that appear as letters among 
the reduced words in $G*H$ that belong to $F^{-1}F$. Then there exist a $(F_H,\eps)$-invariant finite set $K\subseteq H$, an action $G\stackrel{\rho}{\curvearrowright}K$, and permutations $\{\sigma_s\}_{s\in F_H}$ and $\tau$ of $K$
such that if $w=h_1g_1\cdots h_kg_k$ is a nontrivial reduced word in $G*H$ that belongs to $F^{-1}F$ then
\[
\tr(\sigma_{h_1}(\tau\rho_{g_1}\tau^{-1})\sigma_{h_2}(\tau\rho_{g_2}\tau^{-1})\cdots\sigma_{h_k}  (\tau\rho_{g_k}\tau^{-1}))<\eps.
\]
The same conclusion holds for words in $F^{-1}F$ that start and end with elements both in $G$, both in $H$,
or in $G$ and $H$, respectively.
\end{lemma}

\begin{proof}
We assume, without loss of generality, that $\eps < 1/|F^{-1}F|$. Write $F_G$ for the set of all elements of $G$ (including $e\in G$) that appear as letters in reduced words in $G*H$ that belong to $F^{-1}F$. Since $G$ is residually finite, there exists a finite-index normal subgroup $N$ of $G$ that does not contain any element of $F_G^2\setminus \{e\}$.
Let $\ell(w)$ denote the length of a reduced word $w$ in $G*H$. For $d\in \Nb$ let $r_d,t_d\ge 1$ be the constants given by Theorem~\ref{T-CD}, and define $L=\max_{w\in F^{-1}F}\lceil \ell(w)/2 \rceil$, $R=\max_{1\leq d\leq L} r_d$, and $T=\max_{1\leq d\leq L} t_d$.
Since $H$ is amenable and infinite, there exists a finite set $K\subseteq H$ such that
\begin{enumerate}
\item $K$ is $(F_H^2,\eps^2 /2R)$-invariant, and in particular $(F_H,\eps)$-invariant,
\item $|K| > 2T/\eps^2$,
\item $|K|! > |F^{-1}F|/(1 - |F^{-1}F|\eps )$, and
\item $m := |K|/[G : N]$ is a positive integer.
\end{enumerate} 
Consider the diagonal action $G\curvearrowright G/N \times \{ 1,\dots , m\}$ 
coming from left translation on the first factor and the trivial action on the second.
By (iv) there is a bijection $\theta: K\to G/N \times \{ 1,\dots , m\}$.
Define an action $G\stackrel{\rho}{\curvearrowright} K$ by $\rho_sh = \theta^{-1}s\theta h$ for all $s\in G$ and $h\in K$.
By our choice of $N$, we conclude that $\tr(\rho_s)=0$ whenever $s\notin N$.
For every $s\in F_H$ define a permutation
$\sigma_s$ of $K$ by setting $\sigma_sh = sh$ for $h\in s^{-1} K \cap K$ and on $K\setminus s^{-1} K$
taking $\sigma_s$ to be an arbitrary bijection with $K\setminus sK$ that agrees with $\sigma_{s^{-1}}^{-1}$ if $s^{-1}\in F_H$ and $\sigma_{s^{-1}}$ was already defined (this can be guaranteed if the maps $\sigma_s$ are defined recursively, since $F_H$ is a finite set). 
By (i), for all $s\in F_H\setminus \{e\}$ and $s_1,s_2 \in F_H$ such that $\sigma_{s_1}\sigma_{s_2} \ne \id_K$ we have 
$\tr(\sigma_s) \le |K\setminus s^{-1}K|/|K| < \eps^2/2R$ and
\begin{equation}
\tr(\sigma_{s_1}\sigma_{s_2})\leq \frac{|K\setminus \bigcap_{s\in (F_H)^2} s^{-1}K|}{|K|} < \frac{\eps^2}{2R}. \label{E-trace prod}
\end{equation}

Let $w=h_1g_1\cdots h_kg_k\in F^{-1}F$ be a nontrivial word in its reduced form. Identifying the permutations $\sigma_s$ with elements in the set $S_K$ of permutation matrices in the matrix algebra $M_K$ with entries indexed by pairs in $K$,
we apply Theorem~\ref{T-CD} to get
\[
\frac{1}{|K|!} \sum_{U\in S_K}
\tr (\sigma_{h_1} (U\rho_{g_1} U^* ) \sigma_{h_2} (U \rho_{g_2} U^* )\cdots \sigma_{h_k} (U \rho_{g_k} U^* )) < r_k \frac{\eps^2}{2R} + \frac{t_k}{|K|} \stackrel{{\rm (ii)}}{<} \eps^2. 
\]
Since we are averaging positive real numbers, an easy combinatorial argument shows that for at least $\lfloor |K|!(1-\eps) \rfloor$ elements $U\in S_K$ one has that
\[
\tr (\sigma_{h_1} (U\rho_{g_1} U^* ) \sigma_{h_2} (U \rho_{g_2} U^* )\cdots \sigma_{h_k} (U \rho_{g_k} U^* ))  <\eps. 
\]
If $w=g_1h_1\cdots g_kh_k\in F^{-1}F$ is a nontrivial word in its reduced form, then analogously we can find at least $\lfloor |K|!(1-\eps)\rfloor$ elements $U\in S_K$ for which
\[
\tr (\sigma_{h_k} (U\rho_{g_1} U^* ) \sigma_{h_1} (U \rho_{g_2} U^* )\cdots \sigma_{h_{k-1}} (U \rho_{g_k} U^* ))  <\eps,
\] 
and hence also
\[\tr ((U\rho_{g_1} U^* ) \sigma_{h_1} (U \rho_{g_2} U^* )\cdots \sigma_{h_{k-1}} (U \rho_{g_k} U^*) \sigma_{h_k} )  <\eps.\]
If $w\in F^{-1}F\setminus\{e\}$ has reduced form $w=g_1h_1\cdots g_{k-1}h_{k-1}g_{k}$ and $g_kg_1\neq e$, then $g_kg_1\in (F_G)^2\setminus\{e\}$ and so does not belong to $N$. Hence
\begin{align*}
\lefteqn{\frac{1}{|K|!} \sum_{U\in S_K}
\tr ((U\rho_{g_1} U^* ) \sigma_{h_1} (U \rho_{g_2} U^* )\cdots (U \rho_{g_{k-1}} U^* )\sigma_{h_{k-1}} (U \rho_{g_k} U^* ))} \hspace*{10mm} \\
\hspace*{10mm} &=\frac{1}{|K|!} \sum_{U\in S_K}
\tr ((U\rho_{g_kg_1} U^* ) \sigma_{h_1} (U \rho_{g_2} U^* )\cdots (U \rho_{g_{k-1}} U^* )\sigma_{h_{k-1}}) \\
&< r_{k-1}\frac{\eps^2}{2R} + \frac{t_{k-1}}{|K|}\stackrel{\text{(ii)}}{<}\eps^2.
\end{align*}
We conclude again that for at least $\lfloor |K|!(1-\eps) \rfloor$ elements $U\in S_K$ one has that
\[
\tr ((U\rho_{g_1} U^* ) \sigma_{h_1} (U \rho_{g_2} U^* )\cdots \sigma_{h_{k-1}} (U \rho_{g_k} U^* ))  <\eps. 
\]
If $g_kg_1=e$, then we need to bound an expression of the form
\begin{align*}
\frac{1}{|K|!} \sum_{U\in S_K}
\tr (\sigma_{h_1} (U \rho_{g_2} U^* )\cdots (U \rho_{g_{k-1}} U^* )\sigma_{h_{k-1}}).
\end{align*}
It will become clear how to do this once we treat the case of nontrivial words in $F^{-1}F$ that have reduced form that begins and ends with elements of $H$, i.e., $w\in F^{-1}F$ of the form $w=h_1g_1\cdots h_{k-1}g_{k-1}h_k$. In this scenario, if it happens that $\sigma_{h_k}\sigma_{h_1} \ne \id_K$ then we treat $\sigma_{h_k}\sigma_{h_1}$ as a single permutation and apply (\ref{E-trace prod}) and 
Theorem~\ref{T-CD} to obtain
\begin{align*}
\lefteqn{\frac{1}{|K|!} \sum_{U\in S_K}
\tr ( \sigma_{h_1} (U \rho_{g_1} U^* )\cdots \sigma_{h_{k-1}} (U \rho_{g_{k-1}} U^* )\sigma_{h_k})} \hspace*{10mm} \\
\hspace*{10mm} &=\frac{1}{|K|!} \sum_{U\in S_K}
\tr (\sigma_{h_k}\sigma_{h_1} (U \rho_{g_1} U^* )\cdots \sigma_{h_{k-1}} (U \rho_{g_{k-1}} U^* )) \\
&< r_{k-1}\frac{\eps^2}{2R} + \frac{t_{k-1}}{|K|}<\eps^2.
\end{align*}
If on the other hand $\sigma_{h_k}\sigma_{h_1}=\id_K$, then we have fewer (i.e., $2k-3$) permutations at play
in the application of Theorem~\ref{T-CD} to a word beginning and ending with elements of $G$. 
We can continue in this way recursively, with a smaller amount of permutations at each step until either ``cancellation'' of the above nature does not occur, or we end up with a single permutation (some $\rho_g$ with $g\in F_G$ or some $\sigma_h$ with $h\in F_H$), in which case the inequality is clearly satisfied.

Now for each $w\in F^{-1}F\setminus \{e\}$ we found at least $\lfloor |K|!(1-\eps) \rfloor$ elements $U\in S_K$ that satisfy a certain desired inequality. 
We claim that there exists $\tau\in S_K$ that achieves this simultaneously for all $w\in F^{-1}F\setminus \{e\}$.
This is a consequence of the following counting argument: for each nontrivial $w\in F^{-1}F$ there are at most $|K|!-\lfloor |K|!(1-\eps) \rfloor$ elements in $S_K$ that do not satisfy the desired inequality associated with $w$. However, from (iii) it follows that $|F^{-1}F|\cdot (|K|!-\lfloor |K|!(1-\eps)\rfloor ) <|K|!$.
As there are $|K|!$ permutations in $S_K$, we obtain the existence of
a $\tau\in S_K$ satisfying the condition in the corollary statement.
\end{proof}

Since an action of an infinite amenable group on the Cantor set has the URP if and only if it is essentially free (see \cite[Theorem~3.6]{GarGefGesKopNar24}), the following is a strengthening of Theorem~\ref{T-amenable} in the case of the Cantor set. 

\begin{proposition}\label{P-extending}
Suppose that $X$ is the Cantor set, $G$ is residually finite, $H$ is a countably infinite amenable group, and $H\curvearrowright X$ an essentially free minimal action that has comparison.
Let $F$ be a finite subset of $G*H$. Then there is an action  
$G*H\curvearrowright X$ extending $H\curvearrowright X$ (viewing $H$ as a subgroup of $G*H$) such that
\begin{enumerate}
\item the action is $(O_1 , O_2 , E)$-squarely divisible for all nonempty clopen sets $O_1 , O_2 \subseteq X$
and finite sets $E\subseteq G*H$ containing $e$,

\item $M_{G*H} (X) = M_H (X)$, and

\item there exists a nonempty clopen set $A\subseteq X$ such that $(F,A)$ is a tower.
\end{enumerate}
\end{proposition}

\begin{proof}
Let $e\in F_G\subseteq G$ and $e\in F_H\subseteq H$ be the finite sets as defined in the statement of 
Lemma~\ref{L-RandMatApp} and the beginning of its proof. Applying Lemma~\ref{L-RandMatApp}, we find a $(F_H,\eps)$-invariant finite set $K\subseteq H$, an action $G\stackrel{\rho}{\curvearrowright}K$, and permutations $\{ \sigma_s \}_{s\in F_H}$ 
and $\tau$ of $K$ such that if $w=h_1g_1\cdots h_kg_k$ is a nontrivial reduced word in $F^{-1}F$ then the permutation
\begin{gather}\label{E-word}
\rho_w := \sigma_{h_1} (\tau\rho_{g_1}\tau^{-1} ) \sigma_{h_2} (\tau\rho_{g_2}\tau^{-1} )\cdots \sigma_{h_k} (\tau\rho_{g_k}\tau^{-1} )
\end{gather}
satisfies $\tr(\rho_w)<\eps$ for some $\eps<1/(|F^{-1}F|+L)$ where $L=\max_{w\in F^{-1}F} \ell(w)$ and $\ell(w)$ denotes the length of the reduced word $w\in F^{-1}F$.
The same conclusion holds for words in $F^{-1}F$ that start and end with elements both in $G$, both in $H$, or in $G$ and $H$, respectively.
We also recall from the proof of Lemma~\ref{L-RandMatApp} that $\sigma_sh=sh$ whenever $s\in F_H$ and $h\in s^{-1}K\cap K$, and that for some finite-index normal subgroup $N$ of $G$ one has $\rho_g=\id_K$ for every $g\in N$.

For $w=h_1g_1\cdots h_kg_k\in F^{-1}F\setminus \{e\}$ and $1\leq p\leq k$ define the permutation 
$\varphi_{w,p}=(\tau\rho_{g_p}\tau^{-1})\sigma_{h_{p+1}}(\tau\rho_{g_{p+1}}\tau^{-1})\cdots \sigma_{h_k}(\tau\rho_{g_k}\tau^{-1})$ of $K$.
It is straightforward to check that if $h\in \bigcap_{p=1}^{k}\varphi_{w,p}^{-1}(K^{F_H})$, then
\[
\rho_wh=h_1(\tau\rho_{g_1}\tau^{-1})h_2(\tau\rho_{g_2}\tau^{-1})\cdots h_k(\tau\rho_{g_k}\tau^{-1})h,
\]
where by $h_1 , \dots , h_k$ we mean the action of these elements by left translation.
One may define similarly $\varphi_{w,p}$ for any $w\in F^{-1}F\setminus \{e\}$, with the range of $p$ depending on the length of $w$ but uniformly bounded by $L$. Using $\Fix (\cdot )$ to denote the fixed-point set, consider the intersection
\[ 
W:= \bigg(\bigcap_{w\in F^{-1}F}\,\bigcap_{p=1}^k \varphi_{w,p}^{-1} (K^{F_H})\bigg) 
\cap \bigg(\bigcap_{w\in F^{-1}F\setminus\{e\}} \Fix (\rho_w )^c \bigg).
\]
Observe that, for every $h\in W$, when we apply the permutations defining $\rho_w$ in sequence,
each of the permutations $\sigma_{h_i}$ acts as translation by $h_i$. Moreover, since the $\varphi_{w,p}$ are permutations and $K$ is $(F_H,\eps)$-invariant, the complement of $W$ consists of at most $|L||F^{-1}F|\eps|K|+|F^{-1}F|\eps|K|$ elements, which, by our choice of $\eps$, is less than $|K|$. Hence $W$ is nonempty and we may choose an $h_0\in W$.

By the topological freeness of the action $H\curvearrowright X$ we can find a nonempty clopen set $A\subseteq X$ such that $(K,A)$ is a tower.
Define an action $G*H\curvearrowright X$ extending the $H$-action by 
declaring it on an element $g\in G$ to be given by
\[
ghx = \tau\rho_g\tau^{-1}hx
\]
for all $x\in A$ and $h\in K$ and $gx = x$ for all $x\in X\setminus KA$.
Thus the elements of $G$ act on $KA$ by permuting 
the levels of the tower $(K,A)$ in a way that preserves the $K$-orbits of points in $A$, i.e., $g K x = K x$ for all $x\in A$. This has the effect $M_{G*H}(X)=M_{H}(X)$. Indeed if we are given a $\mu\in M_H (X)$, a $g\in G$, and a Borel set $B\subseteq X$
then for every $k\in K$ we have $g (B\cap k A)=\tau\rho_g\tau^{-1}(k) (k^{-1}B\cap A)$ and hence,
using the $H$-invariance of $\mu$,
\[\mu(g (B\cap k A))=\mu(\tau\rho_g\tau^{-1}(k) (k^{-1}B\cap A))=\mu(k^{-1}B\cap A)=\mu(B\cap k A) ,\] 
so that writing $B=(B\setminus K A)\sqcup (\bigsqcup_{k\in K} B\cap k A)$
and noting that $g (B\setminus K A)=B\setminus K A$ we obtain $\mu(gB) = \mu(B)$.
It is moreover straightforward to verify that our choice of $h_0$ means that for every nontrivial element $w\in G*H$ belonging to $F^{-1}F$ and every $x\in A$ we have
\[w(h_0x)=\rho_{w}(h_0)x.\]
This, together with the fact that $h_0\in \bigcap_{w\in F^{-1}F\setminus \{e\}} \Fix (\rho_w )^c$, implies that $(F,h_0 A)$ forms a tower. Thus conditions (ii) and (iii) of the lemma statement are fulfilled. 

It remains to verify that the action $G*H\curvearrowright X$ satisfies condition (i). For this purpose we may regard
it as an action of $D*H$ where $D$ is the finite group $G/N$ (since $\rho_g=\id_K$ for $g\in N$). Let $E$ be a finite subset of $D*H$ containing $e$.
Then there are a finite set $L\subseteq H$ containing $e$ and an $n\in\Nb$ such that
$E \subseteq (D\cup L)^n$. Set $W = (KK^{-1}  \cup L)^n$.

Let $x\in X$. If $x\in h_1 A$ for some $h_1 \in K$ then 
for every $s\in D$ there exists an $h_2\in K$ (namely $\tau\rho_{s}\tau^{-1}(h_1)$) such that $sx = h_2 h_1^{-1} x$, and if $x\notin KA$ then $sx = x$ for every $s\in D$. 
Therefore $Dx \subseteq KK^{-1} x$ and hence $Ex \subseteq Wx$.
This means that for any set $Y\subseteq X$ one has
\begin{gather}\label{E-EYWY}
EY\subseteq WY \hspace*{4mm}\text{and}\hspace*{4mm} Y^W\subseteq Y^E.
\end{gather}
By Theorem~\ref{T-amenable}, the action $H\curvearrowright X$ is $(O_1,O_2,W)$-squarely divisible.
Let $\{V_{i,j}\}_{i,j=1}^{n}$ and $U$ be the sets witnessing this square divisibility (with associated boundary set $B$ as in Definition~\ref{D-SD}).
Then the sets $V_{i,j}$ will automatically be $E$-disjoint by equation~(\ref{E-EYWY}).
Again by equation~(\ref{E-EYWY}), the associated boundary set $B=\overline{V}\cap ((V\cap \overline{U}^c)^W)^c$ contains the set $\overline{V}\cap ((V\cap \overline{U}^c)^E)^c$.
It clearly follows that the sets $\{V_{i,j}\}_{i,j=1}^{n}$ also witness $(O_1,O_2,E)$-squarely divisibility for $D*H\curvearrowright X$. 
\end{proof}

\section{Spaces of actions}\label{S-spaces of actions}

We introduce here certain subspaces of $\Act (G,X)$ 
that will play a central role in subsequent sections. 
Recall that $\Act (G,X)$ denotes the space of all actions $G\curvearrowright X$ with the topology of elementwise compact-open convergence.

In the case that $X$ is the Cantor set, which will be the main focus of the rest of the paper, 
we make the following series of observations about $\Act (G,X)$ that will help us establish Proposition~\ref{P-Polish}.
Write $\sC$ for the countable collection 
of nonempty clopen subsets of $X$.

First we remark that set of minimal actions in $\Act(G,X)$ is a $G_\delta$. Indeed for every $A\in\sC$ consider the set $\sW_{A}$
of all $\alpha\in \Act (G,X)$ for which there exists a finite set $F\subseteq G$ such that $\bigcup_{s\in F} \alpha_s A = X$. Then $\sW_{A}$ is open and the set of minimal actions is equal to 
$\bigcap_{A\in\sC} \sW_{A}$, which is a $G_\delta$.

As observed in \cite[Section~4]{ConJacKerMarSewTuc18}, the set of free actions in $\Act (G,X)$ is a $G_\delta$. 
The same is true for the set of topologically free actions:
fixing an enumeration $s_1 , s_ 2, \ldots$ of $G\setminus \{ e \}$, we observe that 
for every $n\in\Nb$ and $A\in\sC$
the set $\sW_{n,A}$ of all actions $\alpha\in\Act (G,X)$ such that there exists an $x\in A$ for which $\alpha_{s_n} x \neq x$
is open, and taking the intersection $\bigcap_{n\in\Nb} \bigcap_{A\in\sC} \sW_{n,A}$ we obtain
precisely the set of all topologically free actions.

To see that the set of weak mixing actions in $\Act (G,X)$ is also a $G_\delta$, for every $A_1,A_2,B_1,B_2\in\sC$ write $\sW_{A_1,A_2,B_1,B_2}$ for the set of all $\alpha\in \Act(G,X)$ for which there exists $s\in G$ such that $\alpha_sA_1\cap A_2\neq\emptyset$ and $\alpha_sB_1\cap B_2\neq \emptyset$. Then $\sW_{A_1,A_2,B_1,B_2}$ is open, and the set of all weak mixing actions is equal to $\bigcap_{A_1,A_2,B_1,B_2\in\sC}\sW_{A_1,A_2,B_1,B_2}$, which is a $G_\delta$.

Finally we note that the set of actions $\alpha\in \Act(G,X)$ for which $M_\alpha(X)\neq\emptyset$ is closed in $\Act(G,X)$, and hence also a $G_\delta$. To verify this, suppose that $(\alpha_j)_{j\in\mathbb{N}}$ is a sequence of elements in $\Act(G,X)$ which each admit an invariant Borel probability measure and converge to some $\alpha\in \Act(G,X)$. We must show that $M_\alpha(X)\neq\emptyset$. For each $j$ pick a $\mu_j\in M_{\alpha_j}(X)$ and 
take a weak$^*$ cluster point $\mu\in M(X)$ of the sequence $(\mu_j)_{j\in\mathbb{N}}$, which exists
by weak* compactness. We may assume, by passing to a subsequence if necessary, that $\mu_j$ converges to $\mu$. We will show that $\mu\in M_\alpha(X)$. Let $A\in \sC$, where $\sC$ is as above, and let $s\in G$. Since the Borel $\sigma$-algebra of $X$ is generated by $\sC$, it suffices to show that $\mu(\alpha_sA)=\mu(A)$. Since $\alpha_j$ eventually belongs to 
the basic open set $U_{\alpha,\{s\},\sP}$, where $\sP$ is any clopen partition containing $A$, we may assume that $\alpha_{j,s}A=\alpha_sA$ for all $j\in\mathbb{N}$. Given that $1_A\in C(X)$ we now have
\begin{align*}
\lim_j \mu_j(A)=\lim_j \int_X 1_A d\mu_j=\int_X 1_A d\mu=\mu(A)
\end{align*}
while
\begin{align*}
\lim_j \mu_j(A)=\lim_j \mu_j(\alpha_{j,s}A)
&=\lim_j \mu_j(\alpha_{s}A)\\
&=\lim_j \int_X 1_{\alpha_sA} d\mu_j=\int_X 1_{\alpha_sA} d\mu= \mu(\alpha_sA).
\end{align*}
This implies $\mu(\alpha_sA)=\mu(A)$, as desired.

\begin{definition}\label{D-spectrally aperiodic}
A transformation $T:X\to X$ is {\it spectrally aperiodic} if there is no clopen subset $A\subseteq X$ and $n\in\Nb$ such that the sets $A,TA , \dots ,T^{n} A$ partition $X$.
\end{definition}

\begin{definition}\label{D-spaces}
Write $\WA (G,X)$ for the set of all minimal topologically free actions in $\Act (G,X)$ that are weakly mixing
and satisfy $M_G (X) \neq\emptyset$. For $d\in\Nb$ write $\Astar (F_d ,X)$ for the set of all topologically free
actions in $\Act (F_d ,X)$ with $M_G(X)\ne \emptyset$ that are strictly ergodic and spectrally aperiodic on each standard generator.
\end{definition}

\begin{remark}
Note that every $\alpha \in \Astar(F_d,X)$ is itself strictly ergodic. 
\end{remark}

\begin{proposition}\label{P-Polish}
Suppose that $G$ is infinite and $X$ is the Cantor set. Then $\WA (G,X)$ and $\Astar (F_d , X)$ for $d\in\Nb$
are a nonempty $G_\delta$ subsets of the respective spaces $\Act (G,X)$ and $\Act (F_d , X)$. 
In particular, they are Polish spaces.
\end{proposition}

\begin{proof}
The fact that $\WA (G,X)$ is a $G_\delta$ follows from the discussion before Definition~\ref{D-spectrally aperiodic}.

By the discussion at the beginning of the section, the minimal transformations form a $G_\delta$ set in $\Act(\Zb,X)$. So do the uniquely ergodic actions in $\Act(\Zb,X)$, as shown in \cite{Hoc08} (see Theorem~1.3 and the last part of the proof of Proposition~6.7 therein). Thus the set of
strictly ergodic actions in $\Act(\Zb,X)$ is a $G_\delta$. 
Moreover, the set of transformations $T\in\Act(\Zb,X)$ for which there exist a clopen set $A\subseteq X$ and an $n\in\Nb$ so that the sets $A,TA , \dots ,T^{n} A$ partition $X$ is clearly open, and the set of spectrally aperiodic actions is its complement and hence is a $G_\delta$. Since the restriction map 
$\Act(F_d,X)\to \Act(\langle a \rangle,X) = \Act(\Zb,X)$ is continuous for each of the standard generators 
$a$ of $F_d$, we can intersect the inverse images of the above $G_\delta$ sets under these maps
and then further intersect this with the set of topologically free actions in $\Act (F_d , X)$ satisfying $M_{F_d} (X)\neq\emptyset$ 
(which is a $G_\delta$ by the discussion before Definition~\ref{D-spectrally aperiodic})
so as to express $\Astar (F_d , X)$ as the intersection of finitely many $G_\delta$ sets, which is again 
a $G_\delta$ set.

As explained in Example~\ref{E-notSD}, every countably infinite group admits
a free minimal mixing action on the Cantor set with $M_G (X)\neq\emptyset$, and so $\WA (G,X)$ is nonempty.
The examples in Section~\ref{S-examples} show that $\Astar (F_d , X)$ is nonempty for $d\in\Nb$ 
(the construction in Section~\ref{S-examples} is done in the case $d=2$ but can be carried out similarly for other $d$).
\end{proof}

\begin{proposition}\label{P-top free}
Suppose $X$ is the Cantor set.
Suppose that $G$ is residually finite and that $H$ is amenable and contains a normal infinite cyclic subgroup.
Then $\WA(G*H,X)$ is a dense $G_\delta$ set inside the space $\sM$ of all weakly mixing minimal actions with $M_{G*H} (X) \neq\emptyset$. 
\end{proposition}

\begin{proof}
The space $\sM$ is a $G_\delta$ subset of $\Act (G*H,X)$ by the discussion preceding 
Definition~\ref{D-spectrally aperiodic}, and by definition $\WA(G*H,X)$ is the set of all actions in $\sM$
that satisfy the additional requirement of topological freeness.

As we did at the beginning of the section, fix an enumeration $s_1 , s_ 2, \ldots$ of $G*H\setminus \{ e \}$ and for every $n\in\Nb$ and
$A$ in the countable collection $\sC$ of nonempty clopen subsets of $X$ write $\sW_{n,A}$ for the set of 
all actions $\alpha\in\sM$ such that there exists an $x\in A$ for which $\alpha_{s_n} x \neq x$.
Then $\sW_{n,A}$ is an open set in $\sM$ and $\WA(G*H,X)=\bigcap_{n\in\Nb}\bigcap_{A\in\sC}\sW_{n,A}$. We will show that each  $\sW_{n,A}$ is dense in $\sM$, and 
the density of $\WA(G*H,X)$ in $\sM$ will then be a consequence of the Baire category theorem.

To verify the density of $\sW_{n,A}$ in $\sM$, let $\alpha\in\sM$. 
Take a free minimal action $H\curvearrowright  X$ as given by Proposition~\ref{P-min WM}. 
This action has comparison by \cite[Theorem~A]{Nar24} together with \cite[Theorem~6.1]{KerSza20}.
By Proposition~\ref{P-extending} (ignoring the first condition in its statement concerning square divisibility) there is an action
$G*H\stackrel{\beta}{\curvearrowright} X$ extending $H\curvearrowright X$ such that
$M_{\beta} (X)$ is equal to $M_H (X)$ (which is nonempty by the amenability of $H$)
and such that there exists a $y\in X$ with $\beta_{s_n}y\neq y$.

The diagonal action $G*H\stackrel{\alpha\times\beta}{\curvearrowright} X\times X$
is minimal and weakly mixing by our choice of $H\curvearrowright X$ via Proposition~\ref{P-min WM}. 
By assumption $M_\alpha(X)\neq \emptyset$, and so $M_{\alpha\times \beta} (X\times X)$
contains at least one product measure and in particular is nonempty.

By Proposition~\ref{P-extension approximation} (applied with respect to the canonical projection $h: X\times X\to X$ onto the first coordinate) we can conjugate $\alpha\times\beta$ 
to an action $\rho=g\circ(\alpha\times\beta)\circ g^{-1}$ on $X$ which is as close as we wish to $\alpha$ in such a way that the image of $A\times X$
under the conjugacy $g$ is equal to $A$. Since any point $x\in g(A\times \{ y \})\subseteq A$ 
has the property that $\rho_{s_n} x \neq x$, we have $\rho\in\sW_{n,A}$, establishing density.
\end{proof}

In contrast to actions in $\Astar (F_d , X)$, which are minimal on generators by definition,
an action in $\WA (F_d ,X)$, despite being minimal, is often very far from being minimal on generators,
as demonstrated by the following proposition.

\begin{proposition}\label{P-trivial factor}
Suppose $X$ is the Cantor set. Let $d\geq 2$. 
The set $\sV$ of all actions $F_d = \langle a_1 , \dots , a_d \rangle\curvearrowright X$ in $\WA (F_d ,X)$ such that 
each of the restriction actions $\langle a_k \rangle \curvearrowright X$ for $k=1,\dots ,d$ factors onto
the trivial action on the Cantor set is comeagre.
\end{proposition}

\begin{proof}
As usual denote by $\sC$ the countable collection of nonempty clopen subsets of $X$. For each $k=1,\ldots ,d$ and $A\in\sC$ write $\sW_{k,A}$ for the set of all actions $\alpha$ in $\WA (F_d ,X)$ such that if
$\alpha_{a_k} A = A$ then $A$ can be partitioned into two clopen subsets that are fixed by $\alpha_{a_k}$, i.e., there exist disjoint sets $A_1,A_2\in \sC$ such that $A=A_1\sqcup A_2$ and $\alpha_{a_k} A_i=A_i$, for $i=1,2$ 
(if $\alpha_{a_k}A\neq A$ then $\alpha$ automatically belongs to $\sW_{k,A}$).
The set $\sW_{k,A}$ is clearly open by the definition of the topology on $\WA (F_d ,X)$. The set $\sW := \bigcap_{k=1}^d\bigcap_{A\in\sC} \sW_{k,A}$ is a $G_\delta$. 
We show that $\sW\subseteq \sV$. Let $\alpha\in\sW$. Given $1\leq k\leq d$, we set $\sP_{k,0}=\{X\}$ and 
for $n\in\Nb$ recursively find a partition $\sP_{k,n}$ of cardinality $2^n$ refining $\sP_{k,n-1}$ by partitioning 
every set in $\sP_{k,n-1}$ into two disjoint nonempty clopen sets each of which is fixed by $\alpha_{a_k}$.
Take the C$^*$-subalgebra $A_k \subseteq C(X)$ generated 
by the indicator functions of the members of the partitions $\sP_{k,n}$ over all $n\in\Nb$. Set $A_{k,n}=C^*(\{1_{U}: U\in \sP_{k,n}\})$. Then $A_k=\varinjlim_{n}A_{k,n}$, with the canonical inclusions as connecting maps.
It is then straightforward to check that $A_{k,n}\cong \mathbb{C}^{2^n}$, and, under a suitable isomorphism, the inclusion maps correspond to the injective maps $\varphi_{n+1,n}: \mathbb{C}^{2^n}\to \mathbb{C}^{2^{n+1}}$ given by $z\mapsto (z,z)$.
Thus $A_k\cong \varinjlim_n(\mathbb{C}^{2^n}, \varphi_{n+1,n})\cong C(X)$, see \cite[Example~III.2.5]{Dav96}.
From this we obtain an equivariant factor map $(X,\alpha_{a_k})\to (X,\tr)$. It follows that $\sW\subseteq \sV$. 

It remains to show that for every $A\in \sC$, the set $\sW_A=\bigcap_{k=1}^{d}\sW_{k,A}$ is dense in $\WA (F_d ,X)$. By the Baire category theorem it would then follow that $\sW$ is a dense $G_\delta$ set contained in $\sV$, so that $\sV$ is comeagre. Let $A\in\sC$ and let $\alpha\in\WA (F_d ,X)$. We will first construct an action
$\alpha^{(1)} \in\bigcap_{k=2}^d \sW_{k,A}$ that is as close to $\alpha$ as we wish. 
As $\langle a_1 \rangle \cong\Zb$ is amenable,
by Proposition~\ref{P-min WM} there is an action $\langle a_1 \rangle\stackrel{\gamma}{\curvearrowright} X$ such that,
letting $F_d \stackrel{\beta}{\curvearrowright} X$ be the action determined by $\beta_{a_1} x = \gamma_{a_1} x$ for all $x\in X$
and $\beta_{a_k} x = x$ for all $x\in X$ and $k=2,\dots , d$, the diagonal action $\alpha\times\beta$ of $F_d$ on $X\times X$
is minimal and weakly mixing. Since $\alpha$ is topologically free, so is $\alpha\times\beta$. Since every
Borel probability measure on $X$ that is $\gamma$-invariant is also $\beta$-invariant, $M_{\alpha\times\beta} (X\times X)$
contains at least one product measure and hence is nonempty. Thus $\alpha\times\beta$ belongs to 
$\WA (F_d ,X\times X)$.

Let $g: X\times X\to X$ be any homeomorphism which satisfies $g(A\times X)=A$. 
We claim that the conjugated action $g\circ(\alpha\times\beta)\circ g^{-1}$ on $X$ lies in $\bigcap_{k=2}^d \sW_{k,A}$. Indeed, for $2\leq k\leq d$, assume that 
$g\circ (\alpha\times \beta)_{a_k} \circ g^{-1}(A)=A$. We must show that $A$ partition nontrivially into two disjoint clopen sets each of which is fixed by $g\circ (\alpha\times \beta)_{a_k} \circ g^{-1}$. Taking $C$ to be any proper nonempty clopen subset of $X$, it is easy to see that $A = g(A\times C)\sqcup g(A\times (X\setminus C))$ is partition of the desired type.

By Proposition~\ref{P-extension approximation} (applied with respect to the canonical projection $X\times X\to X$ onto the first coordinate) we can find a conjugate $\alpha^{(1)}$
of $\alpha\times\beta$  as close as we wish to $\alpha$, via a conjugation map which satisfies the properties as stated for $g$ above.

Now repeat the above construction substituting $a_2$ for $a_1$
and $\alpha^{(1)}$ for $\alpha$. Then we obtain an action $\alpha^{(2)}\in\WA(F_d,X)$ that in particular lies in $\sW_{1,A}$ (even lies in $\sW_{k,A}$ for all $k\neq 2$)
and is as close as we wish to $\alpha^{(1)}$. Note that in our construction $\alpha^{(2)}$ was achieved as a conjugate (via a homeomorphism which maps $A\times X$ to $A$) of a product of $\alpha^{(1)}$ with another action.
Since $\alpha^{(1)}\in \bigcap_{k=2}^d \sW_{k,A}$, it follows by the nature of the product construction that $\alpha^{(2)}$ must also belong to $\bigcap_{k=2}^d \sW_{k,A}$. Therefore $\alpha^{(2)}$ belongs to $\sW_A = \bigcap_{k=1}^d \sW_{k,A}$,
establishing the desired density.
\end{proof}

\section{Generic square divisibility I}\label{S-SD I}

The proof of our first genericity result, Theorem~\ref{T-SD free products}, rests largely on the density statement for $O$-square divisibility captured in Lemma~\ref{L-SD free product}. The main technical tools have already been developed 
in Section~\ref{S-diagonal machine} in the form of Propositions~\ref{P-min WM} and \ref{P-extending}, which permit us to construct a diagonal action 
that inherits the properties of minimality, weak mixing, and the existence of an invariant Borel probability measure from one of the factors while also possessing the desired degree of square divisibility. 
With such diagonal actions at hand we can then appeal 
to Proposition~\ref{P-extension approximation} to clinch the desired density.
Without the requirements of weak mixing and the existence of an invariant Borel 
probability measure the construction would be much simpler, as we could pass to a minimal
subsystem after taking a product, an operation that preserves neither of the two properties in general.

\begin{lemma}\label{L-SD free product}
Suppose that $G$ is residually finite and that $H$ is infinite and amenable and contains a normal infinite cyclic subgroup. 
Suppose $X$ is the Cantor set, and let $O$ be a nonempty clopen subset of $X$.
Then the set of $O$-squarely divisible actions in $\WA (G*H,X)$ is dense.
\end{lemma}

\begin{proof}
Let $\alpha\in \WA (G*H,X)$ and let us show that we can approximate $\alpha$ arbitrarily well by
an $O$-squarely divisible action in $\WA (G*H,X)$.
By minimality, continuity, and the fact that $X$ has no isolated points, we can find for every $x\in X$ 
a clopen neighbourhood $A_x$ of $x$ in $X$ and an $s_x \in G*H\setminus\{e\}$ such that 
$\alpha_{s_x} A_x \subseteq O$. 
Then $\{ A_x \}_{x\in X}$ is a clopen cover of $X$ and hence by compactness
has a finite subcover. Disjointifying this cover by a standard recursive process and taking unions of
members of the resulting clopen partition for which the corresponding group element $s_x$ is the same,
we obtain a finite set $F\subseteq G*H\setminus \{e\}$ and a clopen partition $\{ W_s \}_{s\in F}$ of $X$
such that $\alpha_ sW_s \subseteq O$ for every $s\in F$. 

Take an action $H\curvearrowright X$ as given by Proposition~\ref{P-min WM}.
This action has comparison by \cite[Theorem~A]{Nar24} together with \cite[Theorem~6.1]{KerSza20}.
By Proposition~\ref{P-extending} there is an action  
$G*H\stackrel{\beta}{\curvearrowright} X$ extending $H\curvearrowright X$ (viewing $H$ as a subgroup of $G*H$)
such that $\beta$ is $(O_1 , O_2 , E)$-squarely divisible for all nonempty clopen sets $O_1 , O_2 \subseteq X$
and finite sets $E\subseteq G*H$ containing $e$, 
$M_{G*H} (X)$ is equal to $M_H (X)$ (which is nonempty by the amenability of $H$),
and there exists a nonempty clopen set $A \subseteq X$ such that $(F\sqcup \{ e \},A)$ is a tower.

Form the diagonal action $G*H\stackrel{\alpha\times\beta}{\curvearrowright} X\times X$, which is minimal and 
weakly mixing by our choice of $H\curvearrowright X$ via Proposition~\ref{P-min WM}.

Fix two disjoint nonempty clopen subsets $O_1$ and $O_2$ of $A$ (which exist since $X$ has no isolated points). Let $E\subseteq G*H$ be a finite subset containing $e$. As $\beta$ is $(O_1 , O_2 , E)$-squarely divisible,
by Proposition~\ref{P-SD0Dim} there exist $n\in\Nb$, a collection $\{ V_{i,j} \}_{i,j=1}^n$ 
of pairwise equivalent and pairwise $E$-disjoint clopen subsets of $X$, and, writing $V = \bigsqcup_{i,j=1}^n V_{i,j}$,
$V_1 = \bigsqcup_{i=1}^n V_{i,1}$, $R = V^c$, and $B = V\cap (V^E)^c$, one has the following:
\begin{enumerate}
\item $V_{i,1} \prec_{\beta } O_1 \cap \bigsqcup_{j=2}^n V_{i,j} \cap B^c$ for every $i=1,\dots ,n$,

\item $R\prec_{\beta } O_2 \cap V \cap (V_1 \cup B)^c$, 

\item $B\prec_{\beta } O_2 \cap R$.
\end{enumerate}

Denoting by $C^\pr$ the product set $X\times C$ for every $C\subseteq X$, we obtain from this
\begin{enumerate}
\item $V_{i,1}^\pr \prec_{\alpha\times\beta } O_1^\pr \cap \bigsqcup_{j=2}^n V_{i,j}^\pr \cap (B^\pr )^c$ for every $i=1,\dots ,n$,

\item $R^\pr\prec_{\alpha\times\beta } O_2^\pr \cap V^\pr \cap (V_1^\pr\cup B^\pr)^c$, 

\item $B^\pr\prec_{\alpha\times\beta } O_2^\pr \cap R^\pr$.
\end{enumerate}

Now since the sets $\beta_s (O_1 \sqcup O_2 )$ for $s\in F$ are pairwise disjoint subsets of $X\setminus (O_1\sqcup O_2)$ (since $e\notin F$) and $\alpha_s W_s \subseteq O$ for every $s\in F$, we have
\begin{align*}
X\times (O_1 \sqcup O_2 )
&= \bigsqcup_{s\in F} W_s \times (O_1 \sqcup O_2 ) \\
&\prec_{\alpha \times \beta} \bigsqcup_{s\in F} (\alpha\times\beta )_s (W_s \times (O_1 \sqcup O_2 )) \\
&\subseteq O\times (X\setminus (O_1\sqcup O_2)).
\end{align*}
This shows that $O_1^\pr \sqcup O_2^\pr = X\times (O_1 \sqcup O_2 ) \prec_{\alpha\times\beta} O\times (X\setminus (O_1\sqcup O_2))$. 
Since the two sets in this subequivalence are disjoint and the definitions of $O_1^\pr$ and
$O_2^\pr$ did not depend on $E$, we have verified that $\alpha\times\beta$ is $(O\times X)$-squarely divisible.

By Proposition~\ref{P-extension approximation},
there is a homeomorphism $g : X\times X\to X$ which satisfies $g(O\times X) = O$ and 
conjugates $\alpha\times\beta$ to an action $\gamma$ of $G*H$ on $X$ that approximates $\alpha$ as closely as we wish.
The first of these conditions together with the $(O\times X)$-square divisibility of $\alpha\times\beta$
ensures that $\gamma$ is $O$-squarely divisible. Since $\alpha\times\beta$ is minimal and weakly mixing and admits an invariant Borel probability measure, the action $\gamma$ has these properties as well. Finally, since $\alpha$ is topologically free, the product $\alpha\times\beta$ 
is topologically free and hence $\gamma$ is topologically free. Thus $\gamma\in\WA (G*H,X)$,
which finishes the proof.
\end{proof}

\begin{theorem}\label{T-SD free products}
Suppose $X$ is the Cantor set. Suppose that $G$ is residually finite and that $H$ is infinite and amenable 
and contains a normal infinite cyclic subgroup.
Then the set of all weakly squarely divisible actions in $\WA (G*H,X)$ is a dense $G_\delta$.
\end{theorem}

\begin{proof}
Let $O$ be a nonempty clopen subset of $X$ and $E$ a finite subset of $G*H$ containing $e$. Since there are only countably many such $O$ and $E$, by the Baire category theorem it suffices to show that the set $\sW_{O,E}$ of all $(O,E)$-squarely divisible actions in $\WA (G*H,X)$ is a dense open set. The openness is Lemma~\ref{L-open} and the density is Lemma~\ref{L-SD free product} (using that $O$-square divisibility implies $(O,E)$-square divisibility).
\end{proof}

For $d\geq 2$ the free group $F_d$ can be written as $G*\Zb$ where $G$ is a free product of $d-1$ copies of $\Zb$. 
From the above theorem we thus obtain:

\begin{corollary}\label{C-SD free groups}
For $d\geq 2$ and the Cantor set $X$ the set of all weakly squarely divisible actions in $\WA (F_d,X)$ is a dense $G_\delta$.
\end{corollary}

\begin{remark}
As mentioned in the introduction, the space of all free minimal actions $F_2 \curvearrowright X$ on the Cantor set
with $M_{F_2} (X)\neq\emptyset$
has a generic action, namely the universal odometer \cite{DouMelTsa25}. 
The arguments for establishing Theorem~\ref{T-SD free products} also apply to this space
and thus show that the universal odometer is weakly squarely divisible, so that its reduced crossed product
has stable rank one by Theorem~\ref{T-SR1}. One can also verify this directly using translation actions on spaces of the form $F_2 /H$,
where $H$ is a finite-index subgroup of $F_2$, in which one generator acts at a much larger scale than
the other so as to effectively reproduce the situation in the proof of Proposition~\ref{P-extending}, but now
inside a single clopen tower partitioning the space, with the tower levels indexed by $F_2 /H$ and
permuted by the translation action of $F_2$.
\end{remark}

\section{Generic square divisibility II}\label{S-SD II}

For $d>1$, Corollary~\ref{C-SD free groups} says that, within the
space $\WA (F_d,X)$ of all actions $F_d \curvearrowright X$ that are minimal, weakly mixing, and topologically free 
and satisfy $M_{F_d} (X) \neq\emptyset$, the weakly squarely divisible ones form a dense $G_\delta$ set. 
We do not know much however about the size of this $G_\delta$ set modulo the relation of conjugacy,
although we will exhibit many examples of these actions in Section~\ref{S-examples}. 

In this section we examine instead the space $\Astar (F_d,X)$ of topologically free actions
$F_d \curvearrowright X$ that satisfy $M_{F_d} (X) \neq\emptyset$ and are strictly ergodic and spectrally aperiodic
on each generator. This intersects the space $\WA (F_d,X)$ but neither of these spaces seems to contain
the other. As we did for $\WA (F_d,X)$, we will prove that the set of all weakly squarely divisible actions in $\Astar (F_d,X)$ 
is a dense $G_\delta$ (Theorem~\ref{T-SD free groups}), but we will also be able to show that every conjugacy class in 
$\Astar (F_d,X)$ is meagre
(this is Theorem~\ref{T-meagre}, established in Section~\ref{S-meagre}), so that the set of actions in $\Astar (F_d,X)$ modulo the relation of conjugacy is quite large from the descriptive viewpoint. 

The \textit{continuous discrete spectrum} of a transformation $S\curvearrowright Y$ is the set of eigenvalues of the linear operator $g\mapsto g\circ S$ acting on $C(Y)$. 

\begin{lemma}\label{L-evalues}
Let $S\curvearrowright X$ and $T\curvearrowright Y$ be minimal homeomorphisms.
Suppose that there exists a $\nu\in M_T (Y)$ such that 
$T\curvearrowright (Y,\nu)$ is weakly mixing. Then $S$ and $S\times T$ have the same continuous discrete spectrum.
\end{lemma}

\begin{proof}
It is clear that the continuous discrete spectrum of $S$ is contained in that of $S\times T$,
since (identifying $C(X)\otimes C(Y)$ with $C(X\times Y)$) we can take a tensor product of a continuous eigenfunction for the former with $\unit_Y$ to get a
continuous eigenfunction for the latter with the same eigenvalue. To establish the reverse inclusion we will
show that all continuous eigenfunctions for $S\times T$ arise in this way.

By the amenability of $\Zb$ there exists a $\mu\in M_S (X)$.
Let $U_S \in \cB (L^2 (X,\mu ))$ and $U_T \in \cB (L^2 (Y,\nu ))$ be the Koopman operators
associated to $S\curvearrowright (X,\mu )$ and $T\curvearrowright (Y,\nu )$.
Then we have orthogonal decompositions of $U_S$ and $U_T$ as $\unit \oplus U_{S,0} \in \cB (\Cb \oplus L^2_0 (X,\mu ))$ and $\unit \oplus U_{T,0} \in \cB (\Cb \oplus L^2_0 (Y,\nu ))$. The Koopman operator $U_{S\times T}$
associated to the diagonal transformation $S\times T\curvearrowright (X\times Y,\mu\times\nu )$ is then unitarily equivalent to 
$1\oplus U_{S,0} \oplus U_{T, 0} \oplus (U_{S,0} \otimes U_{T, 0} )$
acting on $\Cb\oplus L^2_0 (X,\mu ) \oplus L^2_0 (Y,\nu ) \oplus (L^2_0 (X,\mu ) \otimes L^2_0 (Y,\nu ))$.

Suppose that $g\in C(X\times Y)$ is an eigenfunction for $S\times T$ with eigenvalue not equal to $1$ (which we may assume since the continuous discrete spectra of $S$ and $S\times T$ clearly share the eigenvalue $1$ by considering constant functions). 
Since $S$ and $T$ are minimal the measures $\mu$ and $\nu$ have full support and hence so does
the measure $\mu\times\nu$. We can then view $g$ as an eigenfunction
for $U_{S\times T}$ in $L^2 (X\times Y, \mu\times\nu )$. Since $T\curvearrowright (Y,\nu)$ is weakly mixing, $U_{T, 0}$ is weakly mixing 
as a unitary operator (see \cite[Theorem~2.25]{KerLi16}), which also means that the operator $U_{S,0} \otimes U_{T, 0}$ is weakly mixing (see \cite[Theorem~2.23]{KerLi16}).
Since weakly mixing unitary operators have no eigenvalues and $g$ is associated with an eigenvalue different from $1$, this implies that $g$
belongs to $L^2_0 (X,\mu ) \otimes \Cb\unit_Y$ and thus, modulo a.e.\ equality, has the form
$f\otimes \unit_Y$ for some $f\in L^2_0 (X,\mu )$. It follows that for $\mu$-a.e.\ $x\in X$ we have $g(x,y)=f(x)$ for $\nu$-a.e.\ $y\in Y$.
The $Y$-cross sections of $g$, namely $g(x,\cdot): Y\to \mathbb{C}$ for $\mu$-a.e.\ $x\in X$, are therefore 
$\nu$-a.e.\ constant and thus constant by the continuity of $g$ and the fact that $\nu$ has full support. Now the set $\{x\in X: g(x,y)=f(x) \text{ for all } y\in Y\}$ must be dense in $X$ since $\mu$ has full support. We deduce, again using the continuity of $g$, that $g$ has the form 
$g_0 \otimes \unit_Y$ for some $g_0 \in C(X)$. But then $g_0$ is a continuous eigenfunction
for $S$ with the same eigenvalue that $g$ has as an eigenfunction for $S\times T$. 
We conclude that the continuous discrete spectrum of $S\times T$ is contained in that of $S$, as desired.
\end{proof}

\begin{lemma}\label{L-evalues2}
Suppose $X$ is the Cantor set. Let $S\curvearrowright X$ and $T\curvearrowright X$ 
be homeomorphisms such that $S$ is spectrally aperiodic
and there exists a $\nu\in M_T(X)$ for which $T\curvearrowright (X,\nu)$ is weakly mixing. 
Suppose that $S\times T$ is minimal. Then $S\times T$ is spectrally aperiodic.
\end{lemma}

\begin{proof}
Since $S\times T\curvearrowright X\times X$ is minimal, the factors $S$ and $T$ are themselves minimal. As observed in \cite{Orm97} (see also \cite{GioPutSka95}), the minimal homeomorphism $S$ is spectrally aperiodic if and only if there does not exist a positive integer $n\neq 1$ such that $e^{2\pi i/n}$ belongs to the continuous discrete spectrum of $S$, and likewise for $S\times T$. Since the continuous discrete spectra of $S$ and $S\times T$ agree by Lemma~\ref{L-evalues}, it follows from our assumption on $S$ that $S\times T$ is spectrally aperiodic.
\end{proof}

\begin{lemma}\label{L-SD free groups}
Suppose $X$ is the Cantor set. Let $d\in\mathbb{N}$ and let $O\subseteq X$ be a nonempty clopen set. Then the set of $O$-squarely divisible actions in $\Astar (F_d ,X)$ is dense.
\end{lemma}

\begin{proof}
Let $\alpha\in\Astar (F_d ,X)$ and $\mu\in M_\alpha (X)$.
Let $T_1 , \ldots , T_d$ be the homeomorphisms obtained by
restricting $\alpha$ to the $d$ standard generators of $F_d$. 
For each $i=1,\ldots,d$ write $\mu_i$ for the unique measure in $M_{T_i}(X)$, which is ergodic. Since each $T_i$ is minimal and $X$ is infinite, the measures $\mu_i$ must be atomless. Therefore there exists a standard atomless Borel probability space $(Z,\zeta)$ such that $(X,\mu_i)\cong (Z,\zeta)$ for every $i=1,\ldots, d$.
By Lemma~\ref{L-nontorsion}
there is a weakly mixing p.m.p.\ transformation $S \curvearrowright (Z,\zeta)$ which is disjoint from
$T_i \curvearrowright  (X,\mu_i)$ for each $i=1,\ldots , d$. By the Jewett--Krieger theorem
there is a minimal homeomorphism $R \curvearrowright  X$ such that $M_R (X)$ contains a unique
element $\nu$ and the transformations $R \curvearrowright  (X,\nu)$ and $S \curvearrowright (Z,\zeta)$
are measurably conjugate. By Lemma~\ref{L-product disjoint} this implies that for each $i=1,\ldots , d$
the diagonal transformation $T_i \times R \curvearrowright X\times X$ is minimal and, by Lemma~\ref{L-evalues2}, spectrally aperiodic. It is also uniquely ergodic by the unique ergodicity of $T_i$ and $R$
and the disjointness of $T_i \curvearrowright  (X,\mu_i)$ and $R \curvearrowright  (X,\nu)$. 

By the minimality of $T_1$ and the compactness of $X$ we can find
a clopen partition $\{ W_k \}_{k=1}^l$ of $X$ and distinct integers $j_1 , \ldots , j_l \in\Zb$ such that $T_1^{j_k} W_k \subseteq O$
for every $k=1,\ldots,l$. 
We moreover may assume that $j_1,\ldots,j_l\in\Zb\setminus\{0\}$ since $T_1$ is minimal and $X$ has no isolated points.
Define an action $\beta\in\Act (F_d ,X)$ by setting $\beta_{a_i} = R$ for every $i=1,\ldots ,d$,
where $a_1 , \ldots , a_d$ are the standard generators of $F_d$.
Since $X$ is infinite, the homeomorphism $R$, being minimal, is free as a $\Zb$-action, and so we can find a
nonempty clopen set $A \subseteq X$ such that the sets $R^{j_k} A$ for $k=1,\dots , l$ are pairwise disjoint and also disjoint from $A$.
Fix $O_1$ and $O_2$ nonempty clopen disjoint subsets of $A$ (which exist since $X$ has no isolated points). 
As in the proof of Lemma~\ref{L-SD free product}, we then have
\begin{align*}
X\times (O_1 \sqcup O_2 )
&= \bigsqcup_{k=1}^l W_k \times (O_1\sqcup O_2)\\
&\prec_{T_1\times R} \bigsqcup_{k=1}^l T_1^{j_k} W_k \times R^{j_k} (O_1\sqcup O_2) \\
&\subseteq O \times (X\setminus (O_1\sqcup O_2)).
\end{align*}
By Theorem~\ref{T-amenable} the action $R$ is $(O_1,O_2,J)$-squarely divisible for all finite sets $J\subseteq \Zb$ containing $0$.
Arguing as in the proof of Lemma~\ref{L-SD free product}, this shows that $\alpha\times \beta$ is $(X\times O_1,X\times O_2,E)$-squarely divisible for every finite set $E\subseteq F_d$ containing $e$ (here we use the fact that, by the definition of $\beta$, for every $g\in F_d$ there exists a unique $N_g\in\Zb$ for which $\beta_g=R^{N_g}$).
Moreover, since the two sets in the subequivalence $X\times (O_1 \sqcup O_2 )\prec_{T_1\times R} O \times (X\setminus (O_1\sqcup O_2))$ are pairwise disjoint, we conclude that $\alpha\times \beta$ is $(O\times X)$-squarely divisible (see Proposition~\ref{P-SD0Dim}).
Again as in the proof of Lemma~\ref{L-SD free product},
we can conjugate $\alpha\times \beta$ to an action $\gamma$ on $X$ via a homeomorphism $g : X\times X \to X$ 
satisfying $g(O\times X) = O$ so that $\gamma$ approximates $\alpha$ as closely as we wish. 
In view of the $(O\times X)$-square divisibility of $\alpha\times\beta$, the equality $g(O\times X) = O$
ensures that $\gamma$ is $O$-squarely divisible.
Since $\alpha$ is topologically free, so is $\alpha\times \beta$. Moreover 
$M_{\alpha\times\beta}(X\times X) = \{ \mu\times \nu \}$, while $\alpha\times\beta$ is strictly ergodic and spectrally aperiodic 
on each standard generator of $F_d$. Therefore $\gamma\in\Astar (F_d ,X)$, which completes the proof.
\end{proof}

\begin{theorem}\label{T-SD free groups}
Suppose $X$ is the Cantor set. For $d\in\Nb$ the set of all weakly squarely divisible actions in $\Astar (F_d ,X)$ is a dense $G_\delta$.
\end{theorem}

\begin{proof}
As in the proof of Theorem~\ref{T-SD free products}, it suffices to show that for every
nonempty clopen set $O\subseteq X$ and finite set $e\in E \subseteq F_d$ the set
$\sW_{O,E}$ of all $(O,E)$-squarely divisible actions in $\Astar (F_d ,X)$ is open and dense.
Lemma~\ref{L-open} yields the openness and Lemma~\ref{L-SD free groups} yields the density (using that $O$-square divisibility implies $(O,E)$-square divisibility).
\end{proof}

\section{Meagreness of orbits: proof of Theorem~\ref{T-meagre}}\label{S-meagre}

Our goal here is to establish the meagreness-of-orbits result announced in the introduction as Theorem~\ref{T-meagre}.
For this we develop a topological version of the entropy-and-disjointness argument from
Section~4 of \cite{ForWei04} that involved extending the generic disjointness result  
for ergodic p.m.p.\ $\Zb$-actions from \cite{Jun81} 
(whose proof did not involve entropy) 
to the more general setting of amenable groups.

Actually our disjointness argument concerns $\Zb$-actions that arise 
as restrictions of actions $F_d \cong \Zb * F_{d-1}$ preserving a Borel probability measure.
In \cite{Hoc08}, and like in \cite{Jun81} without the use of entropy, Hochman establishes
a generic disjointness result in the space of homeomorphisms of the Cantor set
that we would be able to apply here if it were not for the constraint that we cannot simply
perturb the $\Zb$ part of the action at will. Such a perturbation will always generate an action of the free product 
$\Zb * F_{d-1} \cong F_d$ but can easily result in the nonexistence of Borel probability measures that are
invariant for the whole $F_d$.
Our strategy has been to instead combine the entropy approach of \cite{ForWei04} 
with the orbit equivalence theory of Cantor minimal systems.

\begin{notation}
Suppose that $X$ is the Cantor set. We denote by $\Astar (\Zb,X)$ the space of strictly ergodic spectrally aperiodic 
actions in $\Act (\Zb ,X)$ (this agrees with $\Astar(F_d,X)$, as defined in Section~\ref{S-spaces of actions}, when $d=1$).
For $T\in\Astar (\Zb ,X)$ we write $\Astar (\Zb ,X)_T$ for the set of all transformations in 
$\Astar (\Zb ,X)$ which have the same orbits as $T$.
\end{notation}

When dealing with homeomorphisms $T\curvearrowright X$ of the Cantor set in this and the next section, 
our clopen castles will always partition $X$ (i.e., have empty remainder) and have shapes of the form $\{ 0,\dots , n-1\}$,
even when this is not explicitly stated. We write
$\sC=\{(B_i,n_i)\}_{i=1}^k$ for such a castle when the shape of the $i$th tower is the interval $\{0,\ldots,n_i-1\}$. 
Note that if $p\in\Nb$ divides all of the heights $n_i$, then one can find levels of the castle whose union $A$ 
produces a partition $X = A\sqcup TA\sqcup\dots\sqcup T^{p-1}A$.
Consequently, if $T$ is spectrally aperiodic then $\gcd (n_1,n_2,\ldots, n_k)=1$.
Write $\base(\sC)=\bigsqcup_{i=1}^{k}B_i$ and $\towertop(\sC)=\bigsqcup_{i=1}^{k}T^{n_i-1}B_i$, and note that $T(\towertop(\sC))=\base(\sC)$.
We will moreover consider measurable castles with interval shapes for p.m.p. transformations $S\curvearrowright (Z,\zeta)$, 
where the sets $B_i$ are instead required to be measurable and we similarly ask that the collection
$\{S^jB_i : j=0,\ldots, n_i-1, i=1,\ldots,k \}$ is disjoint and has union of full measure.

Given a minimal homeomorphism $T\curvearrowright X$ of the Cantor set, by a standard first return time
construction one can generate a clopen castle that partitions $X$ starting from any nonempty clopen set $Y\subseteq X$.
Indeed a straightforward consequence of compactness and minimality is the existence
of an $N\in\Nb$ such that $X = \bigcup_{n=1}^N T^{-n}Y$. The first return map $r_Y: Y\to \Nb$ 
determined by $r_Y(x) = \min\{n\in \Nb : T^nx \in Y\}$ for $x\in Y$ is then well-defined and continuous (since $Y$ is clopen) 
and takes finitely many values 
$\{n_1,\ldots,n_k\}$. Setting $B_i = r_Y^{-1}(\{ n_i \})$ for $i=1,\dots,k$, we obtain a clopen castle decomposition 
\[
X = \bigsqcup_{i=1}^k \bigsqcup_{j=0}^{n_i - 1} T^jB_i. 
\]
Note that the base of this castle is exactly $Y$.

The following theorem can be extracted from \cite{Orm97}. 
The fact that the set $\sM$ below is nonempty is 
a special case of Theorem~6.1 of \cite{Orm97}, and the fact that it intersects every neighbourhood of $T$ 
can be verified by making a careful choice of the initial castle
in the recursive construction in \cite{Orm97}, as we explain below.
Recall that two free transformations $T\curvearrowright X$ and $S\curvearrowright Y$ are said to be \textit{strongly orbit equivalent} if there exists an orbit-preserving homeomorphism $\varphi :X \to Y$ between them
with associated cocycle maps (as defined in Section~\ref{S-SD}) that have no more than one point of discontinuity each.

\begin{theorem}\cite[Theorem~6.1]{Orm97}\label{T-SOE}
Let $T\curvearrowright X$ be a minimal spectrally aperiodic homeomorphism of the Cantor set and let $\mu\in M_T(X)$ be an ergodic measure.
Let $S\curvearrowright (Z,\zeta)$ be an ergodic p.m.p.\ transformation of a standard atomless probability space. 
Let $\sM$ be the set of all minimal homeomorphisms $T'\curvearrowright X$ 
such that the identity map $\id: X\to X$ is a strong orbit equivalence between $T$ and $T'$ 
and $T'\curvearrowright (X,\mu)$ 
is measure conjugate to $S\curvearrowright (Z,\zeta)$ (note that $M_{T'}(X)=M_{T}(X)$ since $\orb_T(x)=\orb_{T'}(x)$ for every $x\in X$).
Then every neighbourhood of $T$ in $\Act (\Zb ,X)$ contains an element of $\sM$.
\end{theorem}

We first sketch the ideas in \cite{HerPutSka92} and \cite{Orm97} that yield $\sM\neq \emptyset$.  
For more on Bratteli diagrams, orderings on such diagrams, and Bratteli--Vershik transformations, we refer the reader to \cite{Put18}.
Ormes proves that there exists a simple Bratteli diagram $(V,E)$ with two orderings $\leq$ and $\leq'$ such that

\begin{enumerate}
\item $(V,E,\leq )$ and $(V,E,\leq' )$ share minimal and maximal edges and both have unique minimal and maximal paths,

\item $(V,E,\leq )$ defines a Bratteli--Vershik homeomorphism $\lambda$ of $B_{V,E}$ and 
there is a homeomorphism $h: X\to B_{V,E}$ such that $T = h^{-1}\circ \lambda \circ h$,

\item $(V,E,\leq' )$ defines a Bratteli--Vershik homeomorphism $\lambda'$ of $B_{V,E}$ such that
$S\curvearrowright (Z,\zeta)$ is measure conjugate to $\lambda' \curvearrowright (B_{V,E}, h_* \mu )$.
\end{enumerate}
One can then check that $\lambda$ and $\lambda'$ are strongly orbit equivalent via the identity map and therefore $T':=h^{-1}\circ \lambda'\circ h$ belongs to $\sM$.

In what follows we explain how the above Bratteli diagrams arise, following \cite{HerPutSka92, Orm97}.
See Figures~\ref{F-partition} and \ref{F-tower} for an illustration.
Suppose that $\sC_1 \prec \sC_2 \prec \sC_3 \prec \dots$ is a sequence of clopen castles for $T$, where $\sC_m\prec \sC_{m+1}$ means
that the levels of the castle $\sC_m$ are unions of levels of the castle $\sC_{m+1}$ and $\base(\sC_{m+1})\subseteq \base(\sC_m)$. 
Let $\sC_0=\{(X,0)\}$ be the trivial castle.
Such a sequence gives rise to a Bratteli diagram $(V,E)$ with vertex sets $\{V_m\}_{m\geq 0}$ and edge sets $\{E_m\}_{m\geq 1}$ defined as follows: the vertices in $V_m$ are indexed by the towers in $\sC_m$, so that $|V_m|$ is the number of towers forming the castle $\sC_m$.
The number of edges $e\in E_m$ going from the vertex $k\in V_{m-1}$ to the vertex $l\in V_m$ is 
defined to be the number of levels of the $l$th tower of $\sC_m$ which are contained in the base of the $k$th tower of $\sC_{m-1}$.
Note for example that the number of edges going from the only vertex $1\in V_0$ to the vertex $i\in V_1$ 
is equal to the height of the $i$th tower of the castle $\sC_1$.

\begin{figure}
\includegraphics[width = 0.65\textwidth]{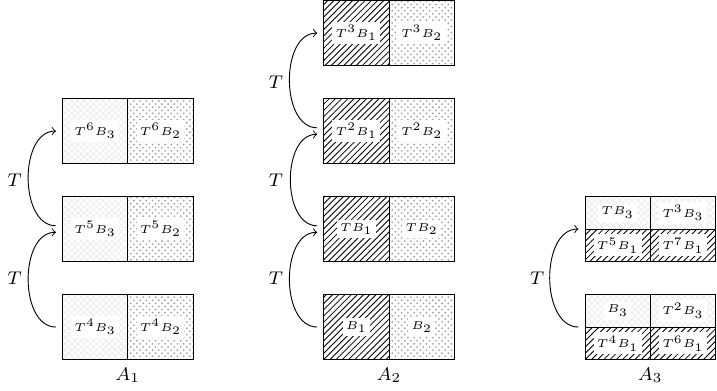}
\caption{Castles $\sC_1 \prec \sC_2$ for $T$ with $\base (\sC_1 ) = A_1 \sqcup A_2 \sqcup A_3$ and 
$\base (\sC_2 ) = B_1 \sqcup B_2 \sqcup B_3$.}
\label{F-partition}
\end{figure}

\begin{figure}
\includegraphics[width = 0.65\textwidth]{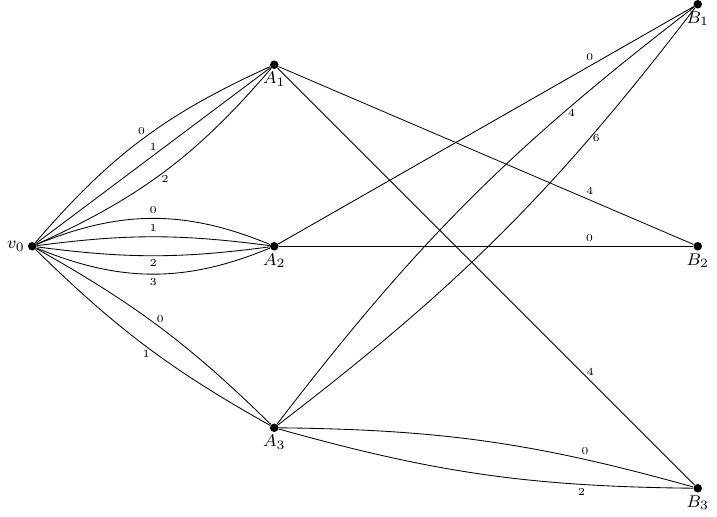}
\caption{The initial paths of length two in the ordered Bratteli diagram associated to Figure~\ref{F-partition}.}
\label{F-tower}
\end{figure}

Next define an ordering $\leq$ on the Bratteli diagram as follows.
For each vertex $l \in V_m$ one introduces a total ordering on the subset $r^{-1}(\{l\}) \subseteq E_m$, where $r: E_m \to V_m$ denotes the range map.
Writing $\sC_m = \{(B_i,n_i)\}_{i=1}^{k}$, we note by construction that the edges in $r^{-1}(\{l\})$ 
can be naturally labeled by the set of integers $\{ j\in \{ 0,\dots ,n_l-1 \} : T^jB_l\subseteq \base(\sC_{m-1})\}$.
For instance, if $l \in V_1$ then $r^{-1}(\{l\})$ is labeled by the entire interval $\{ 0,\dots ,n_l-1 \}$. This labeling induces a natural ordering on $r^{-1}(\{l\})$ via the standard order on the integers.
One moreover shows that for every $m\in \Nb$ there is a bijection 
\[
\big\{\text{initial paths of length } m \text{ in } (V,E)\big\} \to \big\{\text{levels of the castle } \sC_m\big\}
\] 
defined by $e_1e_2\cdots e_m \mapsto T^{\sum_{j=1}^{m}\mathrm{label}(e_j)}B_{m,r(e_m)}$,
where $B_{m,r(e_m)}$ is the base of the $r(e_m)$th tower of the castle $\sC_m$.
If $\bigvee_{m\in\Nb} \sC_m$ generates the topology on $X$, then the compatibility of these maps leads to a homeomorphism $B_{V,E}\cong X$
mapping an infinite path $e_1e_2\cdots\in B_{V,E}$ to $\bigcap_{m=1}^{\infty} T^{\sum_{j=1}^{m}\mathrm{label}(e_j)}B_{m,r(e_m)}$, which is an intersection of a decreasing sequence of sets.
Denote the inverse of this homeomorphism by $h: X\to B_{V,E}$.
If $\bigcap_{m\in\Nb}\towertop(\sC_m)=\{x_0\}$ then the Bratteli diagram constructed in this way has unique minimal and maximal paths $x_\mathrm{min}$ and $x_{\mathrm{max}}$ with $h(x_0)=x_{\mathrm{max}}$, 
and it is straightforward to check, by the definition of the ordering, that the Bratteli-Vershik homeomorphism $\lambda$ of $B_{V,E}$ satisfies $h\circ T=\lambda\circ h$.
 
This is essentially how condition~(ii) above is arranged. However, in order to achieve a Bratteli diagram satisfying conditions (i)--(iii), a recursive procedure is used to simultaneously build a sequence 
of measurable castles $\sD_1\prec \sD_2\prec\sD_3\prec\ldots$ for $S$ (with the relation $\prec$ defined on measurable castles in the same way as for clopen castles).
Letting again $\sD_0$ be the trivial castle, we obtain, exactly as above, an associated ordered Bratteli diagram $(V',E',\leq')$, and if $\{\sD_m\}_{m\in\Nb}$ separates points on a set of full $\zeta$-measure, one additionally gets a measurable isomorphism 
$h':(Z,\zeta)\to (B_{V',E'}, h'_*\zeta)$ 
that conjugates $S$ to the Bratteli–Vershik homeomorphism $\lambda'$ of $B_{V',E'}$, provided that $(V',E',\leq')$ has unique minimal and maximal paths.

In this recursive construction, the initial clopen castle $\sC_1$ can be chosen freely.
During the simultaneous construction of the sequences $\sC_1 \prec \sC_2 \prec \dots$ and $\sD_1 \prec \sD_2 \prec \dots$ one requires that $\sC_m$ and $\sD_m$ have the same number of towers with the same corresponding heights for all $m \in \Nb$.
This ensures that $V$ can be identified with $V'$.
A further technical ad hoc compatibility condition is imposed on the levels of the castles to guarantee that $E = E'$ and that the minimal and maximal edges coincide for the orderings $\leq$ and $\leq'$ on $(V,E)$.
Finally, one requires that if $\sC_m = \{(B_i,n_i)\}_{i=1}^{k}$ and $\sD_m = \{(A_i,n_i)\}_{i=1}^{k}$
then $\mu(B_i) = \zeta(A_i)$ for all $i = 1,\dots,k$.
This condition ensures that the measures $h'_*\zeta$ and $h_*\mu$ on $B_{V,E}$ agree, so that the map $h': (Z,\zeta)\to (B_{V,E},h_*\mu)$ defined above is a measurable isomorphism conjugating $S$ to $\lambda'$ as required in condition (iii). 

A key tool in this construction is Alpern's copying lemma (see \cite[Corollary~2]{Alp79} and \cite[Theorem~4.1]{Orm97}), which in our case can be applied in its original form since $T$ is assumed to be spectrally aperiodic (which, as explained earlier in this section, implies that each associated castle $\sC_m$ has towers whose heights have greatest common divisor equal to $1$).
Using Alpern's lemma, one can, for any clopen castle $\sC$ of $T$, construct a measurable castle $\sD$ for $S$ with the same number of towers and identical corresponding heights such that the $\mu$-measures of the levels of $\sC$ agree with the $\nu$-measures of the corresponding levels of $\sD$.
In particular, Alpern’s lemma is already used in the first step to construct $\sD_1$.

Given a neighbourhood
$U$ of $T$ in $\Act (\Zb ,X)$, we can construct a $\sC_1$ as follows in order to verify 
that $T'=h^{-1}\circ\lambda'\circ h$ moreover belongs to $U$. 

Let $\eps > 0$ be such that every $\tilde{T} \in\Act (\Zb ,X)$ with $d(\tilde{T}x,Tx) , d(\tilde{T}^{-1}x,T^{-1}x)< \eps$ 
for all $x\in X$ belongs to $U$.
Take a nonempty clopen set $B\subseteq X$ with diameter less than $\eps$
such that $T^{-1} B$ also has diameter less than $\eps$ (which is possible by the continuity of $T^{-1}$). 
As described at the beginning of this section, we then obtain a clopen
castle $\sC_1=\{(B_i, n_i)\}_{i=1}^{k}$ with $\base(\sC_1)=B$, via first-return times to $B$.
We may assume that the levels of $\sC_1$ all have diameter less than $\eps$. Indeed, let $\sP$ be any clopen partition of $X$ with elements of diameter less than $\eps$ and partition each $B_i$ by $\bigvee_{j=0}^{n_i-1}T^{-j}\sP$ so as to obtain 
a clopen partition $B_{i,0},\ldots,B_{i,r_i}$ of $B_i$. 
The castle $\{\{(B_{i,j}, n_i)\}_{j=0}^{r_i}\}_{i=1}^{k}$ has base equal to $B$ and levels of diameter less than $\eps$.

Let $A$ be a level of the castle $\sC_1$ that is not at the top of a tower, i.e., $A=T^{j}(B_i)$ for some $1\leq i\leq k$ and $0\leq j\leq n_i-2$. We claim that 
$T'(A)= h^{-1} \circ \lambda' \circ h (A)=h^{-1} \circ \lambda \circ h (A)=T(A)$.
Note that $h(A)$ is the set
\[
\{e_1e_2\dots \in B_{V,E} : r(e_1) \text{ corresponds to the $i$th tower of } \sC_1 \text{ and } \mathrm{label}(e_1) = j \}.
\]
By the construction above, the edges that join the vertex in $V_1$ corresponding to the $i$th tower of $\sC_1$
to the unique vertex in the level $V_0$ are ordered in the same way by $\leq$ and $\leq'$
(as induced from the natural order on the interval $\{0,1,\ldots,n_i-1\}$),
so that the images of $h(A)$ under $\lambda$ and $\lambda'$ are identical and equal to
\[
\{e_1e_2\cdots \in B_{V,E} : r(e_1) \text{ corresponds to the $i$th tower of } \sC_1 \text{ and } \mathrm{label}(e_1)=  j+1 \}.
\]
We conclude that $T'(A)=T(A)$. 

Assume now that $A=\towertop(\sC_1)$. We will show that $T'(A)=T(A)$.
Note that 
\[
h(A)=\{e_1e_2\cdots \in B_{V,E} : e_1 \text{ is a maximal edge (w.r.t.\ both orders)} \}.
\]
We let $\mathfrak{s}$ and $\mathfrak{p}$ denote the successor and predecessor operations,
with respect to $\leq$, on edges in the totally ordered sets $r^{-1}(v)$ for $v\in V$. 
The analogous operations associated with the order $\leq'$ we denote by $\mathfrak{s}'$ and $\mathfrak{p}'$.
We want to show that $\lambda(h(A))=\lambda'(h(A))$. Clearly both sets contain the unique (shared) minimal path $x_{\mathrm{min}}$ in $B_{V,E}$, since the unique (shared) maximal path $x_{\mathrm{max}}$ belongs to $h(A)$. Assume now that $e_1e_2\cdots \in h(A)$ is not the maximal path, and let $n\geq 2$ be the first integer such that $e_n\in E_n$ is not a maximal edge (w.r.t.\ both orders).
Then
$\lambda(e_1e_2\cdots)=y_1\cdots y_{n-1} \mathfrak{s}(e_n) e_{n+1}\cdots$
where $y_1\cdots y_{n-1}$ is the unique path consisting of minimal edges (w.r.t.\ both orders) such that the concatenation $y_{n-1} \mathfrak{s}(e_n)$ is well-defined.
We need to show that $y_1 \cdots y_{n-1} \mathfrak{s}(e_n) e_{n+1}\cdots$ belongs to
$\lambda'(h(A))$.
Note that since $\leq$ and $\leq'$ share minimal and maximal edges, $\mathfrak{s}(e_n)$
cannot be minimal for $\leq'$, and so $\mathfrak{p}'(\mathfrak{s}(e_n))$ is well-defined.
Let $q=z_1\cdots z_{n-1} \mathfrak{p}'(\mathfrak{s}(e_n))  e_{n+1} \cdots$
where $z_1 \cdots z_{n-1}$ is the unique path consisting of maximal edges (w.r.t.\ both orders) such that the concatenation $z_{n-1} \mathfrak{p}'(\mathfrak{s}(e_n))$ is well-defined.
Clearly, $q\in h(A)$, and we have $\lambda'(q)=y_1 \cdots y_{n-1} \mathfrak{s}(e_n) e_{n+1}\cdots$, completing the argument. Since $A=T^{-1}B$, we obtain $T'(T^{-1}B)=T(T^{-1}B)=B$ and finally conclude that $d(T'x,Tx)<\eps$ for all $x\in X$.

To see that $d(T'^{-1}x,T^{-1}x)<\eps$ for all $x\in X$, one similarly shows that for every level $A$ that is not at the base of $\sC_1$ 
one has $T'^{-1} A = T^{-1}A$, while $T'^{-1} B=T^{-1} B$. It follows by our choice of
$\eps$ that $T'$ belongs to $U$. This establishes the conclusion of Theorem~\ref{T-SOE}.

\begin{lemma}\label{L-upe dense}
Suppose $X$ is the Cantor set. Let $T\in \Astar (\Zb , X)$ and let $U$ be a neighbourhood of $T$ in $\Act (\Zb ,X)$.
Then $U$ contains an action in $\Astar (\Zb , X)_T$ with uniformly positive entropy.
\end{lemma}

\begin{proof}
Since $T$ belongs to $\Astar (\Zb , X)$ there is a unique measure $\mu$ in $M_T (X)$.
By Theorem~\ref{T-SOE} we can find a minimal homeomorphism $T' \in U$ with the same orbits as $T$
such that $T' \curvearrowright (X,\mu )$ 
is a Bernoulli shift. The fact that $T' \curvearrowright (X,\mu )$ is Bernoulli
implies that it is mixing. By minimality $\mu$ has full support, and so $T'\curvearrowright X$ is topologically mixing and hence spectrally aperiodic,
so that $T'$ belongs to $\Astar (\Zb , X)_T$.
Bernoullicity also implies by Theorem~A of \cite{GlaWei94} that $T'$ has uniformly positive entropy.
\end{proof}

\begin{lemma}\label{L-zero entropy dense}
Suppose $X$ is the Cantor set. Let $T\in \Astar (\Zb , X)$ and let $U$ be a neighbourhood of $T$ in $\Act (\Zb ,X)$.
Then $U$ contains an action in $\Astar (\Zb , X)_T$ with zero entropy.
\end{lemma}

\begin{proof}
We follow the ideas in the proof of \cite[Corollary~6.3]{Orm97}. Write $\mu$ for the unique measure in $M_T(X)$.
Let $S\curvearrowright (Z,\zeta)$ be a weakly mixing p.m.p.\ transformation on an atomless standard probability space with zero measure entropy. 
The existence of such an $S$ follows, for example, from the genericity of weak mixing \cite{Hal44}
and zero measure entropy \cite{Roh59} in $\Aut(Z,\zeta)$.
By Theorem~\ref{T-SOE}, there exists a minimal homeomorphism $T'\in U$ such that $T'\curvearrowright X$ is strongly orbit equivalent to $T\curvearrowright X$ via the identity map and $T'\curvearrowright(X,\mu)$ is measure conjugate to $S\curvearrowright(Z,\zeta)$.
Since $\mu$ has full support by minimality, we conclude that $T'\curvearrowright X$ is topologically weakly mixing, and in particular spectrally aperiodic.
Finally, since $T'$ is uniquely ergodic and $T'\curvearrowright (X,\mu)$ has zero measure entropy, it follows 
by the variational principle \cite[Theorem~9.48]{KerLi16} that $T'\curvearrowright X$ has zero topological entropy.
\end{proof}

Let $d\geq 2$.
Recall that $\Astar (F_d , X)$ denotes the set of all topologically free actions $\alpha\in\Act (F_d , X)$ with $M_\alpha (X) \neq\emptyset$
that are spectrally aperiodic and strictly ergodic on each standard generator (conditions which imply that $\alpha$ itself is strictly ergodic).

\begin{lemma}\label{L-perp}
Let $X$ be the Cantor set and $\gamma\in \Astar (F_d, X)$. Fix one of the standard generators $a$ of $F_d$. Then the set
\begin{align*}
\sD_\gamma :=\{\beta\in \Astar (F_d , X) : \text{\rm the homeomorphisms $\beta_a$ and $\gamma_a$ are disjoint}  \}
\end{align*}
is a $G_\delta$.
\end{lemma}

\begin{proof}
Since the map $\gamma\mapsto\gamma_a$ from $\Astar (F_d , X)$ to $\Act (\Zb , X)$ is clearly continuous 
and its image is contained in the set of minimal actions in $\Act (\Zb , X)$,
it suffices to show, given a minimal homeomorphism $T\in\Act (\Zb , X)$, that the set $T^\perp$ of all minimal homeomorphisms of $X$ which are disjoint from $T$ is a $G_\delta$. 
Let $U$ be a nonempty open subset of $X\times X$. Write $\sV_U$ for the set of all $S\in\Act (\Zb ,X)$
such that every orbit of $T\times S$ intersects $U$. Let us verify that $\sV_U$ is open. Let $S\in \sV_U$. Then
for every $z\in X\times X$ there is an $n_z\in\Zb$ such that $(T\times S)^{n_z} z \in U$ and hence
by continuity an open neighbourhood $V_z$ containing $z$ such that $(T\times S)^{n_z} V_z \subseteq U$.
By the compactness of $X\times X$ there is a finite set $E\subseteq X\times X$ such that the sets
$V_z$ for $z\in E$ cover $X\times X$. For all $S'$ in a sufficiently small neighbourhood of $S$ in $\Act (\Zb ,X)$ 
we have $(T\times S')^{n_z} V_z \subseteq U$ for every $z\in E$, in which case $S' \in \sV_U$. Thus $\sV_U$ is open.

By metrizability there is a sequence $U_1 , U_2 , \ldots$ of nonempty open subsets of $X\times X$ that 
form a basis for the topology. Then, using the fact that the disjointness of two minimal homeomorphisms is equivalent to the 
minimality of the diagonal action, one sees that the set $T^\perp$ is equal to the set of minimal actions in $\Act(\Zb,X)$ intersected with $\bigcap_{n\in\Nb}\sV_{U_n}$ and hence is a $G_\delta$.
\end{proof}

\begin{proof}[Proof of Theorem~\ref{T-meagre}]
Suppose, towards a contradiction, that there is a $\gamma\in\Astar (F_d , X)$ whose conjugacy class is
nonmeagre. 
Fix a standard generator $a\in F_d$.
Since the map $\rho\mapsto\rho_a$ from $\Astar (F_d  , X)$ to $\Act (\Zb , X)$
is continuous and the set of homeomorphisms in $\Act(\Zb,X)$ with zero entropy is a $G_\delta$ \cite[Lemma~2.4]{GlaWei01}, the set
\[Z=\{\rho\in\Astar (F_d , X): \rho_a \text{ has zero entropy}\}\]
is a $G_\delta$. 

To see that $Z$ is dense in $\Astar(F_d,X)$, we start by letting $\alpha\in\Astar (F_d , X)$. Then there is a unique $\mu\in M_\alpha (X)$.
Since $\alpha_a$ is strictly ergodic by the definition of $\Astar (F_d , X)$,
the measure $\mu$ is also the unique element of $M_{\alpha_a} (X)$ and thus is ergodic for $\alpha_a$.
Lemma~\ref{L-zero entropy dense} tells us that we can find a $T\in\Astar (\Zb ,X)_{\alpha_a}$ with zero entropy
that is as close as we wish to $\alpha_a$. 
Consider the action $\alpha'$ of $F_d$ obtained by sending $a\mapsto T$ 
and having $\alpha'$ equal $\alpha$ on the other standard generators of $F_d$.
Then $\alpha'$ is strictly ergodic and spectrally aperiodic on each of the standard generators of $F_d$.
Since $T$ has the same orbits as of $\alpha_a$, the measure $\mu$ is also $T$-invariant, whence $M_{\alpha'} (X)$
contains $\mu$ and in particular is nonempty.
With the aim of showing that $\alpha'$ belongs to $\Astar (F_d ,X)$, it remains to verify topological freeness.

Since $\alpha$ is topologically free by virtue of its membership in $\Astar (F_d  , X)$,  
there exists a dense set $X_0 \subseteq X$ such that $\alpha_t x \neq x$ for all $t\in F_d \setminus \{ e \}$ and $x\in X_0$.
Let $s\in F_d \setminus \{ e \}$. As an element of $\langle a \rangle * F_{d-1} = F_d$ we can write $s$ uniquely as a reduced
word $a^{n_1} g_1 a^{n_2} g_2 \cdots a^{n_k} g_k$ where possibly $n_1 = 0$ and/or $g_k = e$ (but not both when $k=1$)
but otherwise $n_i \neq 0$ and $g_i \neq e$ for all $i$. Since $T$ and the restriction of $\alpha$
to $\langle a \rangle$ have the same orbits, for every $x\in X$ we can write $\alpha'_s x$ as $\alpha_t x$
where $t = a^{m_1} g_1 a^{m_2} g_2 \cdots a^{m_k} g_k$ for some $m_1 , \dots , m_k \in\Zb$ which depend on $x$ and are nonzero (using the freeness of $\alpha_a$)
except in the case when $n_1 = 0$. In particular this $t$ will be nonzero, and so when $x\in X_0$
we get $\alpha'_s x = \alpha_t x \neq x$, from which we conclude that $\alpha'$ is topologically free 
and hence belongs to $\Astar (F_d ,X)$.

The action $\alpha'$ we can make as close as we wish 
to $\alpha$, depending on how close $T$ is to $\alpha_a$. This shows that $Z$ is dense
in $\Astar (F_d , X)$, and hence is a dense $G_\delta$. The conjugacy class of $\gamma$, being nonmeagre,
must then intersect $Z$, so that $\gamma_a$ has zero entropy by the invariance of entropy under conjugacy.

Applying Lemma~\ref{L-upe dense} in the same way that we did with Lemma~\ref{L-zero entropy dense} above,
we see that the set of all $\rho\in\Astar (F_d , X)$ such that $\rho_a$ 
has uniformly positive entropy is dense in $\Astar (F_d , X)$. 
By Proposition~6 of \cite{Bla93}, every minimal homeomorphism of $X$ with zero entropy is disjoint 
from every homeomorphism of $X$ with uniformly positive entropy, and so we deduce using Lemma~\ref{L-perp}
that $\sD_\gamma=\{\beta\in \Astar (F_d , X) : \text{$\beta_a$ and $\gamma_a$ are disjoint}\}$ is a dense $G_\delta$ subset of $\Astar (F_d , X)$. But then the 
conjugacy class of $\gamma$, being nonmeagre, must intersect $\sD_\gamma$, contradicting the fact
that a homeomorphism of any compact metrizable space with more than one point is not disjoint from itself.
\end{proof}

\section{Examples of squarely divisible actions of free groups}\label{S-examples}

The purpose of this section is to establish Propositions~\ref{P-examples} and \ref{P-SD examples}, which 
together provide examples of topologically free actions $F_2 \curvearrowright X$ on the Cantor set with $M_{F_2} (X)\neq\emptyset$
that are (i) strictly ergodic and weakly mixing on each generator and 
(ii) squarely divisible. These actions in particular belong to both $\WA(F_2,X)$ and $\Astar(F_2,X)$.
By Theorem~\ref{T-SR1}, their reduced crossed products 
have stable rank one. One can also similarly construct actions $F_d \curvearrowright X$
for $d\in\Nb$, $d>2$, with the same properties. 

Given a minimal homeomorphism $T\curvearrowright X$ of the Cantor set and any nonempty clopen set $Y\subseteq X$, 
by the first return time construction described in Section~\ref{S-meagre} 
one can generate a clopen castle $\sC=\{(B_i,n_i)\}_{i=1}^{k}$ that partitions $X$, with $\base(\sC)=Y$.
The $\Zb$-action generated by $T$ is free by minimality and the fact that $X$ is infinite, and so
we can arrange for the heights $n_i$ of the towers to be as large as we wish by choosing $Y$ 
so that its images under a sufficiently large number of iterations of $T^{-1}$ are pairwise disjoint.

To construct our actions $F_2 = \langle a,b \rangle \curvearrowright X$ on the Cantor set,
we will take a minimal homeomorphism $T$ of $X$ (corresponding to $a$)
and then combine this with a second homeomorphism $S$ (corresponding to $b$) that is constructed
by taking some clopen castle decomposition with respect to $T$ and doing some shuffling of the levels within each tower.
This shuffling will be implemented by a permutation of the tower levels that sends the top level to the bottom one,
which permits us to define $S$ on the top level of each tower not according to the permutation (as on the other levels)
but rather so as to be equal to $T$. As a consequence
$S$ will have the same orbits and the same asymptotic behaviour as $T$ (in fact it will be conjugate to $T$).
We thus formulate the following definition.

\begin{definition}\label{D-permutation}
Let $T\curvearrowright X$ be a minimal homeomorphism of the Cantor set. A \emph{tower permutation} of $T$ is a homeomorphism $S\curvearrowright X$ for which there exists a clopen castle $\{(B_i,n_i)\}_{i=1}^k$ for $T$ partitioning $X$
and for each $i=1,\dots k$ a cyclic permutation $\pi_i$ of $\{ 0,\dots ,n_i - 1\}$
with $\pi_i (n_i-1) = 0$
such that $S = T^{\pi_i (j) - j}$ on $T^j B_i$ for $j=0,\dots , n_i -2$ and $S = T$ on $T^{n_i-1} B_i$.
\end{definition}

If $S$ is a tower permutation of $T$, then $T$ is a tower permutation of $S$. Indeed, letting $i\in\{1,\ldots,k\}$
we observe that $\pi_i^{n_i-1}(0)=n_i-1$ and that $S^j$ is equal to $T^{\pi_i^j(0)}$ on $B_i$ for every $j\in \{0,\ldots, n_i-1\}$.
In particular, 
$X = \bigsqcup_{i=1}^k \bigsqcup_{j=0}^{n_i - 1} S^jB_i $
is a clopen castle decomposition for $S$. 
For every $j=0,\ldots,n_i-2$, let $\sigma_i(j)$ be the unique integer in $\{1,\ldots,n_i-1\}$ satisfying the equation $\pi_i^j(0)+1=\pi_i^{\sigma_i(j)}(0)$, and set $\sigma_i(n_i-1)=0$.
It is then straightforward to check that $\sigma_i$ is a cyclic permutation of  $\{0,\ldots,n_i-1\}$
such that $T=S^{\sigma_i(j)-j}$ on $S^jB_i$ and $T=S$ on $S^{n_i-1}B_i$.

This discussion in particular implies that if $S$ is a tower permutation of $T$, then $\orb_T(x)=\orb_S(x)$ for all $x\in X$ and $M_T(X)=M_S(X)$.

\begin{lemma}\label{L-perm conj}
Let $T\curvearrowright X$ be a minimal homeomorphism of the Cantor set, and let $S\curvearrowright X$ be a 
tower permutation of $T$. Then $S$ and $T$ are conjugate. 
\end{lemma}

\begin{proof}
Obtain a castle $\{(B_i,n_i)\}_{i=1}^k$ and cyclic permutations $\pi_i$ 
for $i=1,\dots ,k$ from the definition of tower permutation.
Given $1\leq i\leq k$, the permutation $\pi_i$, being cyclic, is conjugate to the cyclic shift 
$\rho(j) = j+1 \pmod{n_i}$ in $\Sym(n_i)$, i.e., there exists a permutation $\sigma_i \in \Sym(n_i)$ 
such that $\sigma_i \circ \pi_i = \rho \circ \sigma_i$. We can additionally ask that $\sigma_i(0)=0$, 
which also forces $\sigma_i(n_i-1)=n_i-1$
by the nature of $\pi$ and $\rho$, as one can easily check. Define $h : X \to X$ by
\begin{equation*}
hx = T^{\sigma_i(j) - j}x \label{E-perm conjugacy}
\end{equation*}
for all $x \in T^jB_i$, $i=1,\dots ,k$, and $j=0,\dots , n_i -1$.
Since the tower levels form a clopen partition of $X$, this is a homeomorphism. Finally, we verify that $h$ implements the desired conjugacy. Let $x \in X$. Then $x \in T^jB_i$ for some $1\leq i\leq k$ and $0\leq j\leq n_i - 1$.
Suppose that $j\neq n_i - 1$. 
Then $Sx = T^{\pi_i(j) - j}x \in T^{\pi_i(j)} B_i$ and hence
\[
hSx = T^{\sigma_i(\pi_i(j)) - \pi_i(j)} Sx = T^{\sigma_i(\pi_i(j)) - \pi_i(j)}T^{\pi_i(j) - j}x = T^{\sigma_i(\pi_i(j))-j}x
\]
so that, using the equality $\sigma_i \circ \pi_i = \rho\circ \sigma_i$,
\[
Thx = T^{\sigma_i(j) + 1 - j}x = T^{\sigma_i(\pi_i(j)) - j}x = hSx. 
\]
Suppose finally that $j = n_i - 1$. Then on $T^j B_i$ we have $h = \id$ and $S = T$, and since
$h$ is also the identity on $B_{i'}$ for all $i' = 1, \dots , k$ and $T^{n_i}B_i\subseteq \bigsqcup_{i'=1}^{k}B_{i'}$ we obtain
$Thx = Tx = hTx = hSx$ for every $x\in T^j B_i$.
Therefore $S$ and $T$ are conjugate.
\end{proof}

If $S$ is a tower permutation of a minimal homeomorphism 
$T$ of the Cantor set $X$, then
the resulting action $F_2 = \langle a,b\rangle \curvearrowright X$ defined by $a\mapsto T$ and $b\mapsto S$
satisfies $M_T(X) = M_S(X) = M_{F_2}(X)$.
However, this action will in general be far from faithful, let alone topologically free. In order to remedy this we will construct diagonal products from such actions, in the same spirit as Section~\ref{S-diagonal machine}, only now with infinitely many factors, which will permit us to arrange for topological freeness asymptotically via measure-theoretic considerations.

A family of p.m.p.\ actions $G\curvearrowright (Z_i , \zeta_i )$ with countable index set $I$ is
{\it disjoint} if the only probability measure on $\prod_{i\in I} Z_i$ that is invariant under the diagonal action of $G$ and projects factorwise onto $\zeta_i$ for every $i\in I$
is the product measure $\prod_{i\in I} \zeta_i$.
A family of p.m.p.\ actions $G\curvearrowright (Z_i , \zeta_i )$ with arbitrary index set $I$ is {\it disjoint}
if every finite subfamily is disjoint. It is clear that these two definitions agree in the overlapping case
when $I$ is a countably infinite set.

\begin{lemma}\label{L-strerg}
Suppose $G$ is amenable.
Let $G\curvearrowright  X_n$ for $n\in\Nb$ be strictly ergodic actions on compact metrizable spaces
such that, writing $\mu_n$ for the unique measure in $M_G (X_n)$, 
the family of p.m.p.\ actions $G \curvearrowright (X_n,\mu_n )$ for $n\in\Nb$ is disjoint.
Then the diagonal action $G\curvearrowright \prod_{n\in\Nb} X_n$ given by 
$s(x_n)_{n\in\Nb} = (sx_n )_{n\in\Nb}$ is strictly ergodic.
\end{lemma}

\begin{proof}
Set $\mu = \prod_{n\in\Nb} \mu_n$. 
By the definition of disjointness, the action $G\curvearrowright \prod_{n\in\Nb} X_n$ has a unique
invariant Borel probability measure, namely $\mu$. Now for every $m\in\Nb$ the diagonal action
$G\curvearrowright \prod_{n=1}^m X_n$ must be minimal, for any nonempty proper $G$-invariant
subset for this action would support a $G$-invariant Borel probability measure by the amenability of $G$,
and this measure would differ from the product measure $\prod_{n=1}^m \mu_n$ since the latter
has full support given that each $\mu_n$ has full support by the minimality of $G \curvearrowright X_n$,
leading us to a contradiction with the disjointness of the family $G \curvearrowright (X_n,\mu_n )$ for $n=1,\dots , m$.
It follows that if $A$ is a nonempty closed $G$-invariant subset of $\prod_{n\in\Nb} X_n$ then for every $m\in\Nb$
the projection of $A$ onto $\prod_{n=1}^m X_n$ is a nonempty closed $G$-invariant set and hence equal to $\prod_{n=1}^m X_n$,
from which we deduce that $A = \prod_{n\in\Nb} X_n$. Thus $G\curvearrowright \prod_{n\in\Nb} X_n$ is minimal.
\end{proof}

\begin{lemma}\label{L-sofic}
Let $E$ be a finite subset of $F_2 \setminus \{ e \}$ and let $\eps > 0$. Then there is an $N\in\Nb$ such that
for every finite set $X$ of cardinality at least $N$ there exists an action $F_2 \curvearrowright X$ 
in which each of the two standard generators $a$ and $b$ acts transitively and for which
\[
|\{ x\in X : sx = x \} | \leq \eps |X|
\]
for every $s\in E$.
\end{lemma}

\begin{proof}
Take a $q\in\Nb$ such that each element of $E$ can be written as a word of length at most $q$ in the generating set 
$D:= \{ a,a^{-1} , b , b^{-1} \}$ for $F_2$.

Let $J \geq 2$ be an integer satisfying $(4/J) |D^q | \leq \eps/2$.
It is well known that $F_2$ is residually finite, and so there exists a finite quotient $\pi : F_2 \to H$ 
such that none of the images $\pi (s)$ for $s\in E$ and $\pi (a^j)$ and $\pi (b^j)$ for $j=1,\dots , J$ are equal to $e$,
and in particular $|H|\geq J$. 
Let $N$ be any integer larger than $4|H|/\eps$ and let us show that this fulfils the requirements of the lemma.

Let $n$ be any integer with $n\geq N$. 
Then there is a largest integer $r\geq 1$ satisfying $(r+1)|H| \leq n$. Note that by the choice of $n$ we have
\begin{align*}
n-r|H|<\frac{\eps n}{2}.
\end{align*}
Set $H^+ := H\times \{ 1,\dots , r \}$ and enlarge this to some set $H^{++}$ of cardinality $n$.
Consider the diagonal action $\rho$ of $F_2$ on $H^+$ coming from 
the left translation action $(s,h)\mapsto \pi (s)h$ on the first factor and the trivial action on the second,
and extend this to an action of $F_2$ on $H^{++}$, which we also call $\rho$,
by having each of $a$ and $b$ act in some arbitrary transitive way
on $H^{++} \setminus H^+$.
The transitive pieces $Y_1 , \dots , Y_m$ into which $H^{++}$ partitions under the action of $\rho_a$ 
each have cardinality at least $J$ by our assumptions on $\pi$ from the first paragraph together with the fact that 
$|H^{++} \setminus H^+ | \geq |H|\geq J$. Therefore $m\leq n/J$.
For each $k=1,\dots , m$ choose a $y_{k,1}\in Y_k$. 
Then $Y_k$ has the form
\[
\{y_{k,1}, \rho_a(y_{k,1}), \rho_{a^2}(y_{k,1}),\ldots, \rho_{a^{j_k}}(y_{k,1})\}
\]
for some $j_k\geq J\geq 2$. Set $y_{k,2}=\rho_{a^{j_k}}(y_{k,1})$.
Similarly, the transitive pieces $Z_1 , \dots , Z_{m'}$ into which the action of $\rho_b$ partitions $H^{++}$
each have cardinality at least $J$, so that $m'\leq n/J$, and we choose for each $k=1,\dots , m'$ a $z_{k,1}\in Z_k$ and 
write $z_{k,2}\in Z_k$ for the image of $z_{k,1}$ under $\rho_{b^{l_k}}$, where $l_k\geq J\geq 2$ is the largest integer satisfying $\rho_{b^{l_k}}(z_{k,1})\neq z_{k,1}$.
Define a new action $\kappa$ of $F_2$ on $H^{++}$ by setting $\kappa_a (y_{k,2}) = y_{k+1,1}$ for $k=1,\dots , m-1$
and $\kappa_a (y_{m,2}) = y_{1,1}$ and setting $\kappa_a (y) = \rho_a y$ for other $y\in H^{++}$,
and likewise setting $\kappa_b (z_{k,2}) = z_{k+1,1}$ for $k=1,\dots , m'-1$
and $\kappa_b (z_{m',2}) = z_{1,1}$ and setting $\kappa_b (y) = \rho_b y$ for other $y\in H^{++}$. 
Then both $\kappa_a$ and $\kappa_b$ act transitively.

For each $s\in E$ write $s=s_1\cdots s_r$ with $s_1,\ldots, s_r\in D$ and $1\leq r\leq q$.
Suppose that $x\in H^{++}$ and that none of the elements $x, \kappa_{s_r}(x), \kappa_{s_{r-1}s_r}(x),\ldots, \kappa_{s}(x)$
is equal to $y_{k,1}$, $y_{k,2}$, $z_{k',1}$, or $z_{k',2}$ for $k=1,\ldots,m$ and $k'=1,\ldots,m'$. By the way $\kappa$ is defined
we then have $\kappa_s(x)=\rho_s(x)$. In particular, if $x\in H^+$ then $\kappa_s(x)\neq x$ since $\pi(s)\neq e$.

Since $E\subseteq D^q$ and $D^q$ is symmetric, we conclude that the set of all points $x\in H^{++}$ 
for which $\kappa_s x = x$ is contained in the set
\[
H^{++}\setminus H^+ \cup \bigcup_{t\in D^q} \kappa_t ( \{ y_{k,1}, y_{k,2} : k=1,\dots , m \} \cup \{ z_{k,1},z_{k,2} : k=1,\dots , m' \})
\]
and hence 
\begin{align*}
|\{x\in H^{++}: \kappa_s x=x\}\}|
&\leq |H^{++}\setminus H|+2(m+m')|D^q| \\
&\leq (n-r|H|)+\frac{4n}{J}|D^q|
\leq \eps n.
\end{align*} 
Finally, let $X$ be any set of cardinality $n$ and let $\theta: X\to H^{++}$ be any bijection.
Then the action of $F_2$ on $X$ defined by $s\mapsto \theta^{-1}\kappa_s\theta$ for $s\in F_2$ has the desired properties.
\end{proof}

\begin{lemma}\label{L-asymp free}
Let $T\curvearrowright X$ be a minimal homeomorphism of the Cantor set.  
Let $E$ be a finite subset of $F_2\setminus \{e\}$ and let $\eps > 0$. 
Then there exists a tower permutation $S$ of $T$ such that, for the induced action $F_2 = \langle a,b\rangle \curvearrowright X$ sending $a\mapsto T$ and $b\mapsto S$, one has
$\mu (\{x\in X  : sx = x \} )\leq \eps$ for every $s\in E$ and $\mu\in M_T (X)$. 
\end{lemma}

\begin{proof}
Take a $q\in\Nb$ such that each element of $E$ can be written as a word of length at most $q$ in the generating set 
$D:= \{ a,a^{-1} , b , b^{-1} \}$ for $F_2$.

As recalled at the beginning of the section, by minimality $T$ admits a 
clopen castle $\{(B_i,n_i)\}_{i=1}^k$ partitioning $X$ such that each $n_i\geq 2$ is an integer greater than $4|D^q|/\eps$ 
that is also sufficiently large for a purpose to be described imminently.

Let $1\leq i \leq k$. Then by Lemma~\ref{L-sofic} we can assume $n_i$ to be large enough so that
there exists an action $F_2 \stackrel{\kappa}{\curvearrowright} \{ 1, \dots , n_i -2 \}$ such that $\kappa_a$ is the cyclic permutation sending $j$ to $j+1$ for $j=1,\dots , n_i -3$ and sending $n_i-2$ to $1$, the action of $\kappa_b$ is transitive, and 
\begin{gather}\label{E-fixed points}
|\{ x\in \{1,\ldots,n_i-2\} : \kappa_sx = x \} | \leq \frac{\eps}{2} (n_i-2 )
\end{gather}
for every $s\in E$.
Define $S$ on the $i$th tower by setting 
\begin{enumerate}
\item $S = T^{\kappa_b (1)}$ on $B_i$,

\item $S = T^{n_i-2}$ on $TB_i$,

\item $S = T^{\kappa_b (l)-l}$ on $T^l B_i$ for $2\leq l \leq n_i -2$,

\item $S=T$ on $T^{n_i-1} B_i$.
\end{enumerate}
Having done this over all $i=1,\dots , k$, we obtain a tower permutation $S$ of $T$ via the cyclic permutations 
$\pi_i=(0,\kappa_b(1),\ldots,\kappa_b^{n_i-3}(1),1,n_i-1)\in\Sym(n_i)$.
Define an action $F_2 = \langle a,b\rangle \stackrel{\alpha}{\curvearrowright} X$ 
by setting $\alpha_a = T$ and $\alpha_b = S$. 

Now let $s\in E$ and set $X_s = \{x\in X  : \alpha_s x = x \}$. Let $\mu\in M_T (X)$.
Given (\ref{E-fixed points}) and the way that we constructed $S$ on each tower, 
we see that the union of the levels in the $i$th tower 
that do not get sent under $\alpha_s$ to a different level in the same tower has $\mu$-measure at most 
$(\eps /2)(n_i -2)\mu(B_i) + 2|D^q|\mu(B_i)$, with the second summand accounting for the possibility that, under $\alpha$, 
the successive application of the generators in $D$ that spell out a reduced word in $E$
sends us at some point into the bottom or top level of the tower. 
More explicitly, write $C_i = \bigsqcup_{l=0}^{n_i-1} T^l B_i$ and let $x\in C_i$.
If $\alpha_s x\notin C_i$, then $x$ must belong to the union
\[
\bigcup_{t\in D^q} \alpha_t ( B_i \cup T^{n_i-1}B_i),
\]
which has $\mu$-measure bounded by $2|D^q|\mu(B_i)$ given that $M_T(X)=M_S(X)$. 
If $x$ is not in that union then
$x\in \bigsqcup_{l=1}^{n_i-2}T^l B_i$ and one can show by a straightforward
induction on the word length of $s$ that $\alpha_s x\in T^{\kappa_s(l)}B_i$ whenever $x\in T^l B_i$. 
Therefore
\[
X_s\cap C_i\subseteq \bigg( \bigcup_{t\in D^q} \alpha_t ( B_i \cup T^{n_i-1}B_i) \bigg) \cup 
\bigg( \bigsqcup_{l=1}^{n_i-2} T^l B_i\cap X_s \bigg)
\]
and so, using (\ref{E-fixed points}) and the fact that $n_i>4|D^q|/\eps$, 
\begin{align*}
\mu (X_s \cap C_i) 
\leq 2|D^q|\mu(B_i)+\frac{\eps}{2}(n_i-2)\mu(B_i)
\leq \eps n_i \mu (B_i ) 
= \eps \mu (C_i ).
\end{align*}
This yields finally
\begin{align*}
\mu (X_s ) = \sum_{i=1}^k \mu (X_s \cap C_i) \leq \eps \sum_{i=1}^k \mu (C_i ) = \eps ,
\end{align*}
giving the desired conclusion.
\end{proof}

We are now in a position to establish the first main result of this section, which produces actions in both $\WA(F_2,X)$ and $\Astar(F_2,X)$. We note that sequences $(T_n)_{n\in \Nb}$ as in the proposition statement exist in abundance, as explained in 
Remark~\ref{R-examples}. 

\begin{proposition}\label{P-examples}
Let $(T_n)_{n\in \Nb}$ be a sequence of strictly ergodic homeomorphisms of 
the Cantor set $X$ such that, writing $\mu_n$ for the unique measure in $M_{T_n} (X)$, 
the p.m.p.\ transformations $T_n \curvearrowright (X,\mu_n )$ for $n\in\Nb$ are weakly mixing and form a disjoint family.
Then there exists a sequence $(S_n)_{n\in \Nb}$ of homeomorphisms of $X$ 
such that, defining for each $n$ the action $F_2 = \langle a,b\rangle \stackrel{\alpha_n}{\curvearrowright} X$ by $a\mapsto T_n$ and $b\mapsto S_n$, one has the following:
\begin{enumerate}
\item $M_{\alpha_n} (X) = M_{T_n} (X) = M_{S_n} (X)$ for every $n\in\Nb$,

\item $\lim_{n\to\infty} \mu_n (\{ x\in X : \alpha_{n,s} x = x \} ) = 0$ for every $s\in F_2\setminus \{ e \}$,

\item the diagonal action $F_2 \curvearrowright X^\Nb$ is topologically free, topologically weakly mixing, and strictly ergodic,

\item the homeomorphisms $T=\prod_{n\in \Nb} T_n$ and $S=\prod_{n\in \Nb} S_n$ of $X^\Nb$
are strictly ergodic and topologically weakly mixing.
\end{enumerate} 
In particular, any conjugate 
$F_2 \curvearrowright X$ of the diagonal action $F_2 \curvearrowright X^\Nb$ by a
homeomorphism $X^\Nb \to X$ belongs to both $\WA(F_2,X)$ and $\Astar(F_2,X)$. 
\end{proposition}

\begin{proof}

Take an increasing sequence $K_1 \subseteq K_2 \subseteq\dots$ of finite subsets of $F_2\setminus \{ e\}$
with $\bigcup_{n=1}^\infty K_n = F_2\setminus \{ e\}$. 
For each $n$ apply Lemma~\ref{L-asymp free} to obtain a tower permutation $S_n$ of $T_n$ such that the action $\alpha_n$ of $F_2 = \langle a,b \rangle$ given by $a\mapsto T_n$ and $b\mapsto S_n$ satisfies $\mu_n(X_{\alpha_n , s}) \leq 1/n$
for all $s\in K_n$ where $X_{\alpha_n , s}$ denotes the fixed-point set $\{ x\in X : \alpha_{n,s} x =x \}$. It follows that for every $s\in F_2\setminus \{ e \}$ we have 
$\lim_{n\to\infty} \mu_n (X_{\alpha_n ,s} ) = 0$.
As noted earlier, it follows readily from the definition of tower permutation that 
$M_{S_n} (X) = M_{T_n} (X) = M_{\alpha_n} (X) = \{ \mu_n \}$ for all $n$.

It is now easy to see that the actions $F_2\stackrel{\alpha_n}{\curvearrowright} X$ are strictly ergodic and the family of p.m.p.\ actions $F_2\stackrel{\alpha_n}{\curvearrowright} (X,\mu_n)$ for $n\in\Nb$ is disjoint. 
By Lemma~\ref{L-strerg}, the diagonal action $\prod_{n\in\Nb} \alpha_n$ of $F_2$ on $X^\Nb$, which
is determined by $a\mapsto T= \prod_{n\in \Nb} T_n$ and $b\mapsto S= \prod_{n\in \Nb} S_n$, is strictly ergodic, and its unique invariant measure is $\mu=\prod_{n\in\Nb} \mu_n$.

Again applying Lemma~\ref{L-strerg}, the action $T \curvearrowright X^\Nb$ is strictly ergodic, and
as a p.m.p.\ transformation of $(X^\Nb,\mu )$ it is weakly mixing since each factor is weakly mixing.
Since by minimality each $\mu_n$ has full support, so does $\mu$. A simple exercise 
then shows that $T$ is topologically weakly mixing, and therefore so is the diagonal action
$F_2\curvearrowright X^\Nb$.

By Lemma~\ref{L-perm conj} the homeomorphism $S$ is conjugate to $T$ and hence is strictly ergodic and
topologically weakly mixing.

It remains to show that the action $F_2\curvearrowright X^\Nb$ is topologically free.
For every $s\in F_2\setminus \{ e \}$ we see,
given the equality $\{ x\in X^\Nb : sx=x \} =  \prod_{n\in \Nb} X_{\alpha_n ,s}$, that
\[
\mu(\{ x\in X^\Nb : sx=x \}) =\lim_{n\to\infty} \prod_{k=1}^{n} \mu_k(X_{\alpha_k ,s}) \leq \lim_{n\to\infty} \mu_n(X_{\alpha_n,s})= 0 ,
\]
which is tantamount to saying that the p.m.p.\ action $F_2\curvearrowright (X^\Nb , \mu )$ is free. 
From this we deduce that the action $F_2\curvearrowright X^\Nb$, being minimal, must be topologically free. 
\end{proof}

\begin{remark}\label{R-examples}
Examples of sequences $(T_n )_{n\in\Nb}$ as in Proposition~\ref{P-examples} can be produced as follows.
Del Junco showed that there is an uncountable disjoint family $\{ T_i \}_{i\in I}$ 
of weakly mixing p.m.p.\ transformations \cite[Corollary 2(b)]{Jun81},
and by the Jewett--Krieger theorem we may view each $T_i$ as a minimal homeomorphism of the Cantor set
with the invariant Borel probability measure being unique. Then any sequence drawn from this family will fulfill the requirements.
\end{remark}

Our remaining task is to establish square divisibility for the actions appearing in Proposition~\ref{P-examples}. As observed in Remark~\ref{R-0dimOSD}, this will also imply weak square divisibility. We reprise, in a simplified form, some of the arguments at play in Sections~\ref{S-SD I} and \ref{S-SD II}. 
The first of the lemmas we will need 
is naturally phrased in terms of topological full groups. 

\begin{definition}
Let $G\curvearrowright X$ be a minimal action on the Cantor set. The \emph{topological full group}, 
denoted $[[G\curvearrowright X]]$, is the group of all homeomorphisms $h:X\to X$ such that there exists a finite clopen partition $\{A_1,\ldots,A_n\}$ of $X$ and $s_1,\ldots,s_n \in G$ such that $hx = s_ix$ for all $i=1,\ldots,n$ and $x\in A_i$. 
In the case of a single homeomorphism $T\curvearrowright X$ we write $[[T]]$ for the topological full group of the 
action of $\Zb$ that it generates. Note that when the action $G\curvearrowright X$ is faithful 
one has a canonical embedding of $G$ into $[[G\curvearrowright X]]$.
\end{definition}

\begin{lemma}\label{L-SD full group}
Let $G\curvearrowright X$ be a faithful minimal action on the Cantor set and let $O_1,O_2\subseteq X$ be nonempty clopen subsets. Suppose that for every finite set $e\in E\subseteq G$ the action is $(O_1,O_2,E)$-squarely divisible. 
Then for every subgroup $H\subseteq [[G\curvearrowright X]]$ satisfying $G\subseteq H$ 
(identifying $G$ with its image under the canonical embedding into $[[G\curvearrowright X]]$ 
that we get from faithfulness) 
and every finite set $e\in F\subseteq H$, the action $H\curvearrowright X$ associated with the inclusion of $H$ in the homeomorphism group of $X$ 
is $(O_1,O_2,F)$-squarely divisible.
\end{lemma}

\begin{proof}
Let $e\in F\subseteq H$ be a finite set. 
By assumption, for each $h\in F$ there exists a clopen partition $\{A_{h,1},\ldots, A_{h,n_h}\}$ of $X$ and elements $s_{h,i} \in G$ such that $hx = s_{h,i}x$ for all $i=1,\ldots,n_h$ and $x\in A_{h,i}$. Write the join of these partitions 
as $\{B_1,\dots,B_m\}$. This is a clopen partition with the property that for every $h\in F$ and $j=1,\ldots,m$ there exists an $s_{h,j} \in G$ such that $hx = s_{h,j}x$ for all $x\in B_j$, as $B_j\subseteq A_{h,i}$ for some $i$. 
Write $E$ for the finite subset $\{s_{h,j} : h \in F, \ 1\le j \le m\}$ of $G$. By hypothesis, the action $G\curvearrowright X$ is $(O_1,O_2,E)$-squarely divisible. Hence by Proposition~\ref{P-SD0Dim} there exist an $n\in \Nb$ and pairwise equivalent and $E$-disjoint clopen subsets $\{V_{p,q}\}_{p,q=1}^n$ of $X$ such that, writing $V = \bigsqcup_{p,q=1}^n V_{p,q}$, $V_1 = \bigsqcup_{p=1}^n V_{p,1}$, $R = V^c$, and $B = V\cap (V^E)^c$, the following hold:
\begin{enumerate}
\item $V_{p,1} \prec O_1 \cap \bigsqcup_{q=2}^n V_{p,q} \cap B^c$ for every $p=1,\ldots,n$,
\item $R\prec O_2 \cap V \cap (V_1\cup B)^c$,
\item $B\prec O_2 \cap R$. 
\end{enumerate}
Now from the definition of $E$ we see that $FY \subseteq EY$ and $Y^F \supseteq Y^E$ for any set $Y\subseteq X$. The first of these inclusions
implies that the sets $\{V_{p,q}\}_{p,q=1}^n$ are $F$-disjoint for the action of $H$. The second
implies that the set $\tilde{B} := V\cap (V^F)^c$ is contained in $V\cap (V^E)^c$, and so from (i)--(iii) we immediately obtain
the following subequivalences with respect to the action of $G$ and hence also with respect to the action
of $H$ (since $G\subseteq H \subseteq [[G\curvearrowright X]]$ by hypothesis):
\begin{enumerate}
\item $V_{p,1} \prec O_1 \cap \bigsqcup_{q=2}^n V_{p,q} \cap \tilde{B}^c$ for every $p=1,\ldots,n$,
\item $R\prec O_2 \cap V \cap (V\cup \tilde{B})^c$,
\item $\tilde{B}\prec O_2 \cap R$. 
\end{enumerate}
Notice furthermore that the sets $\{V_{p,q}\}_{p,q=1}^n$ are pairwise equivalent for the action of $H$ since they
are pairwise equivalent for the action of $G$. We conclude 
by Proposition~\ref{P-SD0Dim} that $H\curvearrowright X$ is $(O_1,O_2,F)$-squarely divisible. 
\end{proof}

\begin{lemma}\label{L-SD permutation}
Let $T \curvearrowright X$ be a minimal homeomorphism of the Cantor set and $S$ a tower permutation of $T$. Let $F_2 =\langle a,b\rangle\curvearrowright X$ be the action given via $a\mapsto T$ and $b\mapsto S$. 
Let $O_1, O_2 \subseteq X$ be nonempty clopen sets and let $e\in E \subseteq F_2$ be a finite set. Then the action $F_2\curvearrowright X$ is $(O_1, O_2, E)$-squarely divisible. 
\end{lemma}

\begin{proof}
Since $T\curvearrowright X$ is almost finite (as is manifest by the clopen castle decompositions discussed at the beginning of 
this section and the previous one),
by Theorem~\ref{T-amenable} the action $\langle a \rangle \curvearrowright X$ is $(O_1,O_2,F)$-squarely divisible for every finite set 
$e\in F\subseteq \langle a \rangle$. Also, the minimality of $T$ and
the infiniteness of $X$ together imply that the action $\langle a \rangle\curvearrowright X$ is free (and hence faithful). 
Since it is clear from the definition of tower permutation that $S\in [[T]]$, we deduce by 
Lemma~\ref{L-SD full group} that the action $F_2\curvearrowright X$ given by $a\mapsto T$ and $b\mapsto S$ is $(O_1, O_2, E)$-squarely divisible.
\end{proof}

\begin{remark}\label{R-JM}
In Lemma~\ref{L-SD permutation} the action $F_2\curvearrowright X$ is not faithful. In fact 
a theorem of Juschenko and Monod \cite{JusMon13} implies that the action $F_2\curvearrowright X$ must
factor through an action of an amenable group.
\end{remark}

As Lemma~\ref{L-SD permutation} illustrates, 
Lemma~\ref{L-SD full group} yields many examples of squarely divisible actions of (frequently nonamenable) groups on the Cantor set. However, as Remark~\ref{R-JM} highlights, these actions are often far from being topologically free or even faithful. 
We now recapitulate a method that we have already used on a couple of occasions 
for showing square divisibility in the context diagonal actions (cf.\ Lemma~\ref{L-SD free product}).

\begin{lemma}\label{L-SD diagonal}
Let $G\stackrel{\alpha_n}{\curvearrowright} X_n$ for $n\in\Nb$ be actions on the Cantor set such that the 
corresponding diagonal action $G\stackrel{\alpha}{\curvearrowright} X:= \prod_{n\in\Nb} X_n$
is minimal. Suppose that each $\alpha_n$ is $(O_1,O_2,E)$-squarely divisible for all nonempty clopen sets $O_1,O_2\subseteq X_n$ and finite sets $e\in E\subseteq G$. Suppose that for every finite subset $F\subseteq G$ and $M\in\Nb$ there exist an integer $m\geq M$ and a nonempty clopen set $A\subseteq X_m$ such that $(F,A)$ is a tower for $\alpha_m$.
Then the action $G\curvearrowright X$ is squarely divisible. 
\end{lemma}

\begin{proof}
Let $O$ be a nonempty clopen subset of $X$. By Proposition~\ref{P-SD0Dim} it is enough to establish $O$-square divisibility,
and for this we may assume, by passing to a smaller clopen set if necessary, 
that $O$ satisfies $O = \pi_I^{-1} (\pi_I (O))$ for some finite set $I\subseteq\Nb$
where $\pi_I : X\to X_I := \prod_{n\in I} X_n$ is the coordinate projection map.
Since $\alpha$ is minimal so is $\gamma := \prod_{n\in I} \alpha_n$,
and so we can apply a compactness argument as in the proof of Lemma~\ref{L-SD free product}
to find a finite set $F\subseteq G\setminus \{ e \}$ and a clopen partition $\{W_s\}_{s\in F}$ of $X_I$ indexed by $F$ 
such that $\gamma_s W_s\subseteq \pi_I (O)$ for each $s\in F$.

Let $E \subseteq G$ be a finite set containing $e$. By hypothesis there are an $m > \max I$ and a clopen set
$A\subseteq X_m$ such that $(F\cup \{e\},A)$ is a tower for $\alpha_m$.
Take disjoint nonempty clopen sets $O_1,O_2\subseteq A$. 
Since $G\stackrel{\alpha_m}{\curvearrowright} X_m$ is $(O_1,O_2,E)$-squarely divisible, by Proposition~\ref{P-SD0Dim}
there exist an $n\in \Nb$ and pairwise equivalent and pairwise $E$-disjoint clopen subsets $\{V_{i,j}\}_{i,j=1}^n$ of $X_m$ such that, writing
$V = \bigsqcup_{i,j=1}^n V_{i,j}$, $V_1 = \bigsqcup_{i=1}^n V_{i,1}$, $R = X_m\setminus V$, and $B = V\cap (V^E)^c$,
\begin{enumerate}
\item $V_{i,1} \prec_{\alpha_m} O_1 \cap \bigsqcup_{j=2}^n V_{i,j}\cap B^c$ for $i=1,\ldots,n$,
\item $R\prec_{\alpha_m} O_2 \cap V \cap (V_1 \cup B)^c$,
\item $B\prec_{\alpha_m} O_2 \cap R$.
\end{enumerate}
For each $C\subseteq X_m$ set 
$C' = \prod_{n=1}^{m-1} X_n \times C\times\prod_{n=m+1}^\infty X_n \subseteq \prod_{n\in\Nb} X_n$.
Then from the subequivalences above we get 
\begin{enumerate}
\item $V_{i,1}' \prec_\alpha O_1' \cap \bigsqcup_{j=2}^n V_{i,j}'\cap (B')^c$ for $i=1,\ldots,n$,
\item $R'\prec_\alpha O_2' \cap V'\cap (V_1' \cup B')^c$,
\item $B'\prec_\alpha O_2' \cap R'$.
\end{enumerate}
The properties of $E$-disjointness and pairwise equivalence of the clopen sets $V_{i,j}'$ are inherited from the $V_{i,j}$.
Hence the action $G\stackrel{\alpha}{\curvearrowright} X$ is $(O_1',O_2',E)$-squarely divisible. 
Write $\pi$ for the coordinate projection map $\prod_{n\in\Nb\setminus \{  m\}} X_n \to X_I$.
Since $\gamma_s W_s\subseteq \pi_I (O)$ for $s\in F$ and the sets $\alpha_{m,s} (O_1\sqcup O_2)$ for $s\in F$ are pairwise disjoint, 
we obtain, interpreting the product sets below via the identification of $X$ with 
$(\prod_{n\in\Nb\setminus \{  m\}} X_n )\times X_m$,
\begin{align*}
O_1'\sqcup O_2' = \bigsqcup_{s\in F} \pi^{-1} (W_s) \times (O_1\sqcup O_2) 
&\sim \bigsqcup_{s\in F} s(\pi^{-1} (W_s)\times (O_1\sqcup O_2)) \\
&\subseteq \pi^{-1} (\pi_I (O)) \times (X_m\setminus (O_1\sqcup O_2)) .
\end{align*}
This shows that $O_1'\sqcup O_2' \prec \pi^{-1} (\pi_I (O)) \times (X_m\setminus (O_1\sqcup O_2))$, and since the two
sets in this subequivalence are disjoint and the second one is contained in $O$, 
we conclude by Proposition~\ref{P-SD0Dim} that $G\curvearrowright X$ is $O$-squarely divisible.
\end{proof}

\begin{proposition}\label{P-SD examples}
Let $F_2\curvearrowright X$ be an action on the Cantor set as in the last sentence of Proposition~\ref{P-examples}. Then $F_2\curvearrowright X$ is squarely divisible.  
\end{proposition}

\begin{proof}
By assumption there are $T_n$, $S_n$, and $\mu_n$ as in Proposition~\ref{P-examples} so that,
up to conjugacy, we can decompose the action $F_2\curvearrowright X$ as a diagonal product $F_2 = \langle a,b \rangle \curvearrowright X^\Nb$ of actions $F_2 \stackrel{\alpha_n}{\curvearrowright} X$ 
generated by $a\mapsto T_n$ and $b\mapsto S_n$ with 
\begin{gather}\label{E-fixed}
\lim_{n\to\infty} \mu_n ( \{ x\in X : \alpha_{n,s} x=x \} ) = 0
\end{gather}
for every $s\in F_2\setminus \{ e\}$.
By Lemma~\ref{L-SD permutation} each $\alpha_n$ is $(O_1,O_2,E)$-squarely divisible for all nonempty clopen sets $O_1,O_2\subseteq X$ and finite sets $e\in E\subseteq F_2$.

Now given a finite set $ F\subseteq F_2$ and an $M\in\Nb$, by (\ref{E-fixed}) we can find an integer $m\geq M$ 
such that $\mu_m ( \{ x\in X : \alpha_{m,s} x=x \} ) \leq 1/(2|F^{-1} F|)$ for all $s\in F^{-1} F\setminus \{ e \}$, 
which implies that the set of all $x\in X$ such that $\alpha_{m,s} x\neq x$ for all $s\in F^{-1} F \setminus \{ e \}$
has $\mu_m$-measure at least $1/2$. Choosing a particular $x$ in this set we then have $\alpha_{m,s} x\neq \alpha_{m,t} x$ 
for all distinct $s,t\in F$, which permits us to find clopen neighbourhood $A\subseteq X$ of $x$ 
such that $(F,A)$ is a tower for $\alpha_m$.
We have thus verified the conditions that enable us to apply Lemma~\ref{L-SD diagonal} 
and hence conclude that $F_2\curvearrowright X$ is squarely divisible.  
\end{proof}

\end{document}